\documentclass[11pt, makeidx]{amsart}
\usepackage{amssymb,amsfonts}

\usepackage{xcolor}

\input{styleart.sty}

\begin{document}

\begin{abstract}
We prove that a random group, in Gromov's density model with $d<1/16$, satisfies a universal sentence $\sigma$ (in the language of groups) if and only if $\sigma$ is true in a nonabelian free group.
\end{abstract}

\title{First-order sentences in random groups I: universal sentences}
\author{O. Kharlampovich, R. Sklinos}
\thanks{\\
\noindent O. Kharlampovich (CUNY)  supported by the grant No. 422503 from the Simons Foundation.\\
R. Sklinos (CAS) is supported by the grant No. 12350610234 from NSFC}

\maketitle
\section{Introduction}
The idea of random groups draws its origins in Gromov's seminal paper \cite{GroHyp} introducing hyperbolic groups. Gromov, in order to emphasize the importance of the newly defined class of groups, claimed, in a definite statistical sense, that hyperbolicity is a typical property for finitely presented groups. In other words, almost every finitely presented group is hyperbolic. In the same paper he proved his claim for {\em the few-relator model} in a strong sense, by proving that for this model small cancellation $C'(\lambda)$, for any cancellation parameter $0<\lambda<1$, is a typical property.  On the other hand, exhibiting the difference between small cancellation and hyperbolicity he proposed a slightly different model, called {\em the few-relator with various lengths model} and claimed, without a proof, that only hyperbolicity is typical for this model. Ol'shanski \cite{Olsh} and independently Champetier \cite{Cha95} (for two relators) confirmed the claim. In the meantime, Gromov proposed a different model, which seemed to be more flexible and interesting, called {\em the density model} and proved that hyperbolicity is still a typical property for this model \cite[Chapter 9]{Gromov}. A complete proof was later given by Ollivier \cite{Olli2}. The paper \cite{Olli} is an excellent introduction to the subject. In Section \ref{RandomModel} we give definitions for the various models discussed above. In this paper we will be exclusively concerned with the density model.   

On a different line of thought, Tarski in 1946 asked whether  the first-order theories of non-abelian free groups of different ranks coinside. This question was answered in the positive only more than fifty years later, in 2006, by Sela \cite{SelFree} and Kharlampovich-Myasnikov \cite{KMFree}. We will follow the tradition in the subject and call this first-order theory, {\em the theory of the free group}. 

These two lines of thought intersect in the following conjecture of J. Knight loosely stated here (for the precise conjecture see section \ref{RandomModel}). 

\begin{conjIntro}
Let $\sigma$ be a first-order sentence in the language of groups. Then $\sigma$ is true in a nonabelian free group if and only if it is almost surely true in ``a random group".
\end{conjIntro}

In this paper we confirm the conjecture in the case where $\sigma$ is universal and the density of the randomness model is $d<1/16$.  Our main theorem is:

\begin{thmIntro}
Let $\Gamma$ be ``a random group" of some fixed density $d<1/16$. Let $\sigma$ be a universal sentence in the language of groups. Then $\Gamma$ almost surely satisfies $\sigma$ if and only if a nonabelian free group $\mathbb{F}$ satisfies $\sigma$. 
\end{thmIntro}

Our theorem says that a property described by a universal first-order formula is typical for finitely presented groups (in the density model for $d<1/16$) if and only if it holds in a non-abelian free group. 

In the statements of the conjecture and the theorem we abuse notation by referring to ``a random group" as formally it is not a group, but rather an asymptotic behaviour of a collection of groups. Still it is convenient to use this terminology as it simplifies statements. Again, precise statements will be given in Section \ref{RandomModel}. 

It is interesting that a random group is not a limit group (see Proposition \ref{NotLimit}), equivalently it does not satisfy the universal part of the first-order theory of nonabelian free groups. On the other hand, our main result shows that, for $d<1/16$, each universal sentence which is true in the free group  with overwhelming probability (w.o.p.) is true in a random group.

One direction of our main result is easy to check. A random group of fixed density $d<1/2$ does not satisfy a universal sentence which is not in the axioms of the theory of the free group. Indeed, a random group contains a nonabelian free group \cite{OW}, hence it satisfies the existential theory of the free group. 

Finally, we give, as a result of independent interest not interfering with the main proof, a meaningful axiomatization of the universal part of the first-order theory of the free group.

The paper is structured as follows. In the first five sections we develop all background material needed and prove some preliminary results. In section \ref{Logic} we present some basic definitions of first-order logic and give an axiomatization of the universal part of the first-order theory of nonabelian free groups. In section \ref{RandomModel} we record the various models of randomness and state results relevant towards our goal. We also observe, using well-known results, that a random group (in the density model) is not a limit group. In section \ref{VKampen} we give a brief account of van Kampen diagrams that are one of our main tools in proving the main result. In section \ref{universal} we use first-order logic to translate the original question into group theoretic terms.  

In section \ref{Short} we use geometric tools, in particular a version of the shortening argument, in order to prove a local to global result. This section is the most technically challenging part of the paper. Finally in the last section we  calculate the relevant probabilities and bring everything together to show the main theorem of the paper.

\section{First-order Logic \& Axioms for the Universal Theory of the Free Group}\label{Logic}
In order to make the paper self-contained we include some background in first-order logic. We focus on groups although everything can be developed more generally for other natural (or not) mathematical structures. 

The first-order language of groups $\mathcal{L}$ consists of $\{\cdot, ^{-1}, 1\}$, that will be interpreted as group multiplication, inverse, and neutral element. Together with the logic symbols $\land, \lor, \neg, \Rightarrow$, equality $=$, parentheses $(,)$, quantifiers $\exists, \forall$ and an infinite supply of variables are the building blocks of first-order formulas. One can recursively define what is a well-formed formula on these symbols, but we will just give some examples instead. 

\begin{example}[First-order formulas]
Let $G$ be a group seen as an $\mathcal{L}$-structure (i.e. by interpreting the symbols of the language in the natural way). Then 
\begin{itemize}
    \item the formula $\phi(x):= \exists y (x= y\cdot y)$ defines the ``squares" in $G$;
    \item the formula $\phi(x):=\forall y (x\cdot y\cdot x^{-1}\cdot y^{-1}= 1)$ defines the central elements in $G$.
\end{itemize}
When all variables are bounded by a quantifier we talk about a first-order sentence. Hence 
\begin{itemize}
    \item the sentence $\sigma:=\forall x\forall y (x\cdot y\cdot x^{-1}\cdot y^{-1}= 1)$, if true in $G$ it says that it is abelian;
    \item the sentence $\sigma:=\exists y \forall x (y\not= 1 \land [x,y]= 1)$, where $[x,y]$ stands for the commutator of $x,y$, if true in $G$ it says that it has a non-trivial central element.
\end{itemize}
\end{example}

More generally a formula $\phi(\bar{x})$, i.e. a first-order formula whose free variables are amongst $\bar{x}$, in the language of groups is logically equivalent to a formula of the form $\ \ \forall\bar{z}\exists\bar{y}\ldots$ $$\bigwedge_{i=1,\ldots,n}(w_{i1}(\bar{x}, \bar{y}, \bar{z}, \ldots)= 1\lor \ldots \lor 
w_{ik_i}(\bar{x}, \bar{y}, \bar{z}, \ldots)= 1\lor v_{i1}(\bar{x}, \bar{y}, \bar{z}, \ldots)\not= 1\lor\ldots \lor v_{im_i}(\bar{x}, \bar{y}, \bar{z}, \ldots)\not= 1),$$
an alternation of $\forall$ and $\exists$ quantifiers followed by a finite conjunctions of disjunctions of equalities and inequalities of words in the given bounded and free variables.

\begin{definition}
A group $G$ is fully residually free if for any finite set of elements $S\subset G$, there exists a morphism, $h_S:G\rightarrow \mathbb{F}$, to some free group $\mathbb{F}$ that is injective on $S$. A finitely generated fully residually free group is also called a limit group.
\end{definition}

\begin{definition}
A group $G$ is called locally fully residually free if every finitely generated subgroup of $G$ is fully residually free.
\end{definition}

An example of a group which is not fully residually free but locally fully residually free is $(\mathbb{Q},+)$. The additive group of rationals is divisible but every finitely generated subgroup of it is cyclic.

We denote by $T_{fg}^{\forall}$ the set of universal sentences satisfied by a non-abelian free group.

\begin{proposition}\label{LimitGroups}\cite{Rem}
A finitely generated group $G$ is a model of $T_{fg}^{\forall}$ if and only if it is fully residually free.
\end{proposition}

\begin{corollary}
A group $G$ is a model of $T_{fg}^{\forall}$ if and only if $G$ is locally fully residually free.
\end{corollary}
\begin{proof}
Suppose $G$ is locally fully residually free and $\sigma:=\forall\bar{x}Q(\bar{x})$, where $Q(\bar{x})$ is quantifier free, be a universal sentence in $T_{fg}^{\forall}$. Assume for the sake of contradiction that $G\models \lnot \sigma$. Then there exists $\bar{a}$ in $G$ such that $G\models \lnot Q(\bar{a})$. We set  $A:=\langle\bar{a}\rangle$, be the subgroup of $G$ generated by $\bar{a}$, and observe that since $A$ is a finitely generated subgroup of $G$ it is  fully residually free, hence it is a model of $T_{fg}^{\forall}$ and consequently it realizes $\sigma$. On the other hand, since $A$ is a subgroup of $G$ we get $A\models \lnot Q(\bar{a})$ and $A\models \lnot \sigma$, a contradiction.

Suppose that $G\models T_{fg}^{\forall}$. Since any subgroup of $G$ is a model of $T_{fg}^{\forall}$ as well,  if  $L$ is a finitely generated subgroup of $G$, we get, by Proposition \ref{LimitGroups}, that $L$ is fully residually free. 

\end{proof}

The following remark is a direct consequence of the definition of fully residually freeness.

\begin{remark}\label{MRStep1}
Let $G_{\Sigma}:=\langle \bar{x} \ | \ \Sigma(\bar{x})\rangle$ be a finitely generated group, which is not fully residually free. Then, there exist finitely many nontrivial elements of $G_{\Sigma}$  $w_1(\bar{x}),\ldots, w_m(\bar{x})$ such that for any morphism $h:G_{\Sigma}\rightarrow \mathbb{F}$, the morphism $h$ kills one of the elements, i.e.  $h(w_i(\bar{x}))=1$ for some $i\leq m$. In the following proposition we will use the family of all such possible witnesses of the failure of fully residually freeness. 
\end{remark}

\begin{proposition}\label{UniversalAxioms}
The following list of axioms defines  the class of locally fully residually free groups.

\begin{enumerate}
    \item Axioms for groups.
    \item For every finite system of equations $\Sigma(\bar{x})=1$ such that $G_{\Sigma}$ is not a fully residually free group, the list of sentences $\forall\bar{x}(\Sigma(\bar{x})=1\rightarrow w_1(\bar{x})=1\lor\ldots\lor w_m(\bar{x})=1)$, where one sentence is added for each set of words that witness the failure of fully residually freeness given by Remark \ref{MRStep1}.
\end{enumerate}
\end{proposition}
\begin{proof}
If $L$ is locally fully residually free, then it satisfies the universal theory of the free group and consequently it satisfies the given sentences (which are part of the universal theory). 

If $L$ is not a locally fully residually free group, then either it is not a group or it has a finitely generated subgroup $G_{\Sigma}$ 
which is not fully residually free. Let $w_1(\bar{x}), \ldots, w_m(\bar{x})$ be a finite list of non trivial words that witness it. In principle $\Sigma$ might be infinite, but by the equational Noetherianity of $\mathbb{F}$ (see \cite{Guba}) there exists $\Sigma_0\subset\Sigma$ which is finite and equivalent to $\Sigma$ over $\mathbb{F}$. 

It is easy to see that $G_{\Sigma_0}$ is not fully residually free (witnessed by the same nontrivial $w_i$'s). Indeed, we first notice that $w_1,\ldots, w_m$ are not trivial in $G_{\Sigma_0}$. If $w_i$, for some $i\leq m$, was trivial in $G_{\Sigma_0}$, then it would be trivial in $G_\Sigma$, since $\Sigma_0\subseteq \Sigma$. Secondly, by the equivalence of $\Sigma_0$ with $\Sigma$, the values a morphism, $h:G_{\Sigma_0}\rightarrow \mathbb{F}$, gives to the tuple $\bar x$ can be used to define a morphism $h':G_{\Sigma}\rightarrow \mathbb{F}$ and such a morphism kills one of the $w_i$'s, hence $h$ kills one of the $w_i$'s. Consequently, the sentence $\forall\bar{x}(\Sigma_0(\bar{x})=1\rightarrow  w_1(\bar{x})=1\lor w_2(\bar{x})=1\lor\ldots\lor w_m(\bar{x})=1)$ is part of the axioms. But then $G_{\Sigma}$ does not satisfy it, since the identity morphism cannot kill any nontrivial element. Therefore, the group $L$ does not satisfy the sentence.

\end{proof}

\begin{lemma}
The class of locally fully residually free groups is not finitely axiomatizable.
\end{lemma}
\begin{proof}
We will show that the class of non locally fully residually free groups is not closed under ultraproducts. Indeed, consider the finite cyclic group $\mathbb{Z}_{p_i}$ where $p_i$ is a prime and $Z^*:=\Pi_{i\in\mathbb{N}}\mathbb{Z}_{p_i}/\mathcal{U}$ be an ultraproduct with respect to a non-principal ultrafilter $\mathcal{U}$. Then, since $Z^*$ is abelian and torsion free, it is locally fully residually free.    
\end{proof}

\section{Models of randomness and background results}\label{RandomModel}
In this paper we will use Gromov's density model of randomness, but for completeness we also describe the few-relator models. In this section after stating the appropriate definitions, we present some results concerning random groups of various models. In addition, we record that random groups are not limit groups.  

\subsection{The models}

\begin{definition}[Few-relator model]\label{FR}
We fix $n\geq 2$, $k\geq 1.$ Let $\mathcal{G}_{k,\ell}$ be the following set of group presentations:
$$\{\langle e_1, \ldots, e_n \ | \ r_1, \ldots, r_k\rangle \ : \ r_i \ \textrm{reduced} \ \& \ |r_i|\leq \ell \}.$$

We say that property $P$, which is a property of presentations, happens almost surely in the few-relator model if the ratio of presentations in $\mathcal{G}_{k,\ell}$ which have property $P$ over $|\mathcal{G}_{k,\ell}|$ goes to $1$ as $\ell$ goes to infinity.
\end{definition}

As mentioned in the introduction is not hard to check that in this model any small cancellation property $C'(\lambda)$, for $0<\lambda<1$, is almost surely true. Recall that a group has $C'(\lambda)$ if no two defining relators (up to taking inverses) share a subword whose length is greater than $\lambda$ times the minimum length of the relators. We remark here that J. Knight's conjecture was posed for the few-relator model and particularly for $k=1$, but we find more natural and interesting the case of Gromov's  density model which we explain in the sequel.  A few relator model with $k=1$ corresponds to the  density model with $d=0$, hence J. Knight's conjecture falls into the scope of our work as a special case. We properly restate it here. 
 
 \begin{conjecture}[J. Knight]
 Let $n\geq 2$. Let $\mathcal{G}_{1,\ell}$ be as in Definition \ref{FR}. Then a first-order sentence in the language of groups is almost surely true in this randomness model if and only if it is true in a non-abelian free group.
 \end{conjecture}
 
 


 
Finally we present the density model. 

\begin{definition}[Gromov's density model]
We fix $n\geq 2$. Let $\mathbb{F}_n:=\langle e_1, \ldots, e_n\rangle$ be a free group of rank $n$. Let $S_{\ell}$ be the set of reduced words on $e_1, \ldots, e_n$ of length $\ell$. 

Let $0\leq d\leq 1$. Then a random set of relators of density $d$ at length $\ell$ is a subset of $S_{\ell}$ that consists of  $(2n-1)^{d\ell}$-many elements picked randomly (uniformly and independently) among all elements of $S_{\ell}$. 

A group $G:=\langle e_1,\ldots, e_n \ | \ \mathcal{R} \ \rangle$ is called random of density $d$ at length $\ell$ if $R$ is a random set of relators of density $d$ at length $\ell$. 
\end{definition}

A random group of density $d$ almost surely satisfies some property of presentations $P$, if the probability of occurrence of $P$ tends to $1$ as $\ell$ goes to infinity. Sometimes abusing language we may say a random group satisfies a property instead of almost surely satisfies the property. 

From this point on we will exclusively use the density model. Hence any statements about random groups concern this model. We list some well known properties of random groups. 

\begin{proposition}[Gromov \cite{Gromov}, Ollivier \cite{Olli}]
Let $a>0$ and $d<a/2$. Then a random group of density $d$ satisfies the small cancellation property $C'(a)$. 
%
%
\end{proposition}

In particular, a random group of density $1/16$ satisfies the small cancellation property $C'(1/8)$. This latter class of groups has some properties that will prove very useful towards our goal. We note in passing  that every group that satisfies the small cancellation property $C'(1/6)$ is hyperbolic. On the other hand, we record for completeness the following (harder to prove) theorem about the hyperbolicity of random groups.    

\begin{theorem}[Gromov \cite{Gromov}, Ollivier \cite{Olli2}]
Let $G$ be a random group of density $d$. 
\begin{itemize}
    \item If $d<1/2$. Then $G$ is infinite torsion-free hyperbolic.
    \item If $d>1/2$. Then $G$ is either trivial or $\mathbb{Z}_2$.
\end{itemize}
\end{theorem}

 In the next subsection we record an easy corollary of already known results about random groups.

\subsection{A random group is not a limit group}\label{Sectionlimit}

In this subsection we prove that a random group is not almost surely a limit group. We actually give two proofs that will be easy consequences of the following results.
 
\begin{fact}[Kozma-Lubotzky \cite{LK}]\label{FieldReprGromov}
Let $0<d<1$ and $G$ be a random group at density $d$. Then, for any fixed $k$, there is no nontrivial degree $k$ representation of $G$ over any field.
\end{fact}

We say that a group has {\em property (FA)} if whenever it acts on a simplicial tree (in the sense of Bass-Serre), then it fixes a point.

\begin{fact}[Dahmani-Guirardel-Przytycki \cite{DGP11}]\label{FA}
Let $0<d<1$. A random group at density $d$ has almost surely property (FA).
\end{fact}

\begin{proposition}\label{NotLimit}

If $0<d\leq 1/2$, then  a random group is not a limit group.
\end{proposition}
\begin{proof} By Sanov's result \cite{sanov},    $\mathbb{F}_2$ embeds into $SL_2({\mathbb Z})$ and consequently into $SL_2({\mathbb R})$. Since a limit group is universally free, it embeds into an ultrapower of $\mathbb F_2$, which in turn embeds into an ultrapower of $SL_2({\mathbb R})$. The latter is isomorphic to $SL_2({\mathbb R}^*)$, where $\mathbb R^*$ is an ultraproduct of $\mathbb R$. Hence, a limit group admits a nontrivial representation over $\mathbb R^*$. By Fact \ref{FieldReprGromov} a random group of any density cannot be a limit group.  
\end{proof}
\noindent{\em Second proof.} Assume, for a contradiction, that a random group (at any density $d$) is a limit group. By \cite[Theorem 3.2]{Sel1} any limit group admits a non-trivial action on a simplicial tree, contradicting Fact \ref{FA}.\qed
\ \\ \\
For the few-relator model this was proved in \cite{Ho}.

\section{Van Kampen diagrams}\label{VKampen}
Van Kampen diagrams, introduced by E. van Kampen in 1933 \cite{VanKampen}, although neglected for more than thirty years,  were rediscovered by R. C Lyndon and since then have been used as the principal tool for studying the word problem in a finitely presented group. A word $w$ is equal to the identity in the group $G=\langle X\ |\ R\rangle$ if and only if
it is a product of conjugates of relators $w=\Pi_{i=1}^Mg_ir^{\epsilon_i}g_i^{-1}$, where $r\in R$ and $\epsilon_i$ is either $1$ or $-1$.

Van Kampen diagrams can be considered as finite, planar, connected and simply connected $2$-complexes bearing the following extra information: 
\begin{itemize}
\item there exists a base vertex ($0$-cell) that lies in the boundary of the unbounded region of the complement of the $1$-skeleton of the complex in $\mathbb{R}^2$; 
\item each edge ($1$-cell) is oriented and labeled by a generator;
\item the boundary of each face ($2$-cell) corresponds to a word in $R$ (up to a cyclic permutation and/or inversion). It is obtained by reading the labels of the edges starting from a vertex and traversing the boundary in one of the two directions. When an edge is traversed in the direction of its orientation then the label is given the exponent $1$, otherwise it is given the exponent $-1$.  
\end{itemize}

The boundary word of a van Kampen diagram is the word on the boundary of the unbounded region of the complement of the $1$-skeleton in $\mathbb{R}^2$, which is read starting from the base vertex. 

An important  van Kampen lemma states that the boundary word of
a van Kampen diagram is equal to the identity in the presentation and a freely reduced word that is equal to the identity in the presentation is a boundary label of some van Kampen diagram. 

We will call a van Kampen diagram {\em non-filamentous} if every $0$- and $1$-cell is on the boundary of some $2$-cell (see figure \ref{AbPouch}). 

\begin{figure}[ht!]
\centering
\includegraphics[width=0.6\textwidth]{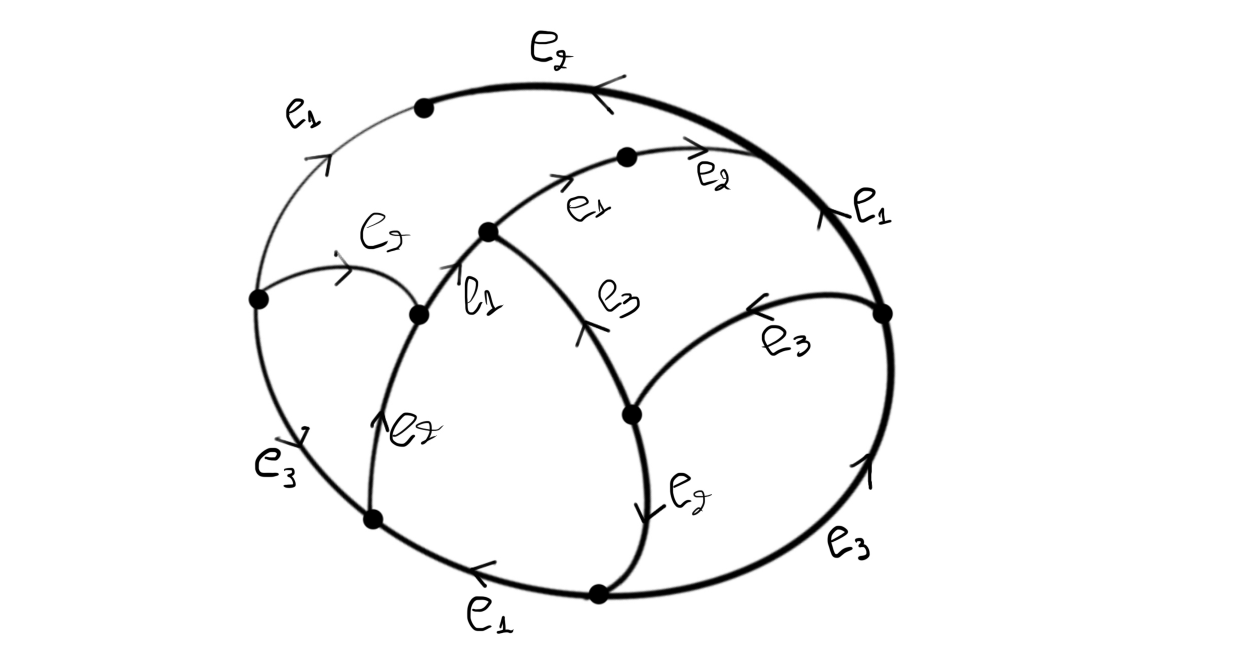}
\caption{Non-filamentous van Kampen Diagrams}
\label{AbPouch}
\end{figure}

One can obtain every van Kampen diagram by connecting non-filamentous components by paths that we call bridges (see figure \ref{AbPouch2}).

We will abbreviate a van Kampen diagram by dVK and a family of van Kampen diagrams by DVK. By an abstract van Kampen diagram, adVK, we mean a dVK but with no relators or generators attached on the $2$-cells and $1$-cells respectively. Still an adVK contains a starting point, orientation and a number on each $2$-cell (the number is for identifying $2$-cells that will bear the same relator if we wish to produce a dVK with the abstract data). Similarly we denote families of adVKs by aDVK. In \cite{Olli} an aDVK is called a decorated abstract van Kampen diagram.

\begin{figure}[ht!]
\centering
\includegraphics[width=0.8\textwidth]{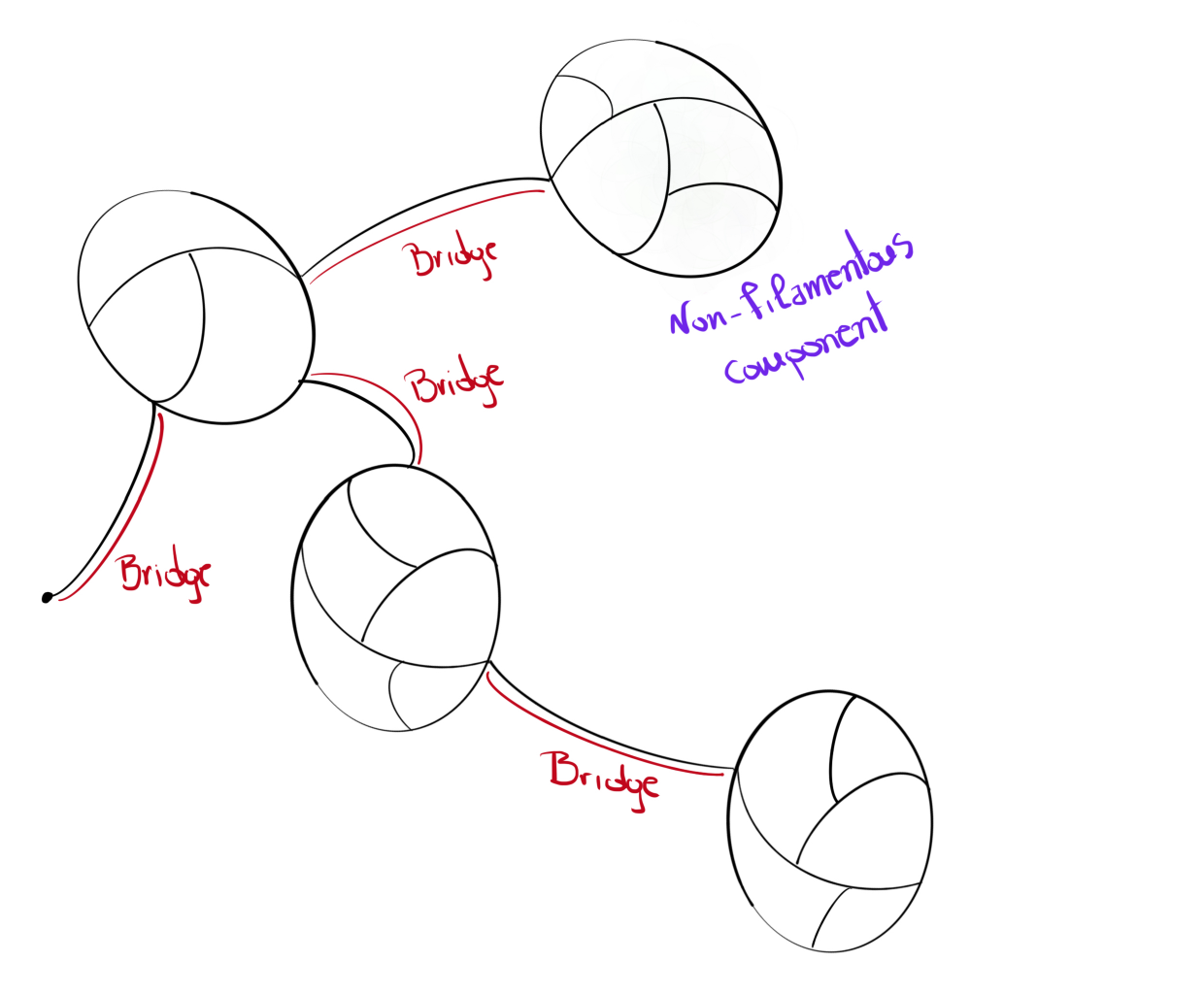}
\caption{A filamentous topological van Kampen Diagram}
\label{AbPouch2}
\end{figure}


\section{Reducing universal sentences to  equations in random groups}\label{universal}

We first note that every universal sentence in the language of groups can be put in the following form: $\forall\bar{x}\bigwedge(w_1(\bar{x})=1\lor w_2(\bar{x})=1\lor w_k(\bar{x})=1\lor v_1(\bar{x})\neq 1\lor v_2(\bar{x})\neq 1 \lor\ldots\lor v_m(\bar{x})\neq 1)$. Since the universal quantifier distributes over conjunctions we get that each such sentence is logically equivalent to finitely many  conjunctions of sentences of the form $\forall\bar{x}(w_1(\bar{x})=1\lor w_2(\bar{x})=1\lor w_k(\bar{x})=1\lor v_1(\bar{x})\neq 1\lor v_2(\bar{x})\neq 1 \lor\ldots\lor v_m(\bar{x})\neq 1)$. Hence, in order to prove that each universal sentence of the theory of the free group is almost surely true in a random group (for some fixed density), it is enough to prove that each sentence of the form $\forall\bar{x}(w_1(\bar{x})=1\lor w_2(\bar{x})=1\lor w_k(\bar{x})=1\lor v_1(\bar{x})\neq 1\lor v_2(\bar{x})\neq 1 \lor\ldots\lor v_m(\bar{x})\neq 1)$ holds in such random group. 

We fix $n\geq 2$, the free group $\mathbb{F}_n=\langle e_1,\ldots ,e_n\rangle$ of rank $n$, and a random group $\Gamma_\ell:=\langle e_1, \ldots, e_n \ | \ \mathcal{R}\rangle$ of density $d$ at length $\ell$ with the canonical epimorphism $\pi:\mathbb{F}_n\rightarrow \Gamma_\ell$.  
We consider a universal sentence $\sigma$ in $T_{fg}^{\forall}$ of the form $$\forall\bar{x}(w_1(\bar{x})=1\lor w_2(\bar{x})=1\lor w_k(\bar{x})=1\lor v_1(\bar{x})\neq 1\lor v_2(\bar{x})\neq 1 \lor\ldots\lor v_m(\bar{x})\neq 1).$$ 
If we denote by $V(\bar x)=1$ the system of equations $v_1(\bar{x})=1,\ldots ,v_m(\bar{x})=1,$ then this is equivalent to 
$$\forall\bar{x}(V(\bar{x})=1\rightarrow w_1(\bar{x})=1\lor\ldots\lor w_k (\bar{x})=1).$$
In order to prove that a random group (of density $d$) has the property described by $\sigma$, we need to prove that the probability that $\sigma$ is true in $\Gamma_\ell$ goes to $1$ as $\ell$ goes to infinity.
Equivalently, we need to prove that the probability that $\Gamma_\ell$ satisfies $\lnot\sigma$ goes to $0$ as $\ell$ goes to infinity. Now,  $\Gamma_\ell$ satisfies $\lnot\sigma$, while $\mathbb{F}_n$ satisfies $\sigma$, only if there exists a tuple $\bar{b}$ of elements  in $\Gamma_\ell$ such that $v_1(\bar{b})=1\land\ldots\land v_m(\bar{b})=1$, while for any pre-image $\bar{c}$ of $\bar{b}$ via $\pi$  we have $\vee _{i=1}^m v_i(\bar{c})\neq 1$ in $\mathbb{F}_n$.  Indeed, if $\Gamma_\ell$ satisfies $\lnot\sigma$, then there exists $\bar{b}$ in $\Gamma_\ell$ such that $w_1(\bar{b})\neq 1\land\ldots\land w_k(\bar{b})\neq1\land v_1(\bar{b})=1\land\ldots\land v_m(\bar{b})=1$. Now, since $\mathbb{F}_n$ satisfies $\sigma$ we must have, for any pre-image, say $\bar{c}$, of $\bar{b}$, that $w_1(\bar{c})=1\lor w_2(\bar{c})=1\lor w_k(\bar{c})=1\lor v_1(\bar{c})\neq 1\lor v_2(\bar{c})\neq 1 \lor\ldots\lor v_m(\bar{c})\neq 1$. But if some $w_i(\bar{c})$ is trivial in $\mathbb{F}_n$, then $w_i(\bar{b})$ is necessarily trivial in $\Gamma_\ell$. Thus, we must have that for some $i\leq m$, $v_i(\bar{c})\neq 1$ in $\mathbb{F}_n$.    

Our main goal is to show that for a fixed density $d<1/16$ and any fixed system of equations $V(\bar{x})=1$, the probability of the existence of some tuple $\bar{b}$ of elements in $\Gamma_\ell$ such that $V(\bar{b})=1$
 while for any pre-image $\bar{c}\in\pi^{-1}(\bar{b})$, $\bar c$ is not a solution of $V(\bar{x})= 1$ in $\mathbb{F}_n$  goes to $0$ as $\ell$ goes to infinity. We recall that the important fact about choosing $d<1/16$ is that a random group of such density almost surely satisfies the small cancellation property $C'(1/8)$. 


 
\section{ Bounding the length of short solutions}\label{Short}

In this section we will prove that if $V(\bar{x})=1$ is a (finite) system of equations and $(\Gamma_i)_{i<\omega}$ is a sequence of random groups of density $d<1/16$ and such that for each $i<\omega$, $\Gamma_i$ is at length $i$, then there exists a linear bound on the length of a ``shortest" (non-free) solution of $V(\bar{x})=1$ in $\Gamma_i$ with respect to $i$ (see Section \ref{6.6} for the notion of shortest and non-free).  This in turn is used In Section \ref{loc-gl} to bound the number of faces of a family of van Kampen diagrams that correspond to such solution. In particular, this allows us to claim that if a ``bad" solution exists, then there exists one such solution with boundedly many faces. Hence, it can be thought of as a local to global result. This together with the bound on faces help us to calculate the probability of the existence of ``bad" solutions and show that it goes to $0$ as the varying parameter $\ell$ goes to infinity.

We will use real trees and the shortening argument. There is a standard set of ideas on how to obtain a limit action on a real tree, given a sequence of actions of a finitely generated group $G$ on a sequence of hyperbolic spaces for which the corresponding hyperbolic constants tend to $0$. In many cases, under different assumptions, one can prove that the limit action can be analyzed using, what is called, the Rips machine  and in addition the limit actions can be ``shortened" by a family of automorphisms called modular automorphisms and are encoded in the JSJ decomposition of the ``limit group" when such decomposition exists. All these techniques are known under the general term ``the shortening argument". For results of this kind we refer the reader to \cite{ReWe}.

We will show that in the case of random groups with $d<1/16$ we can use some sort of a shortening argument. 
In this case some extra difficulty arises because our spaces (before rescaling) do not have bounded hyperbolicity constants. Hence standard arguments to prove that, in the limit action, tripod stabilizers are trivial or stabilizers of arcs are abelian fall through. There is a remedy for that once one observes that in small cancellation groups, for $C'(1/8)$, the geodesics between points in their Cayley graphs are much more ``rigid" than in a general hyperbolic group. We will exploit these properties  to prove that indeed in our setting one can use Rips machine and the shortening argument in order to obtain the desired result.

 \subsection{Actions on real trees}\label{ActionsRtrees}
 In this section we explain how a sequence of morphisms, $(h_i)_{i<\omega}:G\rightarrow \Gamma_i$, from a finitely generated group $G$ to small cancellation groups $\Gamma_i$, satisfying certain properties, induces an action of $G$ on a real tree that can be moreover decomposed using Rips machine. We will closely follow \cite{ReWe}.

We fix a finitely generated group $G$ and we consider the set of non-trivial equivariant pseudometrics $d:G\times G\to \R^{\geq 0}$, denoted by $\mathcal{ED}(G)$. 
We equip $\mathcal{ED}(G)$ with the compact-open topology (where $G$ is given the discrete topology). Note 
that convergence in this topology is given by: 
$$(d_i)_{i<\omega}\to d\ \ \textrm{if and only if} \ \ d_i(1,g)\to d(1,g)\ \ (\textrm{in $\R$}) \ \ \textrm{for all}\ \ g\in G.$$


We also note that any based $G$-space $(X,*)$ (i.e. a metric space with a distinguished point equipped with an action of 
$G$ by isometries) gives rise to an equivariant pseudometric on $G$ as follows: $d(g,h)=d_X(g\cdot *,h\cdot *)$. The pseudo-metric is $\delta$-hyperbolic if $d_X$ is $\delta$-hyperbolic.  
We say that a sequence of $G$-spaces $(X_i,*_i)_{i<\omega}$ converges to a $G$-space $(X,*)$, if the 
corresponding pseudometrics induced by $(X_i,*_i)$ converge to the pseudometric induced by $(X,*)$ in $\mathcal{ED}(G)$. 

\begin{lemma}\cite[Lemma 2.6]{ReWe}\label{ConvSpaces}
Let $G$ be a finitely generated group. Let $(X_i, *_i)$, for $i<\omega$, be a sequence of based $\delta_i$-hyperbolic $G$-spaces. Suppose the sequence of the corresponding induced pseudometrics $(d_i)_{i<\omega}$ converge in $\mathcal{ED}(G)$ to a pseudometric $d_{\infty}$ and $\delta_i$ goes to $0$ as $i$ goes to $\infty$. 

Then, there exists a based real $G$-tree $(T,*)$ such that $T$ is spanned by the orbit of the base point $*$ under the action of $G$ and $d_{\infty}$ is the corresponding pseudometric.
\end{lemma}

The next result is an easy corollary of Lemma \ref{ConvSpaces} (the proof follows the proof of  \cite[Corollary 2.8]{ReWe}). We remark that this is in principle the point of divergence from the standard arguments that usually assume a bounded hyperbolicity constant for the sequence of the based $G$-spaces $(X_i,*_i)$. 

\begin{remark}
Although it is not needed for all arguments that follow, we will assume for convenience that all hyperbolic spaces are geodesic. 
\end{remark}

We define {\em the length of the action} of a based $G$-space $(X,*)$ to be the sum of the translations of $*$ by a fixed generating set $S_G$, i.e. $\Sigma_{s\in S_G}d_X(*,s\cdot *)=\Sigma_{s\in S_G}d(1,s)$. 

\begin{corollary}\label{LimitActionDominating}
Let $G$ be a finitely generated group with $S_G$ a finite generating set, and $(X_i, *_i)$, for $i<\omega$, a sequence of based $\delta_i$-hyperbolic $G$-spaces. Suppose the lengths $\ell_i:=\Sigma_{s\in S_G}d_i(1,s)$, for $i<\omega$, of the corresponding pseudometrics dominate in the limit the hyperbolicity constants, i.e. $\frac{\delta_i}{\ell_i}\rightarrow 0$. 

Then there exists a based real $G$-tree $(T,*)$ such that the sequence of the rescaled pseudometrics $d_i/\ell_i$ has a subsequence converging in $\mathcal{ED}$ to the pseudometric that corresponds to $(T,*)$.

\end{corollary}

\begin{proof}
For each $i<\omega$ and $s\in S_G$ we have, by the definition of $\ell_i$,  $\frac{d_i(1, s)}{l_i}\leq 1$. Hence, the sequence $(\frac{d_i}{l_i})_{i<\omega}$ has a convergent subsequence. In addition, the pseudometric $\frac{d_i}{l_i}$ is $\frac{\delta_i}{\ell_i}$-hyperbolic and it goes to $0$ as $i$ goes to $\infty$. Hence, by Lemma \ref{ConvSpaces}, we can conclude.
\end{proof}

If one chooses the base points carefully the limit action can be guaranteed to be minimal.  

\begin{theorem}{\cite[Theorem 2.11]{ReWe}}\label{CenLoc}
Assume the hypothesis of Lemma \ref{ConvSpaces}. If the base points in $(X_i,*_i)$, for $i<\omega$, are chosen so that $\ell_i\leq\ell_i'$, for any other base point $*_i'$, then the obtained limit action is minimal. 
\end{theorem}

This can be proved using the notion of approximating sequences, which is a useful concept for the sequel. 

\begin{definition}
Let $(X_i,*_i)$ be a sequence of based geodesic  $G$-spaces. Let the sequence of the induced pseudometrics $(d_i)_{i<\omega}$ converge to a pseudometric $d$ induced by the based $G$-space $(X,*)$. We say that $x$ in $X$ is approximated by the sequence $(x_i)_{i<\omega}$ with $x_i$ in  $(X_i)$ for each $i<\omega$ if: 
$$lim_{i\rightarrow\infty}(d_{X_i}(x_i, g\cdot *_i)=d_X(x,g\cdot *) \ \ \ \textrm{for each $g$ in $G$}.$$ 
\end{definition}

It is clear that the base point $*$ of the limit space $(X,*)$ can be approximated by the base points $(*_i)$ of the limiting sequence $(X_i,*_i)_{i<\omega}$. As a matter of fact under the hypotheses of Lemma \ref{ConvSpaces} every point in the limit real tree admits an approximating sequence.

\begin{lemma}{\cite[Lemma 2.10]{ReWe}}
Assume the hypothesis of Lemma \ref{ConvSpaces}. Then for every $x,y$ in the limit real $G$-tree $(T,*)$ the following hold:
\begin{itemize}
    \item[(1)] $x$ has an approximating sequence.
    \item[(2)] if $(x_i)_{i<\omega}, (x_i')_{i<\omega}$ are approximating sequences for $x$, then $lim_{i<\omega}d_{X_i}(x_i,x_i')=0$.
    \item[(3)] if $(x_i)_{i<\omega}$ is an approximating sequence for $x$, then $(g\cdot x_i)_{i<\omega}$ is an approximating sequence for $g\cdot x$.
    \item[(4)] if $(x_i)_{i<\omega}, (y_i)_{i<\omega}$ are approximating sequences for $x, y$ respectively, then $lim_id_{X_i}(x_i,y_i)=d_X(x,y)$.
\end{itemize}

\end{lemma}

We will obtain actions on hyperbolic spaces through group morphisms as follows: 
a morphism $h:G\to H$ where $H$ is a hyperbolic group induces an action of $G$ on  
$X_H$ (the geometric realization of the Cayley graph of $H$) in the obvious way, thus making $X_H$, a hyperbolic space, a $G$-space. 

For a finitely generated group $G$, endowed with a fixed generating set $S_G$, and $g\in G$ we denote by $\abs{g}_{S_G}$ the word length of $g$ with respect to $S_G$, i.e. the length of the shortest word in $S_G$ representing $g$. When there is no ambiguity we will omit the subscript $S_G$ and denote the length of an element simply by $\abs{g}$.

\begin{definition}\label{LengthMorph}
Let $G:=\langle g_1, g_2, \ldots, g_m\rangle$ be a finitely generated group. Let $h:G\rightarrow \Gamma$ be a morphism to a finitely generated group $\Gamma$. Then the length of $h$, denoted $\abs{h}$, is the sum $\abs{h(g_1)}+\abs{h(g_2)}+\ldots+\abs{h(g_m)}$.
\end{definition}

\begin{remark}\label{CareBase}
Let $G,\Gamma$ be finitely generated groups and $h:G\rightarrow \Gamma$ a morphism. If we choose the base point of the (geometric realization) of the Cayley graph of $\Gamma$ to be the identity element, then the length of the corresponding pseudometric coincides with the length of $h$.  
\end{remark}



\subsection{Diagrams in groups satisfying small cancellation properties} 

In this subsection we collect some results particular to groups satisfying the small cancellation property $C'(1/8)$. We will call such a group an eighth group.  We fix an eighth group $\Gamma$ and a finite generating set $\Sigma_{\Gamma}$. We consider its Cayley graph $Cay(\Gamma, \Sigma_\Gamma)$ with respect to $\Sigma_\Gamma$ and we endowed it with the word metric, i.e. $d_{Cay(\Gamma, \Sigma_\Gamma)}(g,h):=d(g,h)$ is the length of the shortest word (in $\Sigma_{\Gamma}$) representing $g^{-1}h$. Our aim is to analyze geodesics, i.e. shortest paths, between two and three points in the above metric space. There can be many geodesic segments between two points. When two geodesic segments between two points only meet in these points, we call their union a {\em digon}. Likewise when geodesic segments between three distinct points only meet at these points, we call their union a {\em simple triangle}. It is obvious that any geodesic triangle (not necessarily simple) consists of at most one simple triangle and several digons all connected by geodesic segments (see figure \ref{GeodesicT}). If we remove the simple triangle (or the midpoint in the degenerate case) from a geodesic triangle, then we end up with three connected components, we call them (after adding the missing points) the {\em legs} of the geodesic triangle. The length of a leg is the distance from the endpoint it contains to the simple triangle (or the midpoint in the degenerate case).  Finally, by the above discussion, it is apparent that in order to classify the geodesic triangles we need to classify digons and simple triangles.  

\begin{figure}[ht!]
\centering
\includegraphics[width=1.\textwidth]{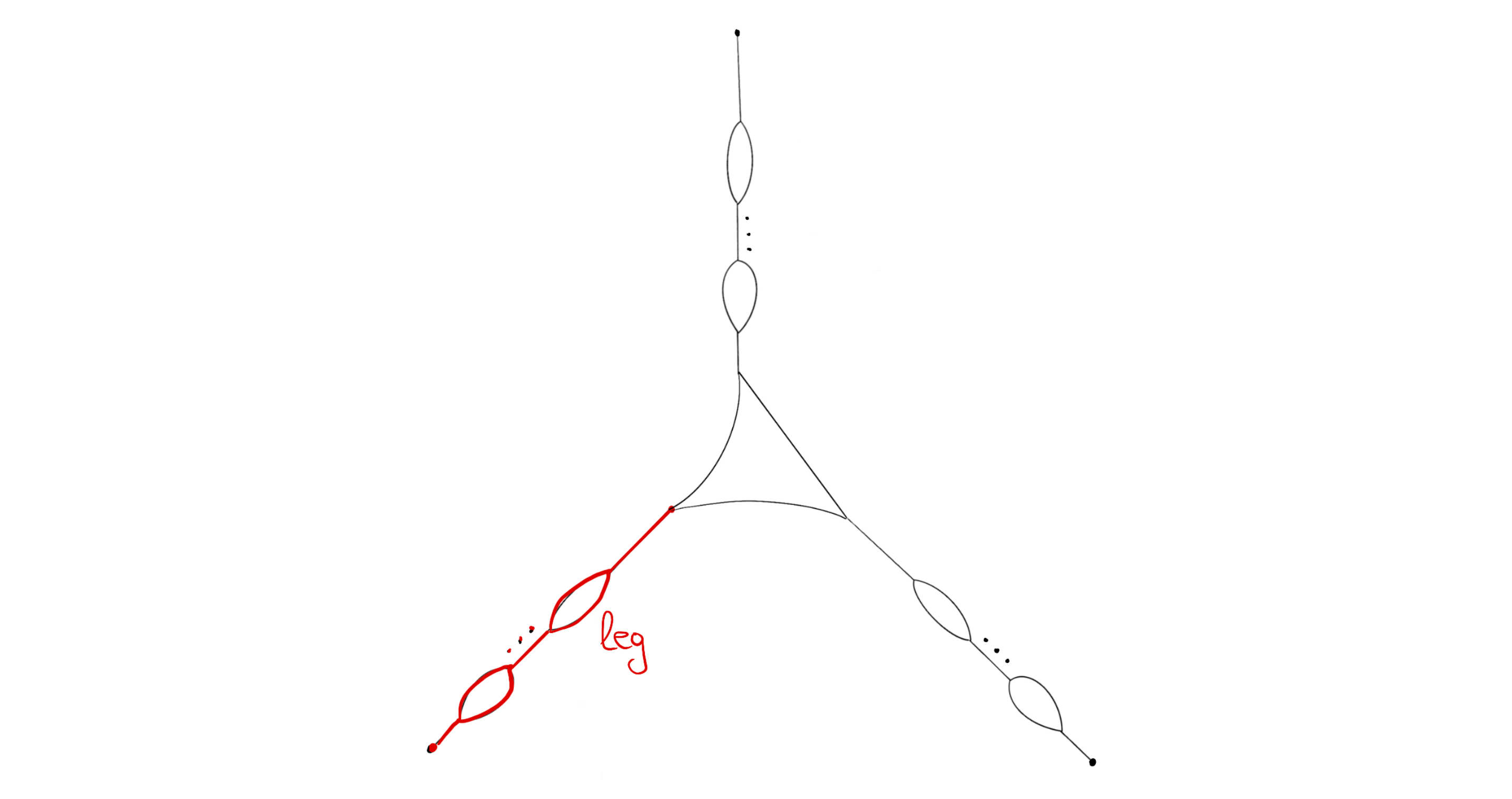}
\caption{A geodesic triangle in a Cayley graph.}
\label{GeodesicT}
\end{figure}

 Digons and triangles in small cancellation groups have been classified in \cite[Section 3.4]{St}. In contrast to arbitrary tessellations of a disk (see figure \ref{Tessellation}), digons in small cancellation groups can only take form $I_1$ (see figure \ref{strebel}). 
 
 \begin{figure}[ht!]
\centering
\includegraphics[width=.5\textwidth]{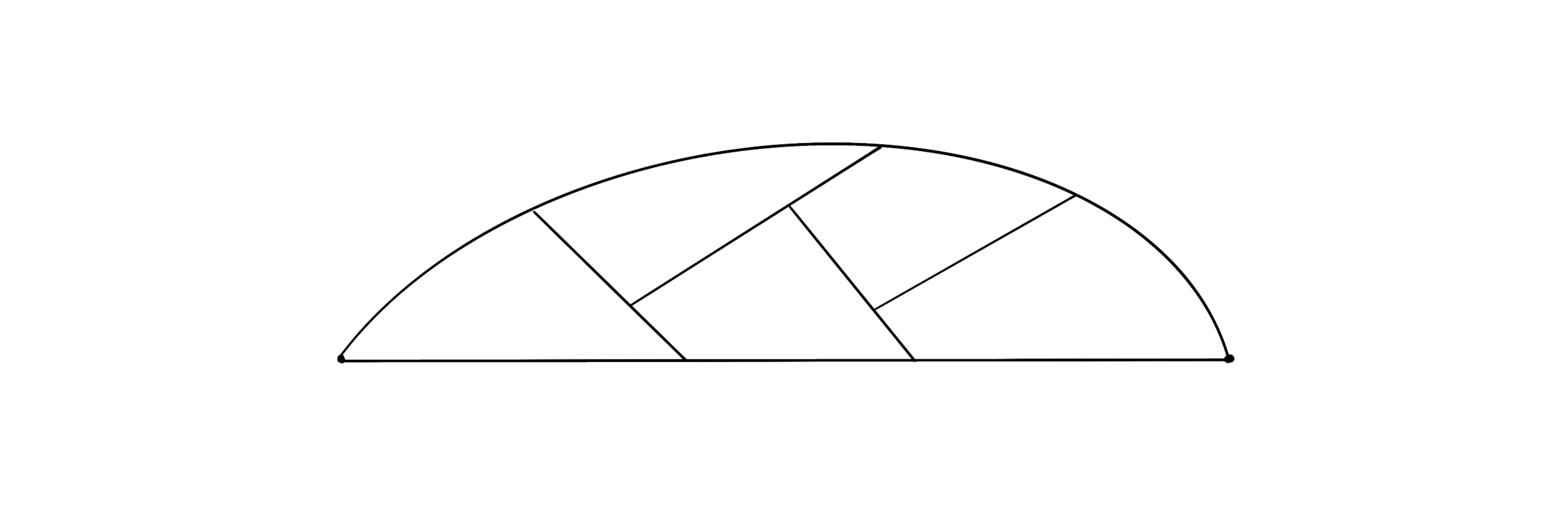}
\caption{An arbitrary tessellation of a digon.}
\label{Tessellation}
\end{figure}
 
 There are more possibilities for simple triangles. A simple triangle has one of the following forms: $I_2, I_3, II, III, IV$ and $V$ (see figure \ref{strebel}). 

\begin{figure}[ht!]
\centering
\includegraphics[width=1.\textwidth]{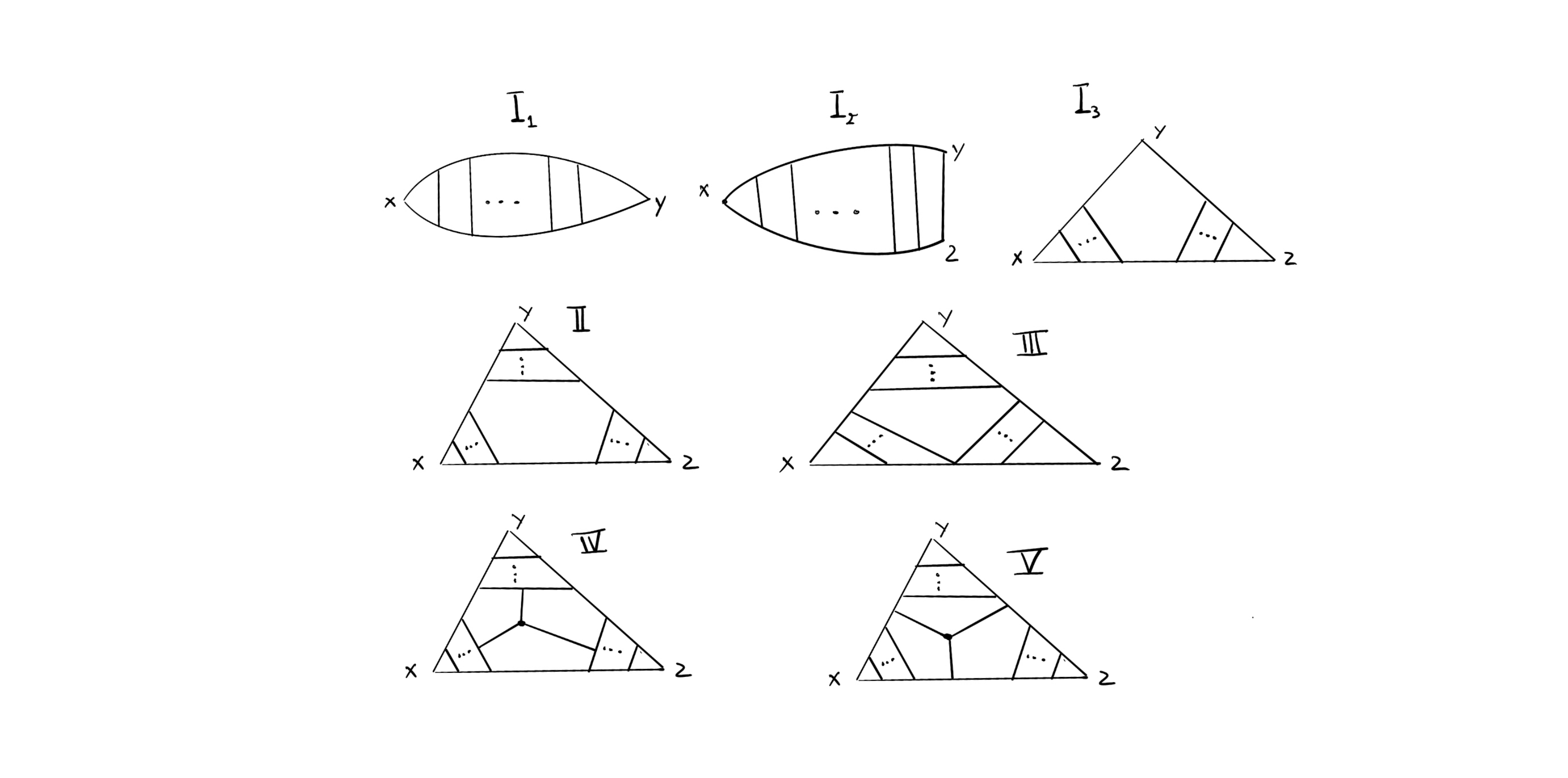}
\caption{Simple geodesic digon and triangles}
\label{strebel}
\end{figure}

 We now focus to {\em eighth groups} that appear as random groups of density $d<1/16$ at length $\ell$. In particular all defining relators have fixed length $\ell$ and groups are torsion free. We first introduce some terminology.
 
 If $\mathfrak{D}$ is the class of digons of $\Gamma$ we define two functions $up, low:\mathfrak{D}\rightarrow Cay(\Gamma, \Sigma_\Gamma)$ such that for any digon $D$, $up(D)$ is the ``upper" geodesic segment and $low(D)$ is the ``lower" geodesic segment of the digon. The ``upper" and ``lower" terms do not have any geometric meaning they are only meant to distinguish between the two geodesic segments forming the digon. Also note that the length of $up(D)$ equals the length of $low(D)$ which is the distance between the two endpoints of the digon. A {\em side of a digon} is inadvertently either $up(D)$ or $low(D)$. 
 
 When a simple triangle of type $I_2$, determined by $x, y, z$ (as in Figure \ref{strebel} type $I_2$), is such that $d(y,z)$ is less than $\lambda\ell$ we call it a {\em semidigon}. We call the point $x$, where the two geodesic segments meet, the {\em summit of the semidigon}. We extend our functions $up$ and $low$ to semidigons in a way that for any semidigon $sD$,  $up(sD)=[x,y]$ and $low(sD)=[x,z]$. In contrast to the digon case, the lengths of $up(sD)$ and $low(sD)$ may differ. A side of a semidigon is inadvertently either $up(sD)$ or $low(sD)$.   
 
 A {\em cell} in $Cay(\Gamma, \Sigma_\Gamma)$ is a simple cycle bearing a defining relator.  Note that this is not formally a ($2$)-cell, i.e. a closed disk up to a homeomorphim, but since such cycles appear in Van Kampen diagrams as boundaries of $2$-cells, we will abuse notation and for convenience we will call them cells. By the  small cancellation property and the fact that Cayley graphs are ``folded", i.e. we cannot have two distinct edges emanating from the same vertex bearing the same generators we observe that if two cells intersect in a segment of length greater or equal to $\lambda\ell$, then they are equal. We will exploit the latter fact together with our understanding of digons and semidigons in $\Gamma$. 
 
 \begin{lemma}\label{LongSegments}
 Let $R=[u,c,\bar{v}, d]$ be the boundary of a cell which is part of a digon or a semidigon in $\Gamma$ (see figure \ref{Sides}). Then the lengths of $u$ and $v$ are strictly greater than $\ell/4$. 
 \end{lemma}
 
\begin{figure}[ht!]
\centering
\includegraphics[width=.8\textwidth]{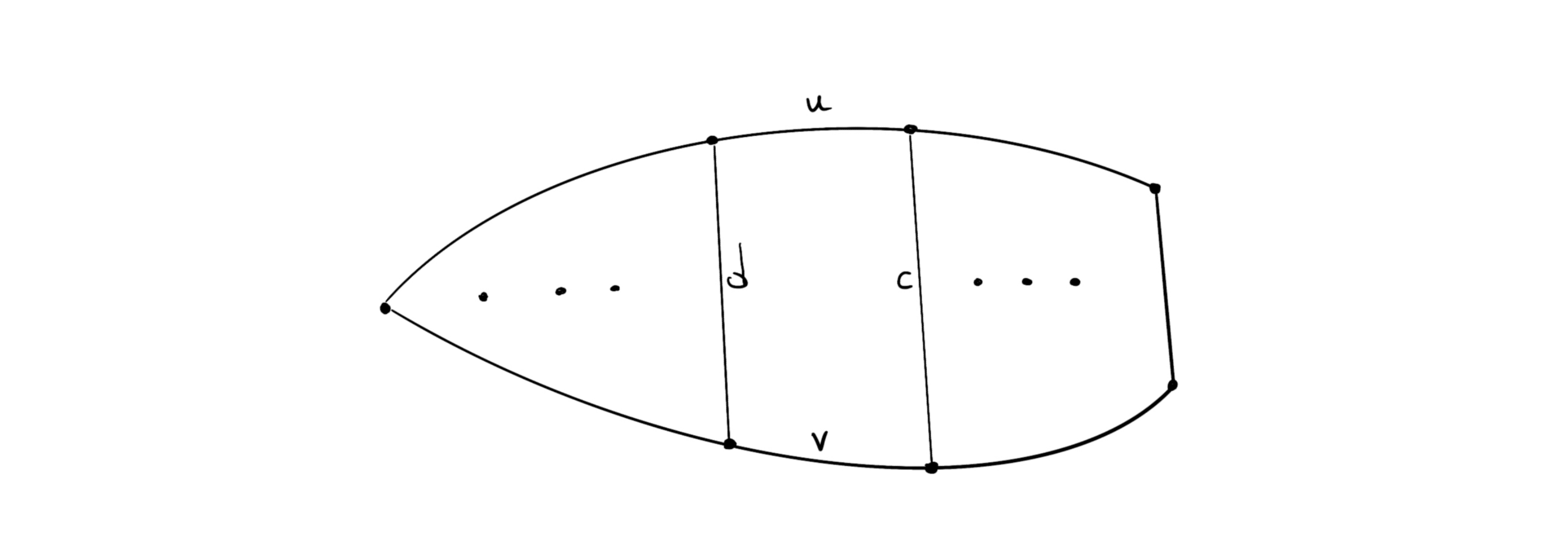}
\caption{A cell in a semidigon.}
\label{Sides}
\end{figure}

\begin{proof}
As noted above the lengths of $d$ and $c$ are strictly less than $1/8\ell$. Hence $\abs{v}\leq \abs{u} +2/8\ell$, by the fact that $u$ is a geodesic segment (as part of a geodesic segment). Now since every defining relator has length $\ell$, we have $\abs{u}+\abs{v}+\abs{c}+\abs{d}=\ell$, thus $\abs{u}+\abs{v}+2/8\ell> \ell$. Combining the two inequalities we get $\abs{u}>\frac{(1-4/8)\ell}{2}=\ell/4$ and similarly $\abs{v}>\ell/4$. 
\end{proof}

\begin{lemma}\label{short_geod_tr}
 If a geodesic triangle has one side of length less than $\ell /8$, then it has two legs of length less than $\ell /8$  and either it does not contain a simple triangle (degenerate case) or, if it does, this simple triangle is a semidigon. 
\end{lemma}
\begin{proof}
Let $x,y,z$ be the endpoints of the geodesic triangle. If $|[y,z]|<\ell/8$, then apparently the leg that contains $y$ and the leg that contains $z$ have lengths less than $\ell/8$. We then assume that the geodesic triangle contains a (non degenerate) simple triangle and we take cases according to the classification of simple triangles by Strebel (see figure \ref{strebel}). If it is of type $I_2$, then, since $|[y,z]|<\ell/8$, it is a semidigon and we can conclude. In any other case, since one side of the simple triangle is contained in $[y,z]$, it has length less than $\ell/8$. Therefore, by the main argument in Lemma \ref{LongSegments}, this side cannot have any division points. In particular, types $I_3, II$ and $III$ also degenerate to semidigons. A simple triangle of type $IV$ cannot occur since the ``legs" of the tripod in the center are short, i.e. $<\ell/8$, hence two of them together with the subsegment of $[y,z]$ that consists of the side of the simple triangle cannot define a cell. Finally, a simple triangle of type $V$ cannot occur as well, in this case there exists a side of the simple triangle which will not be a geodesic since the path that goes through the tripod in the center consists of short segments, i.e. $<\ell/8$, and it will be shorter.          
\end{proof}

Using Lemma \ref{LongSegments} we can prove that a side of a digon determines uniquely the other side of the digon  hence the digon. We call the non-trivial segments of length less than $1/8\ell$ whose endpoints lie on a digon (respectively semidigon) {\em divisors} and their endpoints {\em division points}. 

\begin{corollary}\label{DigonsFixed}
Let $low(D)$ be a side of a digon $D$ in $\Gamma$. Then $up(D)$ is determined uniquely by $low(D)$  (and vice versa). 
\end{corollary}
\begin{proof}
Fix a digon $D$ and the side $low(D)=[x,y]$, order the division points $A_1,A_2,\ldots A_n$ with $A_1$ the closest to $x$. Let $D'$ be a digon with the same lower side as $D$, but different upper side with division points $A'_1, \ldots, A'_m$. The first cells of $D$ and $D'$ meet in a segment longer than $1/4\ell$, since both segments $[x,A_1]$, $[x,A'_1]$ are, by Lemma \ref{LongSegments}, longer than that. In particular, by the small cancellation condition, they must bear the same defining relator. But then if $A_1$, $A'_1$ are distinct we get a vertex in the Cayley graph and two distinct edges emanating from it bearing the same generator (see figure \ref{DigonDetermined} lefthand side). Hence the  first divisors of the digons $D$ and $D'$ have some common initial segment. Suppose, for a contradiction, that the division points nearest to $x$ in $up(D)$ and $up(D')$, say $B_1$ and $B'_1$ respectively are distinct. We consider the second cell in both digons. As in the first cell the division points $A_2$ and $A'_2$ must be the same and the cells must bear the same defining relator, but then if $B_1\neq B'_1$ there would be a reduced path in the union of segments $[x,B'_1]\cup[B'_1,B'_2]$ and this contradicts that $up(D')$ is a geodesic segment (see figure \ref{DigonDetermined} righthand side). 

\begin{figure}[ht!]
\centering
\includegraphics[width=1.\textwidth]{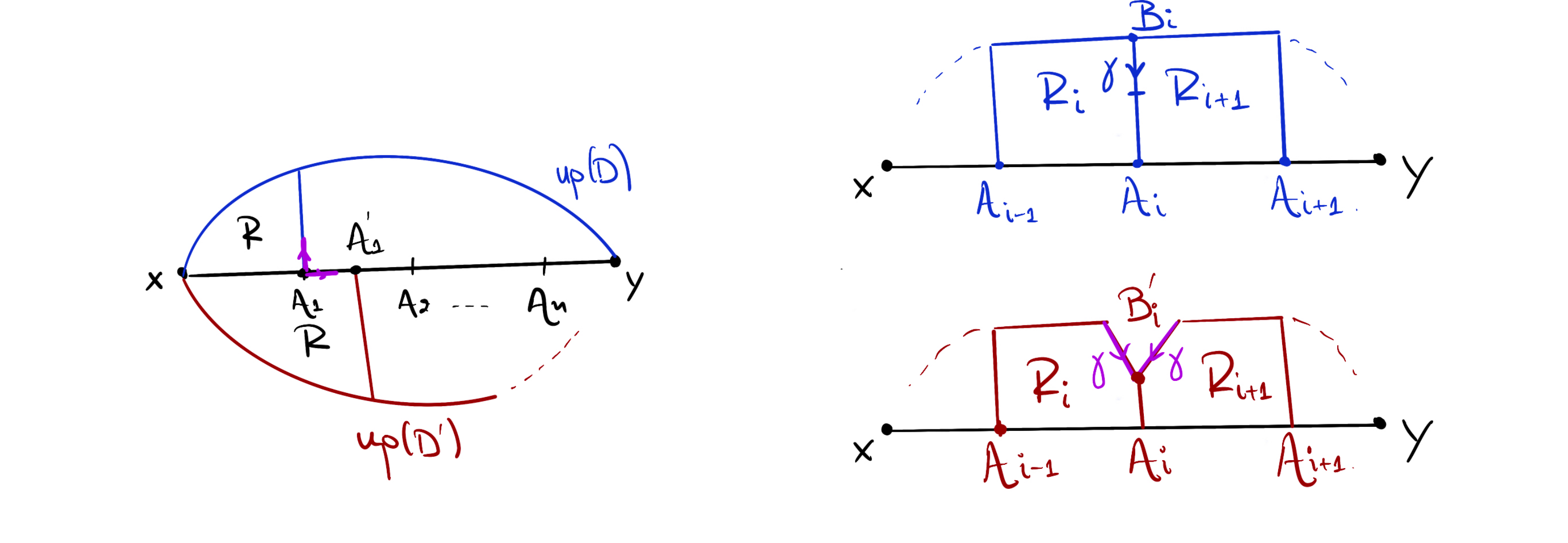}
\caption{A segment as a side of two distinct digons.}
\label{DigonDetermined}
\end{figure}

The above analysis can be carried out cell by cell and an inductive argument gives the result.

\end{proof}

Almost the same result is true for semidigons. Note that having defined the summit of a semidigon we can naturally order the cells of the semidigon according to their distance from the summit. In particular, we can talk about the {\em first cell}, i.e. the cell containing the summit, and the {\em last cell}, i.e. the cell containing the last division points. 

In the case of semidigons the only difference is that the division points of the final cells may differ by at most $1/8\ell$. 

\begin{corollary}\label{SemiDigonsAlmostFixed}
Let $low(sD):=[x,y]$ be a side of a semidigon $sD$ in $\Gamma$. Let $A_1, A_2,\ldots ,A_n$ be the division points on $low(D)$ ordered in a way that $A_n=y$. Let $B_1, B_2,\ldots ,B_n$ be the division points on $up(sD)$ ordered in a way that $B_1$ is closest to $x$. 

Then $up(sD)$ is almost determined by $low(sD)$, i.e. if $sD'$ is another semidigon with $low(sD')=[x,y]$, division points $A'_1, A'_2, \ldots, A'_m$ on $low(sD')=low(sD)$ and $B'_1, B'_2,\ldots ,B'_m$ on $up(sD')$, then $m=n$, $A_i=A'_i$, for $i\leq n$, and $B_i=B'_i$ for $i<n$, and all $n$ cells bear the same defining relators, the last division points $B_n$, $B'_n$ may differ by a distance strictly less than $1/8\ell$.  
\end{corollary}
\begin{proof}
The proof goes word for word with the proof of Corollary \ref{DigonsFixed} up to the last cell. In the last cell, by a small cancellation argument, the division points in $low(sD)$, $A_n$ and $A'_m$ are the same, hence we have $m=n$. Moreover, the last cells bear the same defining relator. We show that if the division points $B_n$, $B'_n$ differ then their distance is no greater than $1/8\ell$. But, by the definition of a semidigon $d(A_n, B_n)<1/8\ell$ and $d(A_n, B'_n)<1/8\ell$, moreover either $B_n$ is on the geodesic segment $[A_n, B'_n]$ or $B'_n$ is on the geodesic segment $[A_n,B_n]$. Thus, $d(B_n,B'_n)<1/8\ell$.      
\end{proof}

\begin{figure}[ht!]
\centering
\includegraphics[width=1.\textwidth]{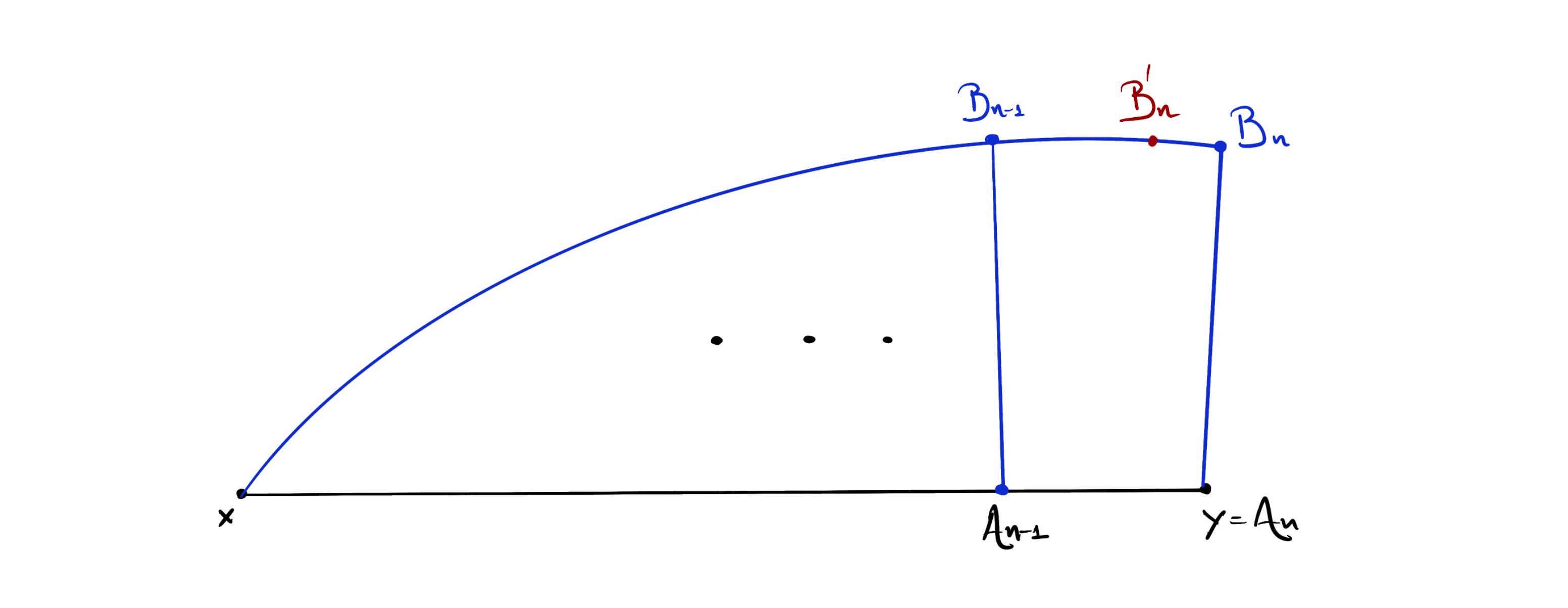}
\caption{The last division points of two semidigons with a common side.}
\label{SemidigonsDetermined}
\end{figure}

We observe that a semidigon can be turned into a digon for a suitable point in its last divisor. We will use next lemma for proving Lemma \ref{IntersectionSemiDigons}. 

\begin{lemma}\label{SemiToDigons}
Let $sD$ be a semidigon. Let $low(sD), up(sD)$ be its sides, $A_0, A_1,\ldots,A_n$ be the division points on $low(sD)$, together with the initial point $A_0$, and $B_1, B_2, \ldots, B_n$ be the division points on $up(sD)$. Then there exists a point on the divisor $[A_n, B_n]$, say $A_{n+1}$, such that $low(sD)\cup[A_n, A_{n+1}]=low(D)$ and $up(sD)\cup [B_n, A_{n+1}]=up(D)$ for a digon $D$. 
\end{lemma}
\begin{proof}
First note that since $low(sD)$ is a geodesic we have $|up(sD)|+d(A_n, B_n)\geq |low(sD)|$, similarly because $up(sD)$ is a geodesic we have $|low(sD)|+d(B_n, A_n)\geq |up(sD)|$. Hence, $|low(sD)|, |up(sD)| \leq \frac{|low(sD)|+|up(sD)|+d(A_n,B_n)}{2}$ and consequently  there exists a point $A_{n+1}$ in the divisor $[A_n, B_n]$ such that $|low(sD)|+d(A_n, A_{n+1})=|up(sD)|+d(B_n, A_{n+1})$. Therefore, we only need to show that either  $up(sD)\cup [B_n, A_{n+1}]$ or $low(sD)\cup [A_n, A_{n+1}]$ is a geodesic.  

Let $[A_0, A_{n+1}]$ be a geodesic segment between $A_0$ and $A_{n+1}$ and consider the geodesic triangle determined by the vertices $A_0, B_n, A_{n+1}$ and the segments $up(sD)$, $[B_n, A_{n+1}]$ and $[A_0, A_{n+1}]$. If there is no simple triangle, we are done. Hence we may assume there is a simple triangle and in particular it must have $B_n$ and a point, say $C$, between $B_n$ and $A_n$ as two of its vertices (this is because $up(sD)$ and $[A_n,B_n]$ have only two points in common). In addition, as $[B_n,A_n]$ is short, i.e. less than $1/8\ell$, the simple triangle can only be of type $I_2$ with division points on $up(sD)$ and $[A_0, A_{n+1}]$ (see figure \ref{DigonFromSemi}).   

\begin{figure}[ht!]
\centering
\includegraphics[width=.8\textwidth]{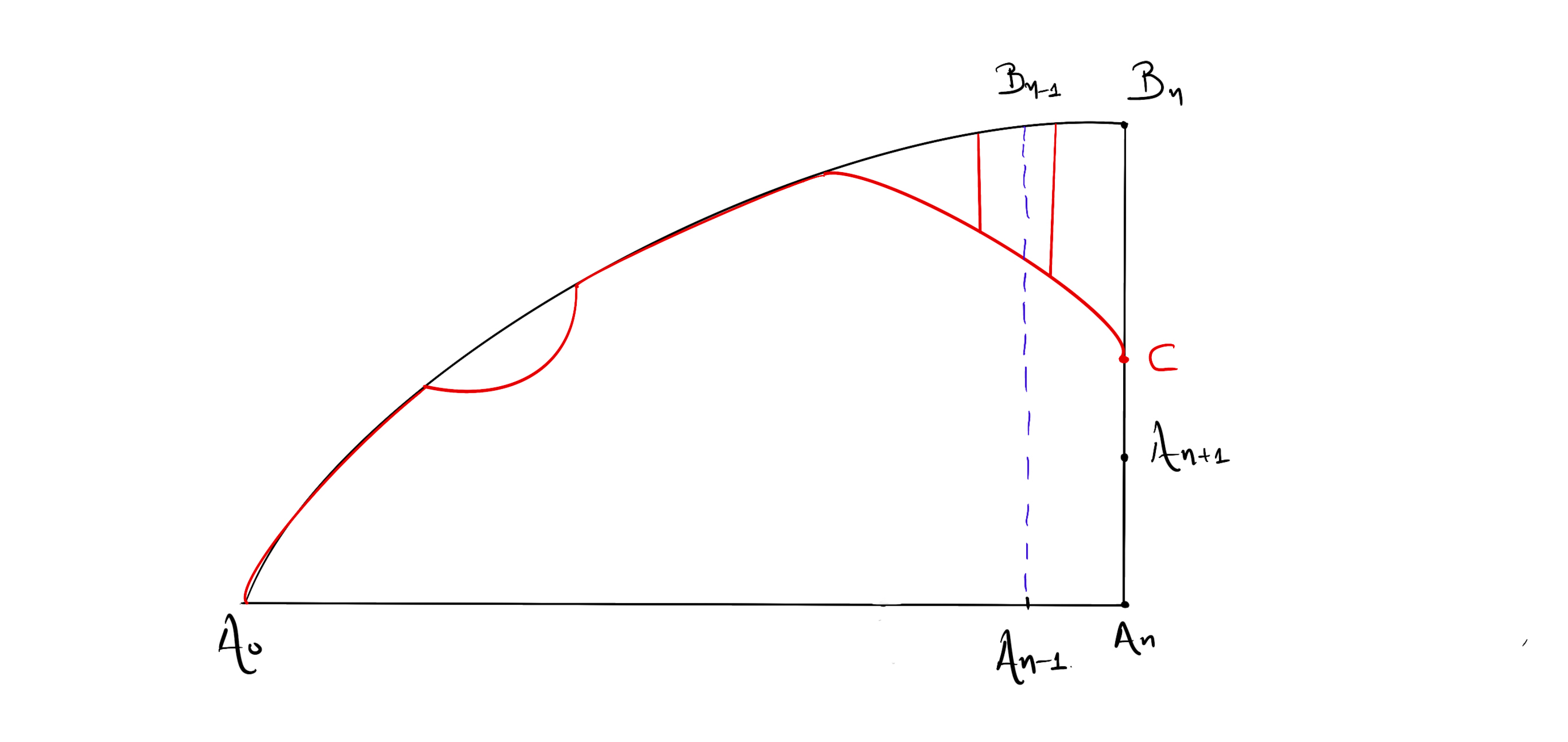}
\caption{The geodesic between $A_0$ and $A_{n+1}$.}
\label{DigonFromSemi}
\end{figure}

If the side of the simple triangle that lies on $up(sD)$ is short, then $up(sD)\cup[B_n, A_{n+1}]$ is obviously shorter than $[A_0, A_{n+1}]$. Hence, the above side is not short and in particular, the last cell of the original semidigon $sD$ and the last cell of the simple triangle of the geodesic triangle $\{A_0, A_{n+1}, B_n\}$ must coincide. An inductive argument shows that all the cells of the simple triangle coincide with cells of $sD$ up to the vertex of the triangle which is not $B_n$ and lies on $up(sD)$ (see figure \ref{DigonFromSemi2}). The latter implies, that  $low(D)\cup[A_n,A_{n+1}]$ is a geodesic as we wanted. 

\begin{figure}[ht!]
\centering
\includegraphics[width=.6\textwidth]{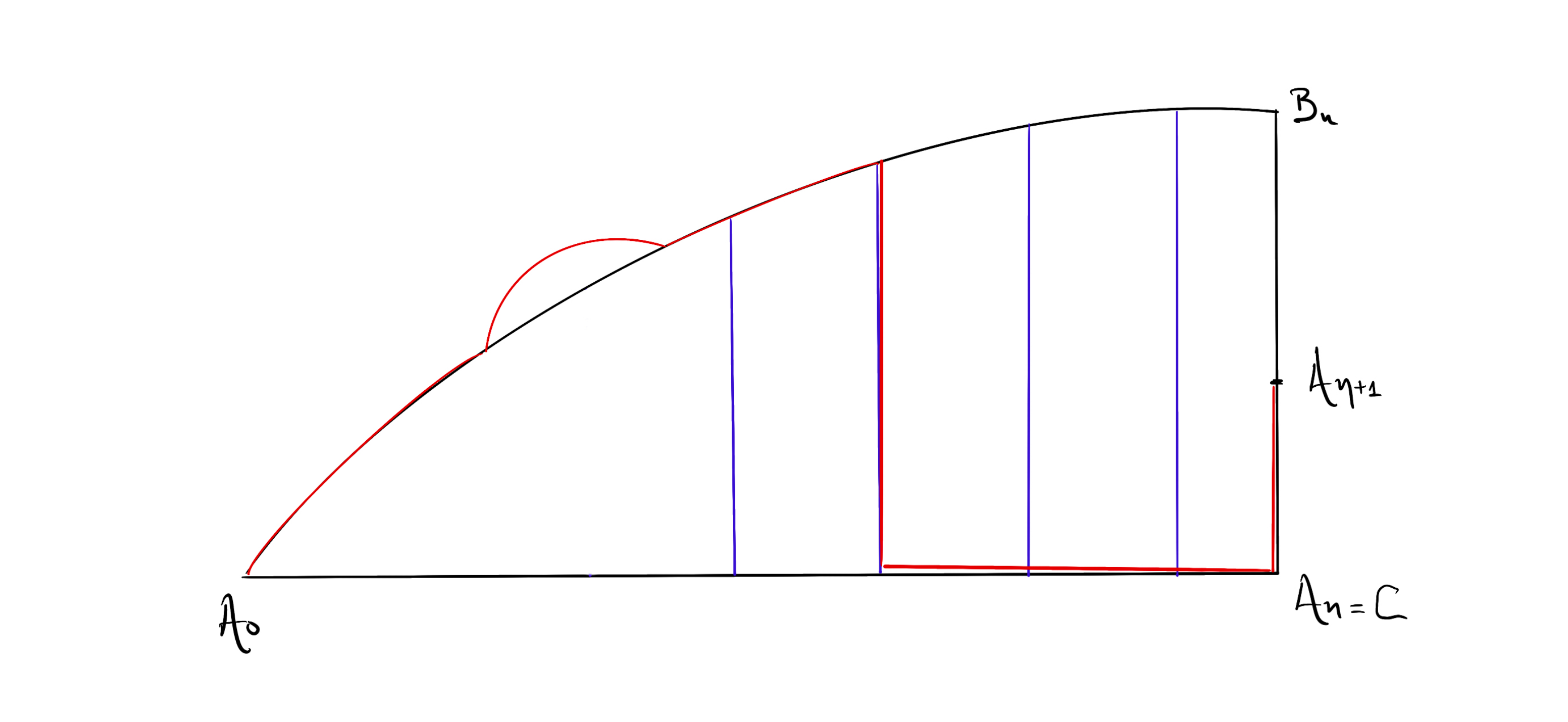}
\caption{The geodesic between $A_0$ and $A_{n+1}$.}
\label{DigonFromSemi2}
\end{figure}

\end{proof}

We now focus on how two geodesic segments between two points in $Cay(\Gamma, S_\Gamma)$ may intersect. We introduce some more terminology before doing so.

\begin{definition}[Subdigon]
Let $D$ be a digon. Let $low(D), up(D)$ be its sides with $A_0, A_1, \ldots,$  $A_{n+1}$ the division points on $low(D)$ together with the initial and terminal points $A_0$, $A_{n+1}$ and $B_1, \ldots, B_n$ the division points on $up(D)$. 

Then a subdigon $E$ of $D$ is a digon which is a subset of $D$ and moreover:
\begin{itemize}
    \item $A'_0=A_i$ for some $0\leq i\leq n$, $A'_j=A_{i+j}$ for $0<j\leq m+1$.   
    \item $B'_j=B_{i+j}$ for $1\leq i\leq m$.
    \item $A'_{m+1}=A_{m+1+i}$ or $B_{m+1+i}$.
\end{itemize}

Where $A'_0, A'_{m+1}$ are the initial and terminal points, $A'_i$, for $0<i\leq m$, are the division points on $low(E)$ and $B'_i$, for $1\leq i\leq m$, are the division points on $up(E)$ of the digon $E$. 
\end{definition}

One can similarly define the notion of a {\em sub-semidigon}, the only difference is that the last division point on $up(sE)$, where $sE$ is the sub-semidigon of the semidigon $sD$, can differ from the corresponding division point on $up(sD)$ but only by a distance no more than $1/8\ell$ .

\begin{lemma} \label{intersectionDigons} Suppose $D_1$, $D_2$ are digons for which $low(D_1)\cup low(D_2):=[x,y]$ is a geodesic segment and $low(D_1)\cap low(D_2)$ has length greater or equal to $1/8\ell$. Then $low(D_1)\cup low(D_2)=low(D)$ for some digon $D$ and $low(D_1)\cap low(D_2)=low(E)$ for some digon $E$.

Moreover, in the case $low(D_2)\subseteq low(D_1)$ the digon $D_2$ is a subdigon of $D_1$. 
\end{lemma} 

\begin{proof}
Let $A_1, \ldots, A_n$ be the division points of $low(D_1)$ to which we add $A_0$ and $A_{n+1}$ - the initial and terminal points of the segment $low(D_1)$. Likewise, $A'_0, A'_1, \ldots, A'_{m+1}$ be the division, initial and terminal points of $D_2$. Without loss of generality $x=A_0$, $y=A_{m+1}'$.  
Digon $D_1$ has $n+1$ cells. Firstly we clain that the first cell in $D_2$ coincides with some cell of $D_1$. This is easily seen by cases. If $A_0'\in [A_n,A_{n+1}]$ then, since $|low(D_1)\cap low(D_2)|\geq \ell/8$, the first cell of $D_2$ is the last cell of $D_1.$ If $A_0'\in[A_k,A_{k+1}]$ where $k<n$, then as $|[A_j,A_{j+1}]|\geq \ell/4$ for all $0\leq j\leq n$ we get $|[A_0',A_1']\cap[A_j,A_{j+1}]|\geq\ell/8$ for either $j=k$ or $j=k+1$. Consequently, the first cell of $D_2$ coincides with the $(j+1)$-st cell of $D_1$. 

Let the first cell of $D_2$ coincide with the $(j+1)$-st cell of $D_1$. We next claim that the $r$-th cell of $D_2$ corresponds to the $(j+r)$-th cell of $D_1,$ where $1\leq r$ and $j+r\leq n+1$. We prove the claim by induction. It is true when $r=1$, as shown in the previous paragraph. We next suppose the claim is true for $1\leq r$ and $j+r<n+1$. As $low(D_1)\cap low(D_2)$ is a geodesic we have that $A_{r}'=A_{j+r}$ and by Lemma \ref{LongSegments} the overlap $[A_r',A_{r+1}']\cap [A_{j+r},A_{j+r+1}]$ is of length at least $\ell /8$. Then by small cancellation the $r+1$ cell of $D_2$ coincides with the $j+r+1$ cell of $D_1$ and the claim is proved.

 Suppose first that $j\neq n$, then  $B_{j+1}\in up(D_2)$ (when $j=n-1$ $B_{j+1}$ should be considered to be $A_{n+1}$), for otherwise we get $B_{j+1}\in [A_1',B_1']$ or $B_{j+1}\in low (D_2)$. In both cases $[B_j,B_{j+1}]\cup [B_{j+1},B_{j+2}]$ is not a geodesic, contradiction. Therefore $B_{j+1}\in up(D_2)$ and we should have $B_{j+1}=B_1'$ because otherwise the contour of either $(j+1)$-st or $(j+2)$-nd cell of $D_1$ would not be reduced.

Now we will show that the union $[A_0,B_{j+1}]\cup [B_{j+1},A_{m+1}']$, where $[A_0,B_{j+1}]\in up(D_1)$ and 
$[B_{j+1},A_{m+1}']\in up(D_2),$ is geodesic, so that, in particular, the two geodesic segments $low(D_1)\cup low (D_2)$ and $[A_0,B_{j+1}]\cup [B_{j+1},A_{m+1}']$ define a digon (they clearly only intersect at the points $A_0$ and $A_{m+1}'$). Note that
$$d(x,y)=|[A_0,A_{0}']\cup [A_{0}',A_{m+1}']|=|[A_0,A_{0}']\cup [A_{0}',B_{j+1}]\cup [B_{j+1},A_{m+1}']|$$

 since $[A_{0}',B_{j+1}]\cup [B_{j+1},A_{m+1}']=up(D_2)$ is geodesic.
And

$$|[A_0,A_{0}']\cup [A_{0}',B_{j+1}]|\geq |[A_0,B_{j+1}]|.$$ Therefore $[A_0,B_{j+1}]\cup [B_{j+1},A_{m+1}']$ is geodesic.
We are done with the case $j\neq n.$

Suppose now that $j=n$, then  we have a situation in Fig 11. Note that by length considerations the division point $B'_1$ must be a point in the segment $[A_{n+1}, B_n]$ and of distance strictly less than $1/8\ell$ from $A_{n+1}$. In order to show that the union $low(D_1)\cup low(D_2)$ is $low(D)$ for some digon $D$ we need to show that the union of segments $[A_0, B'_1]\cup [B'_1, A'_{m+1}]$ is a geodesic segment. 

From the digon $D_2$ we get: 
$$|[A_0', B'_1]|+|[B'_1, A'_{m+1}]|= |[A_0', A_{n+1}]|+|[A_{n+1}, A'_{m+1}]|.\ \ \ \ (1)$$
Therefore, $|[A_0,A_{m+1}]|=|[A_0,A_0']|+|[A_0',B_1']|+|[B_1',A_{m+1}']|$.  The sum of the first two summands in the left side is not less than $|[A_0,B_1']|$ because $[A_0,B_1']$
 is geodesic. Therefore $|[A_0,A_{m+1}]|\geq |[A_0, B'_1]\cup [B'_1, A'_{m+1}]|$ and $[A_0, B'_1]\cup [B'_1, A'_{m+1}]$ is a geodesic segment. This proves that a union of $D_1$ and $D_2$ is a digon.


\begin{figure}[ht!]
\centering
\includegraphics[width=1.\textwidth]{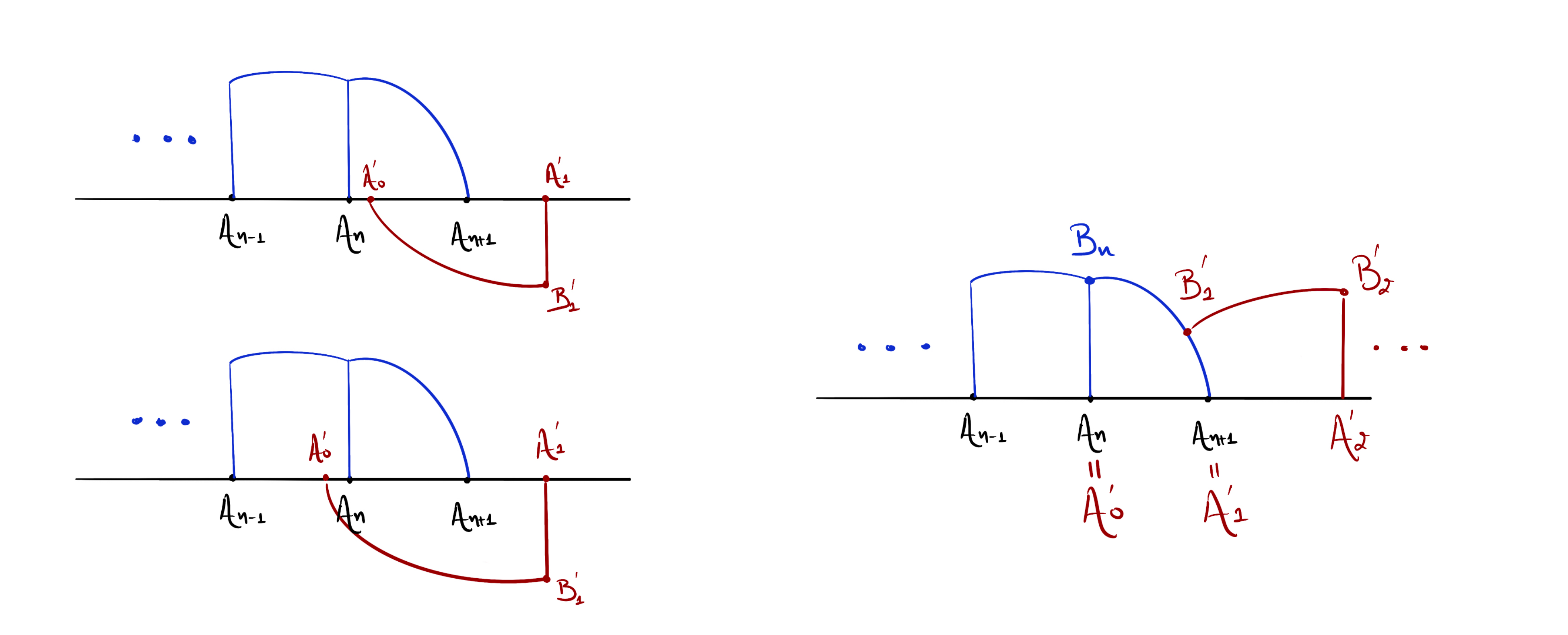}
\caption{The intersection at the final cell of $D_1$.}
\label{LastCells}
\end{figure}




In the case $j=n$ the intersection $low(D_1)\cap low(D_2)$ is $low(E)$ for the digon with a single cell and no divisors determined by the geodesic segments $[A_0',A_{n+1}]$ and $[A_0', B_n]\cup [B_n, A_{n+1}]$. In this case we need to show that $|[A_0',A_{n+1}]|=|[A_0',B'_1]|+|[B'_1,A_{n+1}]|$ (see Figure \ref{LastCells}). Indeed, from the digon $D_1$ we get: 
$$|[A_0, A_{n+1}]|=|[A_0, B_n]|+|[B_n ,A_{n+1}]|. \ \ \ (2)$$ 
Now from the fact we just proved, that the union of $low(D_1)\cup low(D_2)$ is a side of the particular digon $D$, and that $[B_n, A_{n+1}]$ is a geodesic segment that contains $B'_1$, we get:
$$|[A_0, A_{n+1}]|+|[A_{n+1},A'_{m+1}]|= |[A_0, B_n]|+ |[B_n, A_{n+1}]|- |[A_{n+1},B'_1]|+|[B'_1, A'_{m+1}]|. (3) $$

Combining $(2)$ and $(3)$ we get:

$$|[A_{n+1}, A'_{m+1}]|= |[B'_1, A'_{m+1}]|-|[A_{n+1},B'_1]|. \ \ \ (4)$$

Now from $(1)$ using $(4)$ we finally show:

$$ |[A_0',A_{n+1}]|= |[A_0',B'_1]| + |[B'_1, A_{n+1}]|.$$

 The statement about the intersection for the  case $j<n$ follows by similar argument.

\end{proof}

We extend the above lemma to semidigons. Before that we need to extend our vocabulary once more.

\begin{definition}[Orientation of semidigons]
Let $sD_1, sD_2$ be two semidigons such that $low(sD_1), low(sD_2)$ lie on a geodesic segment $[x,y]$. Assume $A_0, A'_0$ are the respective summits of $sD_1, sD_2$ and $A_n, A'_m$ are the respective final division points on $low(sD_1), low(sD_2)$. We say that the semidigons have the same orientation if their summits are closer than their endpoints to the same endpoint of $[x,y]$. Otherwise we say that the semidigons have opposite orientation (see figure \ref{SemiDigonsOrientation}).  
\end{definition}

\begin{figure}[ht!]
\centering
\includegraphics[width=1.\textwidth]{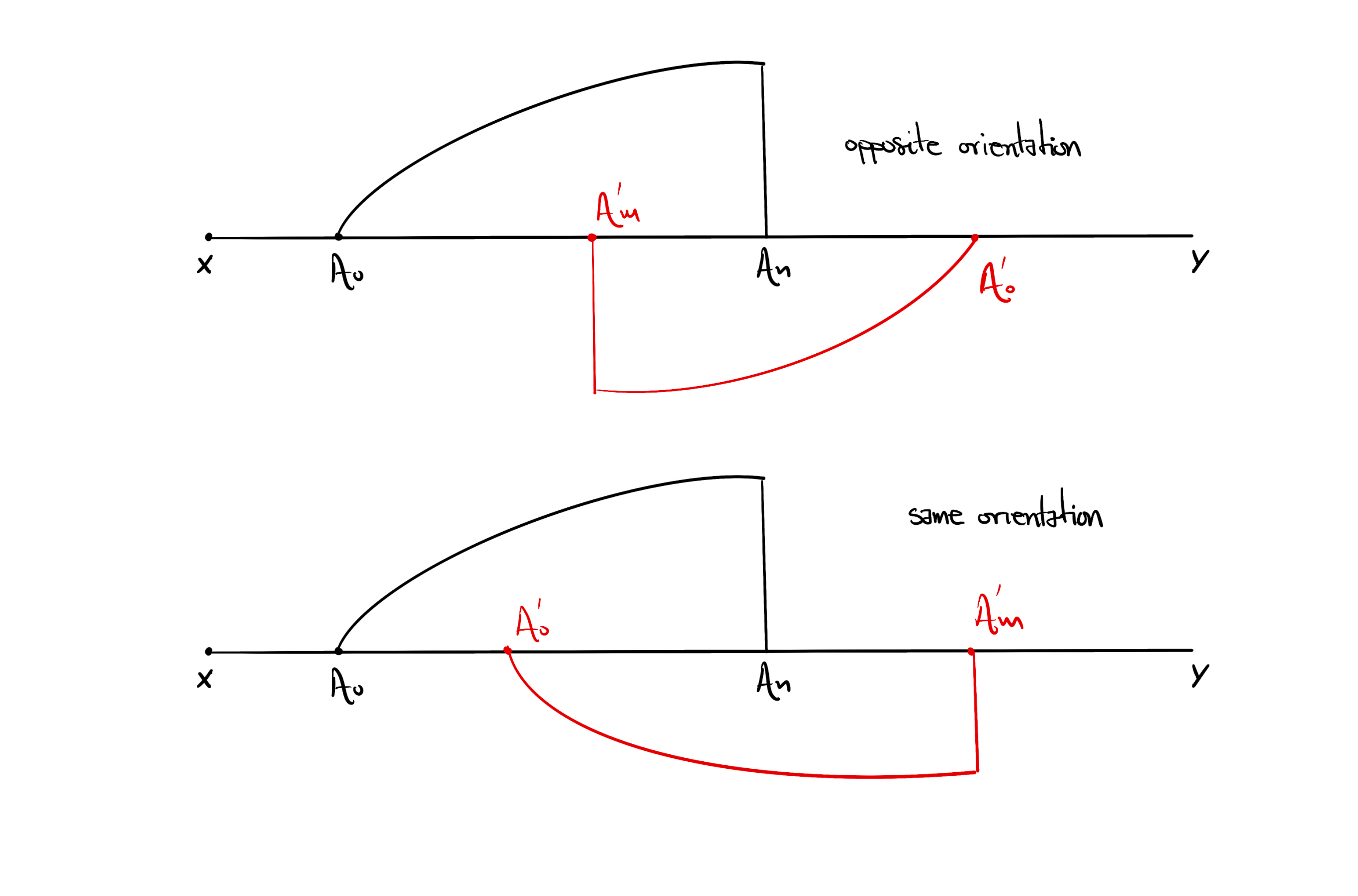}
\caption{Oriented semidigons.}
\label{SemiDigonsOrientation}
\end{figure}

\begin{definition}[Double semidigon]
A double semidigon is the union of four geodesic segments $[x,y], [z,w], [x,z]$ and $[y,w]$ that satisfy the following requirements: $[x,y]$ is disjoint from $[z,w]$ and $[x,z]$ is disjoint from $[y,w]$. The segment $[x,y]$ intersects $[x,z]$ only on $x$ and $[y,w]$ only on $y$. Symmetrically, the segment $[z,w]$ intersects $[x,z]$ only on $z$ and $[y,w]$ only on $w$. Finally, $[x,z], [y,w]$ are short, i.e. $<1/8\ell$, and the tessellation of $\{x,y,z,w\}$ is as in a simple digon (see figure \ref{DoubleSemiDigonTessellation}) 
\end{definition}

For any double semidigon $sDD$ we can define in a similar way to digons and semidigons its sides $low(sDD)$ and $up(sDD)$. 

\begin{figure}[ht!]
\centering
\includegraphics[width=.6\textwidth]{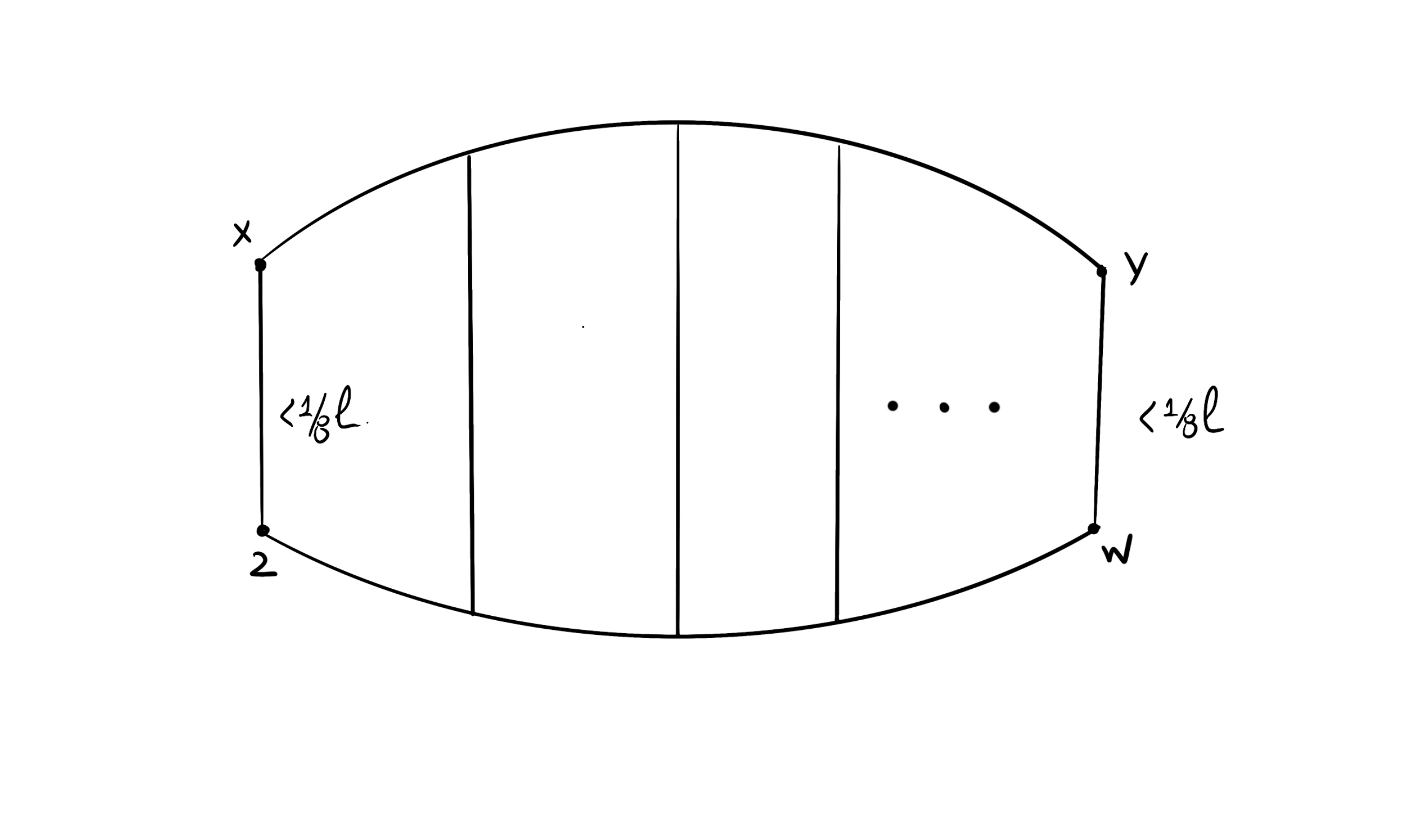}
\caption{A double semidigon.}
\label{DoubleSemiDigonTessellation}
\end{figure}

\begin{lemma}\label{IntersectionSemiDigons}
Suppose $sD_1, sD_2$ are semidigons for which $low(sD_1)\cup low(sD_2):=[x,y]$ is a geodesic segment and $low(sD_1)\cap low(sD_2)$ has length greater or equal to $1/8\ell$. 

If $sD_1, sD_2$ have the same orientation, then $low(sD_1)\cup low(sD_2)=low(sD)$ for some semidigon $sD$ and $low(sD_1)\cap low(sD_2)=low(sE)$ for some semidigon $sE$.

If $sD_1, sD_2$ have opposite orientation, then $low(sD_1)\cup low(sD_2)=low(sDD)$ for some double semidigon $sDD$ and $low(sD_1)\cap low(sD_2)=low(sEE)$ for some double semidigon $sEE$.

Finally, in the special case $low(sD_2)\subseteq low(sD_1)$, independent of orientation, the semidigon $sD_2$ is a sub-semidigon of $sD_1$. 
\end{lemma}
\begin{proof}  Suppose $sD_1, sD_2$ have the same orientation. We let $A_0, A_1, \ldots, A_n$ be the summit and the division points (in order of proximity from the summit) in $low(sD_1)$ and likewise $A'_0, A'_1,\ldots, A'_m$ for $sD_2$. We note in passing that $sD_1$ has $n$ cells and $sD_2$ has $m$ cells.

\begin{figure}[ht!]
\centering
\includegraphics[width=.7\textwidth]{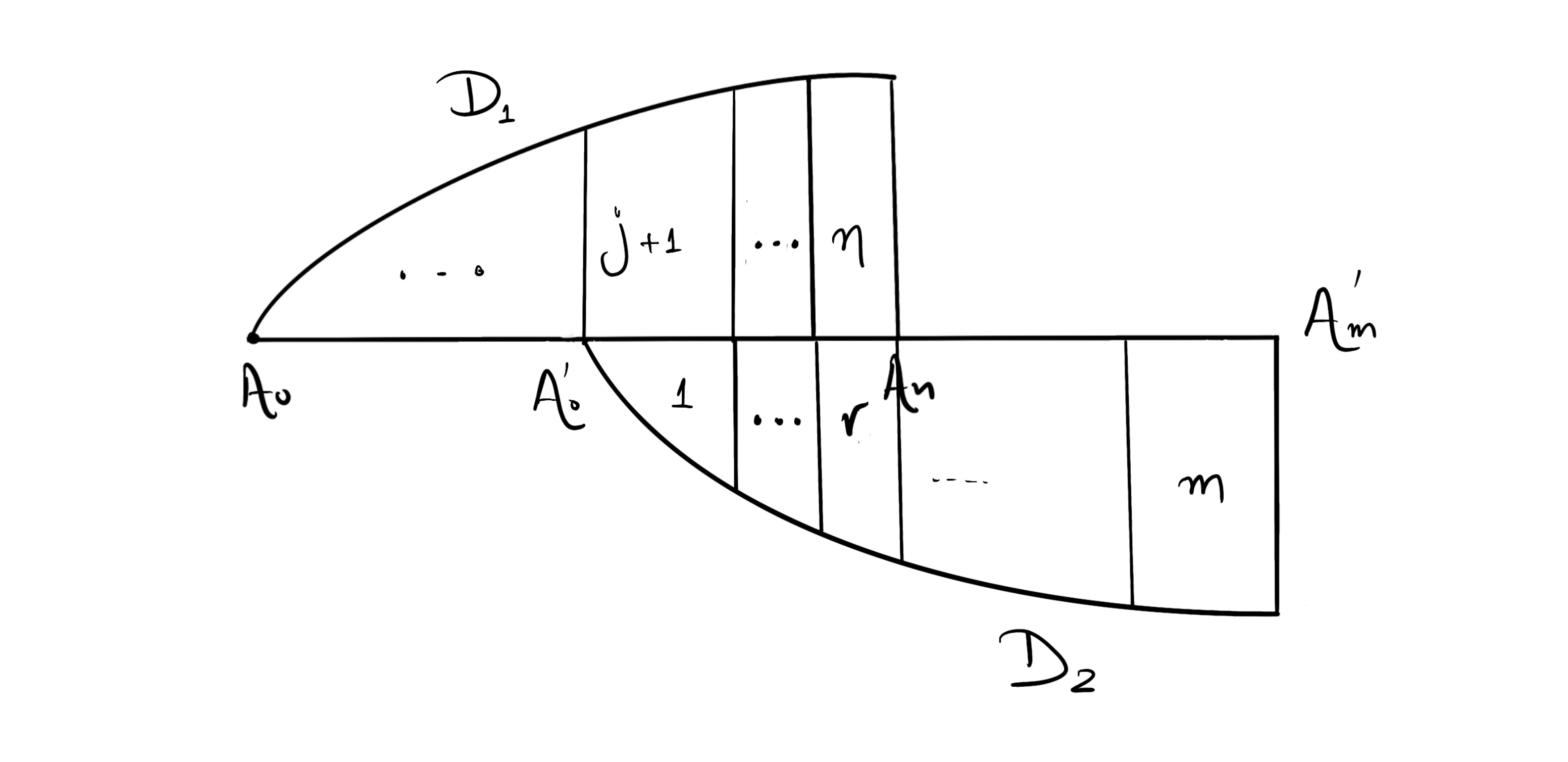}
\caption{Common cells for same orientation semidigons.}
\label{SemiDigonsSameOrientation}
\end{figure}

The same argument as in Lemma \ref{intersectionDigons}  shows  that the $r$-th cell of $sD_2$ corresponds to the $(j+r)$-th cell of $sD_1,$ where $1\leq r$ and $j+r\leq n$.  So the first cell of $sD_2$ coincides with the $(j+1)$-st cell of $sD_1$. In addition, if $sD_2$ has at least $n-j$ cells, then $A'_0=A_j, A'_1=A_{j+1},\ldots, A'_{m-1}=A_{j+m-1}$ and $B'_1=B_{j+1}, \ldots, B'_{m-1}=B_{j+m-1}$. The division points $A'_{n-j}$, $B'_{n-j}$ of $sD_2$ may differ from the corresponding last division points of $sD_1$, but they certainly belong to the segment $[A_n, B_n]$ (see figure \ref{SemiDigonsSameOrientation2}). There is no harm in moving $A_n$ to $A'_{n-j}$ and $B_n$ to $B'_{n-j}$, thus assuming that they coincide. 

\begin{figure}[ht!]
\centering
\includegraphics[width=.8\textwidth]{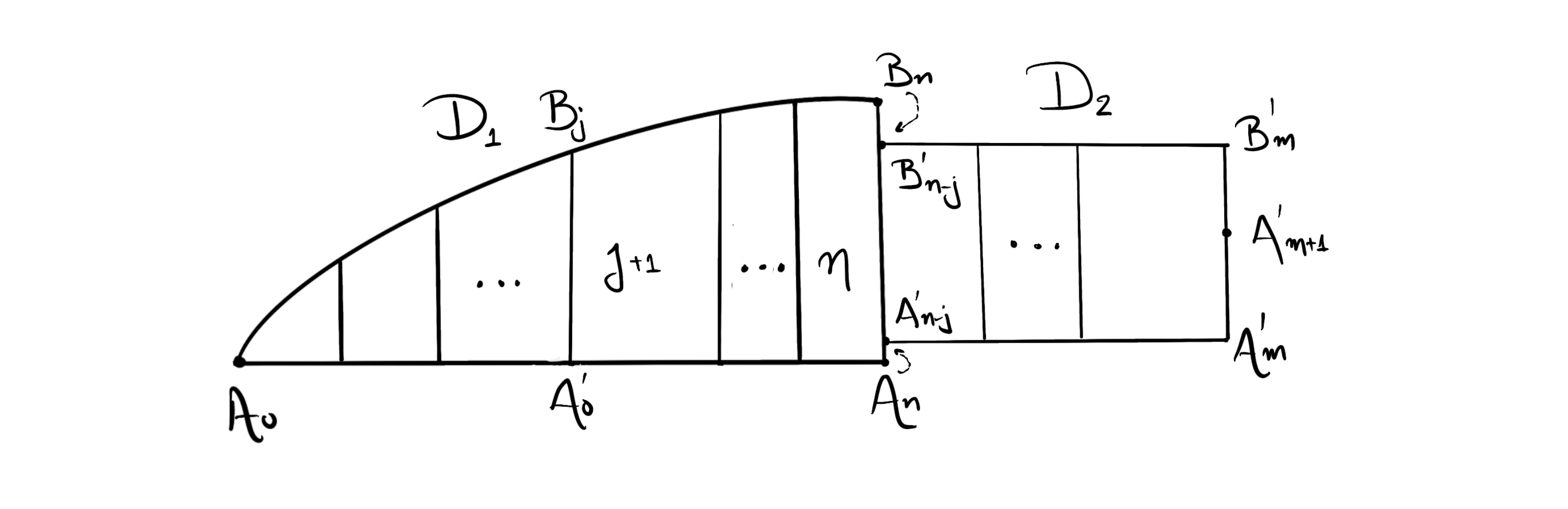}
\caption{Last division points of $sD_1$ may differ from $A'_{n-j}, B'_{n-j}$.}
\label{SemiDigonsSameOrientation2}
\end{figure}

Now let's consider a geodesic triangle  with sides $[A_0,A'_m]=low(sD_1)\cup low(sD_2)$, $[A'_{m},A'_{m+1}]$ (a subsegment of the last divisor of $sD_2$) and $[A_0,A'_{m+1}]$, where  $[A_0,A'_{m+1}]$ is some geodesic between $A_0$ and $A'_{m+1}$. The geodesic triangle consists of some digons and possibly a simple triangle, unless the geodesic $[A_0,A'_{m+1}]$ coincides with the union of the other two sides of the triangle. In the latter case, the union of the (obvious selection of) segments $[A_0,A'_0]\cup [A'_0,B_j]\cup [B_j,A_{m+1}]$ is also a geodesic (as the choice of $A'_{m+1}$ turns $sD_2$ into a digon) and since $[A_0,B_j]$ (the lies in the $up(sD_1)$) is a geodesic, we have that $[A_0,B_j]\cup[B_j,A_{m+1}]$ is a geodesic. In the former case, assuming there are digons and possibly a simple triangle, we may use Lemma \ref{short_geod_tr} and claim that the simple triangle is a semidigon (recall that $|[A'_{m},A'_{m+1}]|<\ell/8)$. Hence, every cell of the geodesic triangle is identified with some cell in either $sD_1$ or $sD_2$. In particular, the geodesic $[A_0,A'_{m+1}]$ passes through either $A'_0$ or $B_j$. In either case, we may see as before that $[A_0,B_j]\cup[B_j,A_{m+1}]$ is a geodesic. Consequently, in both cases, the union of segments 
 $[A_0,B_{j}]\cup[B_{j},B'_{m}]$ (that lies in $up(sD_1)\cup up(sD_2)$) is a geodesic. So $low(sD_1)\cup low(sD_2)=low(sD)$ for a semidigon $sD$ consisting of all cells of $sD_1$ and $sD_2$.  

%

%

When the last cells of $sD_1, sD_2$ coincide the last division points on $up(sD_1), up(sD_2)$ may differ but only by distance strictly less than $1/8\ell$. In this case  a semidigon $sD$ consists of all cells of $sD_1$. In the case $sD_2$ has less than $n-j$ cells, $sD=sD_1$. 

Other statements of the lemma can be proved similarly.

\end{proof}


Interestingly when two digons or semidigons intersect in a short segment, i.e. $<1/8\ell$, and the union of their lower (or symmetrically upper) sides is a geodesic segment, then this union cannot be part of the lower (respectively upper) digon or semidigon. This simple fact will help us prove that any two geodesics between two points can only intersect in a ``single layer" (see figure \ref{OneLayer}).

\begin{lemma}\label{ShortIntersection}
Let $D_1, D_2$ be digons for which $low(D_1)\cup low(D_2)=:[x,y]$ is a geodesic segment and $low(D_1)\cap low(D_2)$ has length strictly less than $1/8\ell$. Then there is no digon or semidigon a side of which contains the union $low(D_1)\cup low(D_2)$.   
\end{lemma}
\begin{proof}
Suppose such a digon, say $D$, existed (the proof for the non-existence of a semidigon is identical). Note that, by Lemma \ref{intersectionDigons} (respectively Lemma \ref{IntersectionSemiDigons}), both $D_1, D_2$ are subdigons of $D$. In particular, the cell $C_1$ of $D_1$ and the cell $C_2$ of $D_2$, that intersect, must either be both identical to a cell of $D$ or each one identical to consecutive cells of $D$. Both cases are clearly impossible as these two cells, $C_1$ and $C_2$, intersect only in a short segment and only on $[x,y]$ (see figure \ref{IncompatibleCells}).  

\begin{figure}[ht!]
\centering
\includegraphics[width=1.\textwidth]{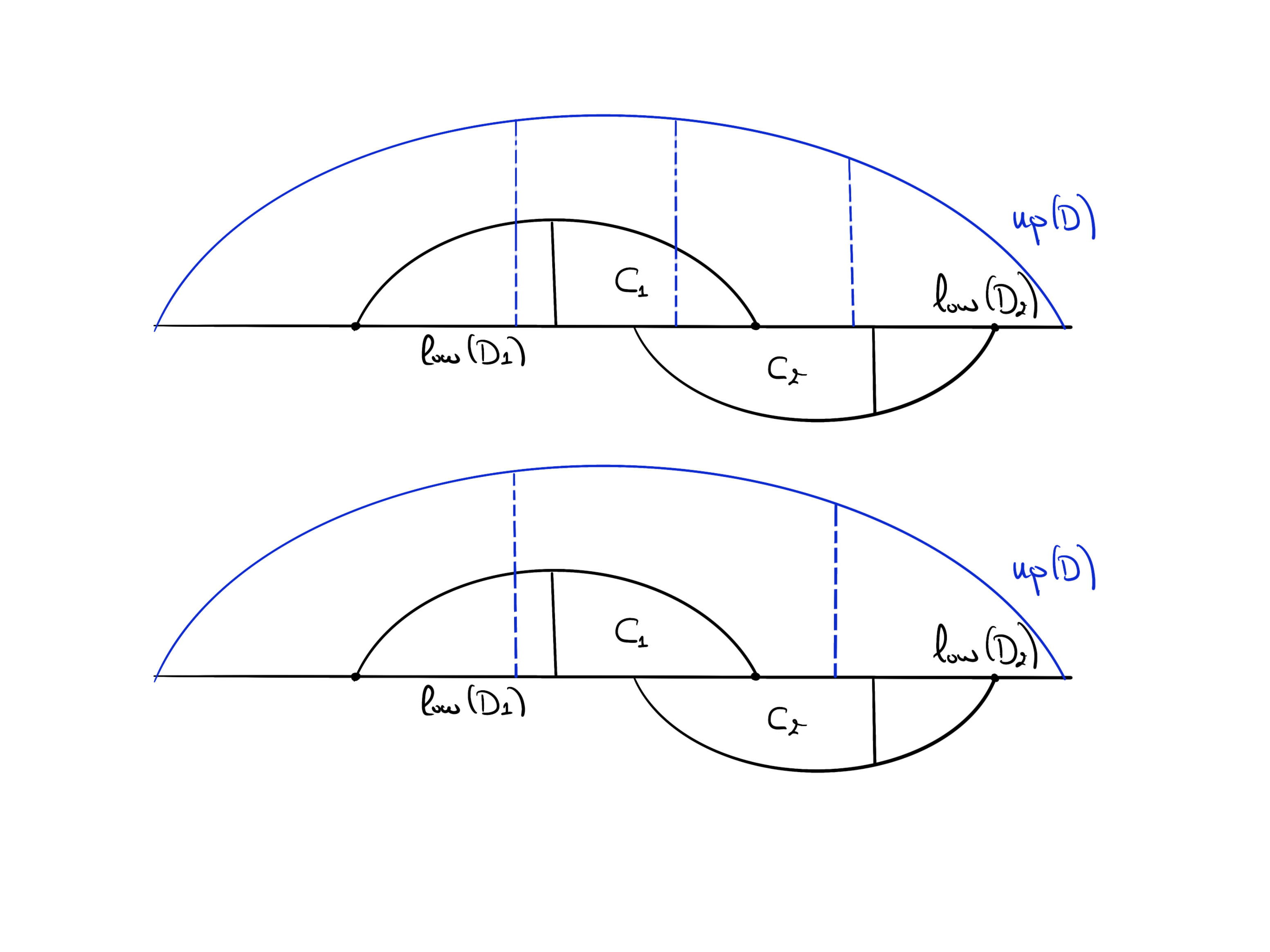}
\caption{Digons intersecting in a short segment.}
\label{IncompatibleCells}
\end{figure}

\end{proof}

\begin{proposition}[Single Layer Intersection]\label{SingleLayer}
Let $[x,y]$ be a geodesic segment in $Cay(\Gamma, \Sigma_{\Gamma})$. Then there exist finitely many digons, $D_1, D_2, \ldots, D_k$, with the following properties:
\begin{itemize}
    \item for each $i\leq k$, $low(D_i)\subseteq [x,y]$;
    \item for any two $i,j\leq k$, $low(D_i)\cap low(D_j)<1/8\ell$;
    \item if $J$ is a geodesic segment between $x$ and $y$, then it is contained in $[x,y]\cup\{D_1, D_2, \ldots,$ $D_k\}$.
\end{itemize}
\end{proposition}
\begin{proof}
We first observe that there can only be finitely many geodesic segments between $x,y$. Let $J_0,J_1, \ldots, J_m$ be a list of them with $J_0=[x,y]$. The union $J_0\cup J_i$ is equal to the union of $[x,y]$ with finitely many digons $D_{i1}, \ldots, D_{im_i}$, whose lower sides are contained in $[x,y]$. Now, consider the digons obtained in this way by the segments $J_1, J_2$. For any two such digons, if their lower sides intersect by more than $1/8\ell$, then, by Lemma \ref{intersectionDigons}, their union is a lower side of a digon and we replace them with this digon (note that by the same lemma both are subdigons of this digon). If their lower sides interesect in strictly less than $1/8\ell$, then, by Lemma \ref{ShortIntersection}, no other lower side of a digon obtained by a third segment can contain the union of their lower sides. After ``merging" the digons that interesect in more than $1/8\ell$, we obtain finitely many digons whose union with $[x,y]$ contains $J_0, J_1, J_2$. Now, assume we have obtained a set of  digons whose union with $[x,y]$ contains $J_0, \ldots, J_i$, and we consider the digons obtained by $J_{i+1}$. As in the basis of the induction, for each such digon we can either ``merge" it with some existing ones or if it intersects by a segment less than $1/8\ell$ with one or two existing digons then the union of their lower sides cannot be part of some other digon (see figure \ref{OneLayer}). This finishes the proof.   
\end{proof}

\begin{remark}\label{OrderSingleLayer}
Note that in the situation of the above proposition we can order the digons from the one closest to $x$ to the one furthest to $x$ and only two consecutive digons in this order can interesect and at most at a short segment, i.e. $<1/8\ell$. 
\end{remark}

\begin{figure}[ht!]
\centering
\includegraphics[width=1.\textwidth]{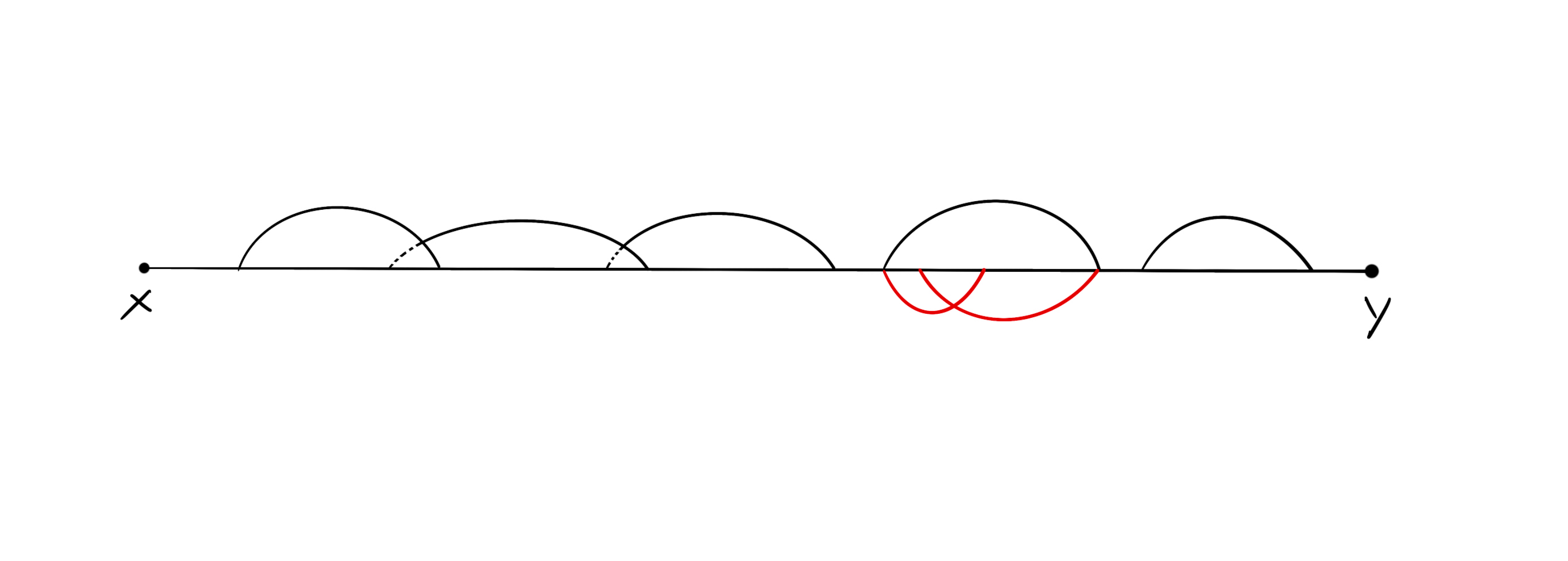}
\caption{Geodesics connecting two points.}
\label{OneLayer}
\end{figure}

Lemma \ref{SingleLayer} is a preliminary result whose generalization will be needed when dealing with geodesic triangles in a limiting sequence approximating a tripod in the limit tree. In the more general case we need to prove a ``single layer" result for the middle part of geodesics that might have different endpoints than the fixed geodesic. Nevertheless, these endpoints are not arbitrary but stay close to the original endpoints relative to the length of the fixed geodesic.     

We next prove that points minimizing distances are rigidly defined in this situation.  

\begin{lemma}\label{MinDistance}
Let $a,b,c$ be three points in $Cay(\Gamma, \Sigma_{\Gamma})$. Fix a geodesic $[a,b]$ between $a$ and $b$. Then there exist at most two points in $[a,b]$ that minimize the distance to $c$. 
\end{lemma}
\begin{proof}
First, note that the claim is trivial if $c$ is a point in $[a,b]$. Thus, we may assume otherwise, i.e. any geodesic triangle defined by $\{a,b,c\}$ and $[a,b]$ is not trivial. 

We assume, for a contradiction, that there exist three distinct points $x_1, x_2, x_3$ on $[a, b]$ that minimize the distance to $c$. We may further assume that $x_2$ is in between $x_1$ and $x_3$ and we consider the possible types, according to Strebel's classification (see figure \ref{strebel}), of geodesic triangles defined by $\{x_1, x_2, c\}$ and $\{x_3, x_2, c\}$ with the segment $[x_1,x_2]$ and $[x_3,x_2]$ lying on $[a,b]$. A simple argument counting lengths, since the lengths of divisors are comparatively small, shows that the only possibility is triangles of type $I_2$. In any other case, there exist points in $[a,b]$ closer to $c$ than $x_1$ and $x_3$ (see figure \ref{Tripods1}). 

\begin{figure}[ht!]
\centering
\includegraphics[width=1.\textwidth]{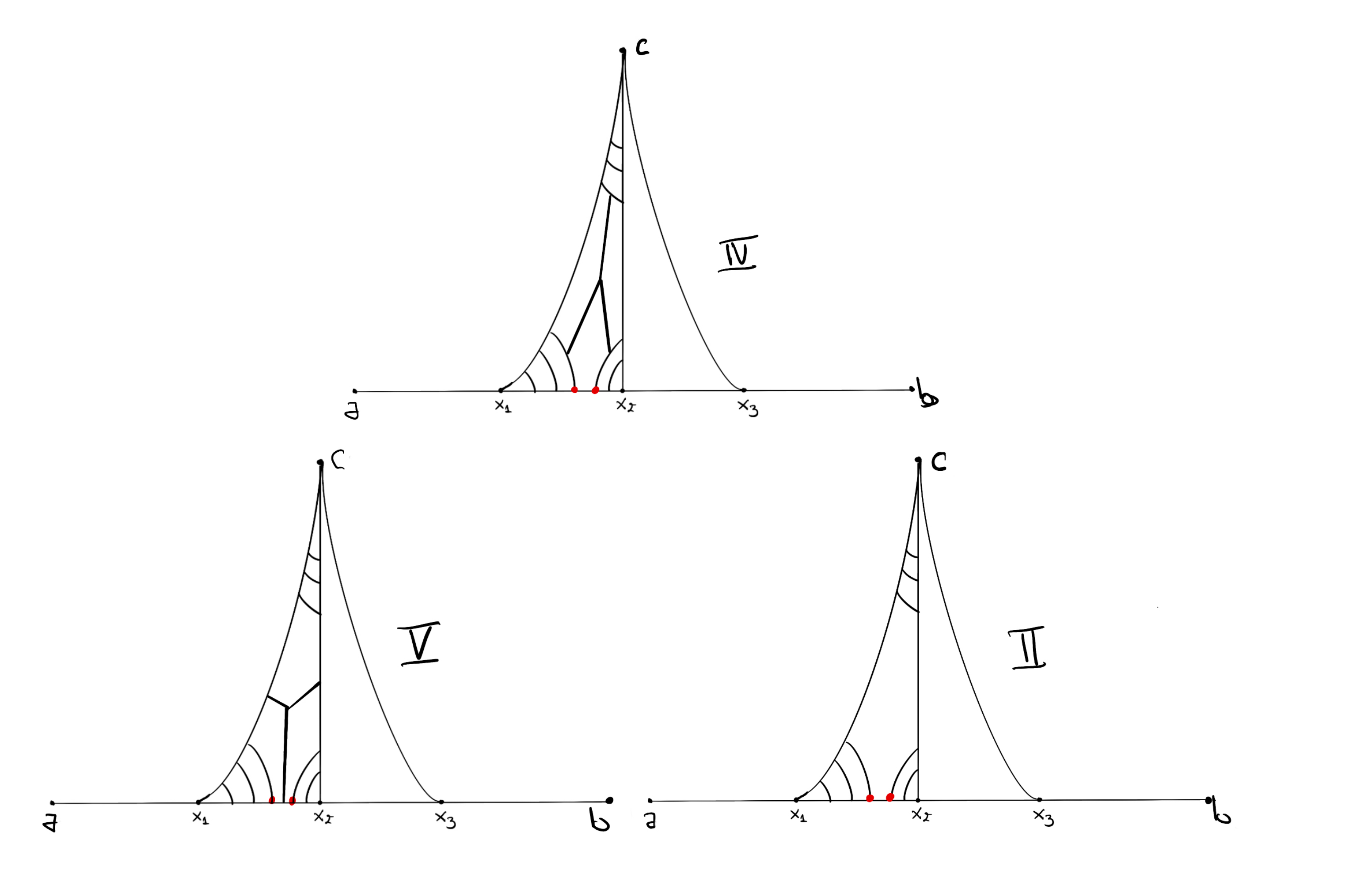}
\caption{Geodesic triangles between $x_1, x_2, c$.}
\label{Tripods1}
\end{figure}

We now fix a geodesic $[x_2, c]$ between $x_2$ and $c$ and consider the geodesic triangles as above that include this geodesic. By our previous argument both simple triangles in these geodesic triangles must be of type $I_2$.  And the endpoints of the simple triangles must lie on $[a, b]$ and in particular be $x_1, x_2$ for the geodesic triangle determined by $\{x_1, x_2, c\}$ and $x_2, x_3$ for the geodesic triangle determined by $\{x_2, x_3, c\}$. If not, the points $x_1,x_2, x_3$ do not minimize the distance to $c$ (see figure \ref{Tripods2}).

\begin{figure}[ht!]
\centering
\includegraphics[width=.8\textwidth]{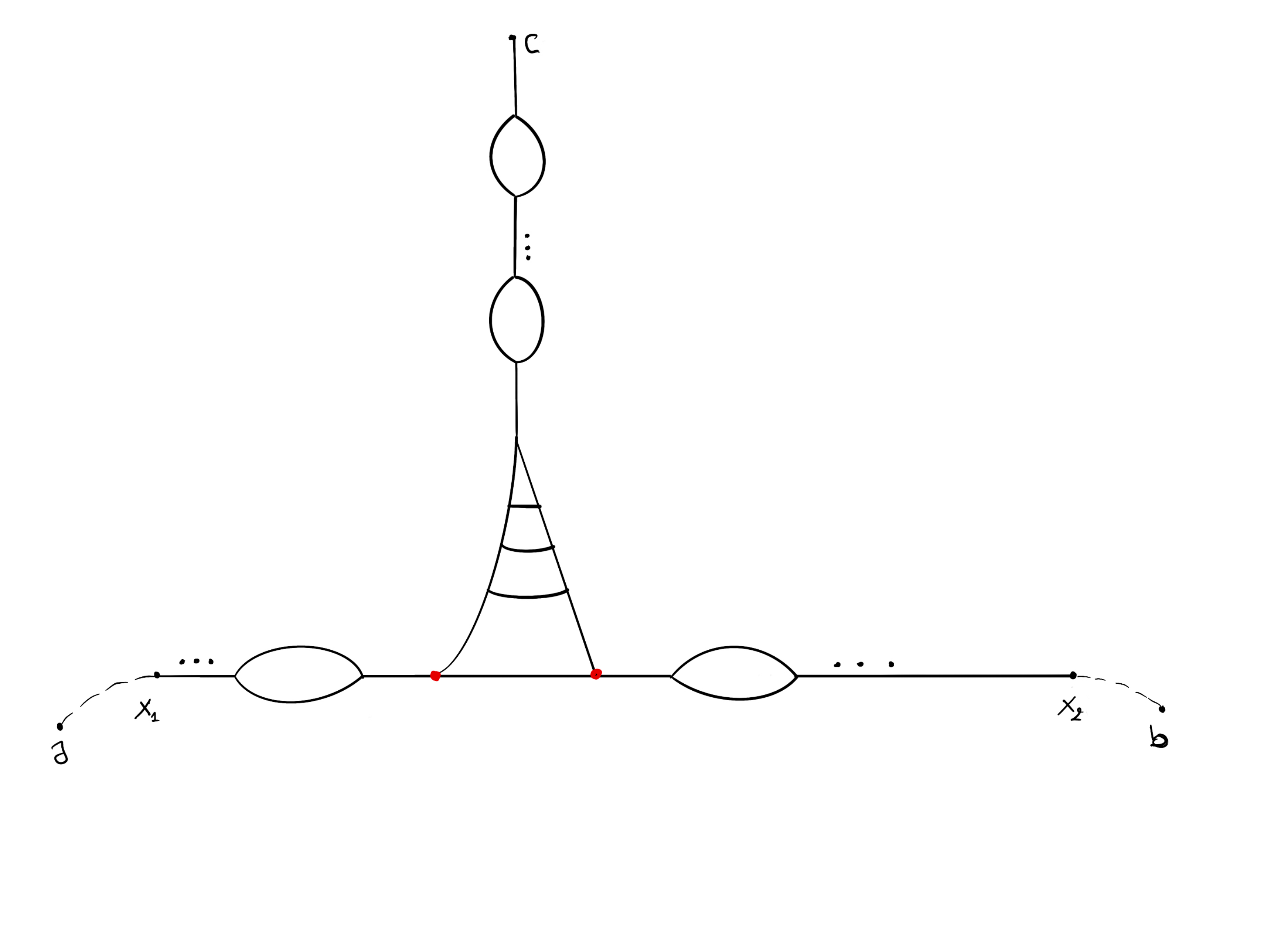}
\caption{Corners of simple triangles coincide with $x_1, x_2, x_3$.}
\label{Tripods2}
\end{figure}

We now take cases, according to the length of $[x_2, c]$. First assume that the length is strictly less than $1/8\ell$. Then, by definition, none of the geodesic triangles contain a digon, in particular they are both cells and have length $\ell$. But then a simple calculation shows that $x_2$ is strictly closer to $c$ than $x_1$ or $x_3$, unless $[x_1, x_3]$ is not a geodesic. Note that this argument shows that the segment from $x_2$ to $A_0$, the endpoint of the simple triangle that lies on $(x_2, c)$, in the geodesic triangle determined by $\{x_1, x_2, c\}$, $[x_1, x_2]$ and $[x_2, c]$, must have length greater than $1/8\ell$. Symmetrically the same holds for the segment $[x_2, A'_0]$, i.e. it has length greater than $1/8\ell$ (see figure \ref{Tripods3}). Hence, the simple triangles must have a long intersection, i.e. greater than $1/8\ell$. In particular the division points of the divisors in the corresponding semidigons must coincide and the last cells of  the simple triangles either intersect each other in a short segment, i.e. $[x_2, A_n]$ has length strictly less than $1/8\ell$, or they must be identical (see figure \ref{Tripods3}). In the former case, the distance from $x_1$ to $A_0$ is strictly greater than the distance from $x_2$ to $A_0$. This is because, $d(x_2, A_0)\leq d(x_2, A_n)+d(A_n, B_n)+d(B_n, A_0)\leq 2/8\ell+d(B_n,A_0)$ and $d(x_1, B_n)>2/8\ell$, otherwise $[x_1,x_2]$ is not a geodesic. This contradicts that $x_1$ minimizes the distance to $c$. In the latter case, this forces $x_3$ (or symmetrically $x_1$) to be a point in the geodesic $[x_1, x_2]$, i.e. in between $x_1$ and $x_2$, which is a contradiction, since we have chosen otherwise.  

\begin{figure}[ht!]
\centering
\includegraphics[width=.8\textwidth]{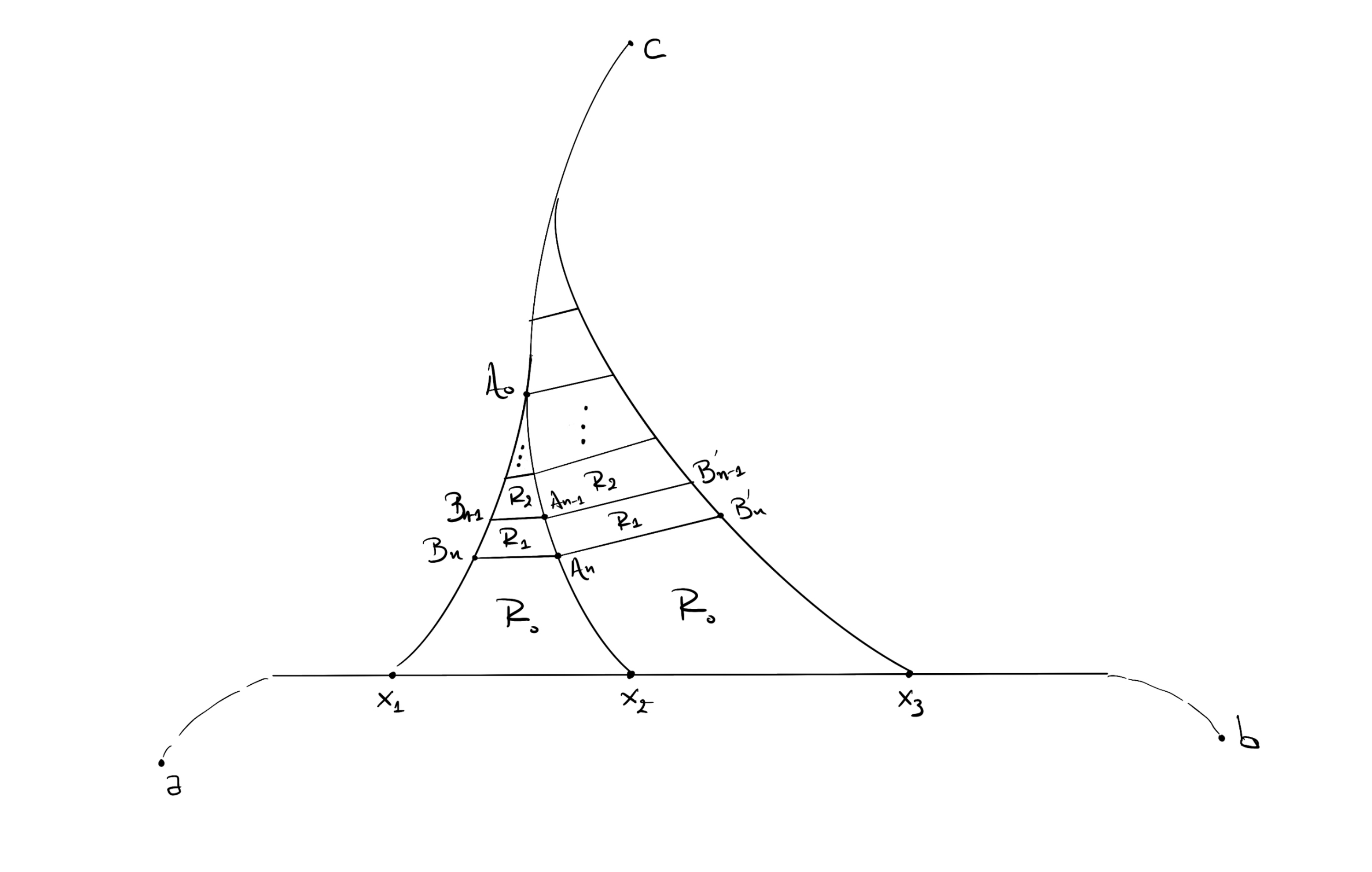}
\caption{Two intersecting simple triangles of type $I_2$.}
\label{Tripods3}
\end{figure}

\end{proof}

%

\subsection{Tripod stabilizers}\label{tripod}
In this subsection we will show that an action obtained by a sequence of morphisms, $(h_i)_{i<\omega}:G\rightarrow \Gamma_i$, from a finitely generated group $G$ to a finitely presented eighth group with fixed length of defining relators $i$, for $i<\omega$, factors through a faithful action with trivial tripod stabilizers. 

We note that since $G$ is countable, up to refining the sequence we may assume that it is {\em stable}, i.e. for every $g$ in $G$, there exists $n(g)<\omega$, such that either $h_i(g)\neq 1$ for all $i>n(g)$ or $h_i(g)=1$ for all $i>n(g)$. The {\em stable kernel}, $\underrightarrow{Ker}(h_i)$, of such sequence consists of the elements that are eventually mapped to $1$.


Some preliminary lemmas are in order. For the rest of the subsection we assume the setting described in the beginning. In particular, by Corollary \ref{LimitActionDominating} and Theorem \ref{CenLoc}, a subsequence of the obtained actions on the (rescaled) Cayley graphs $(X_{\Gamma_i}, 1)$ converges to a minimal action on a real tree $(T, *)$. We will denote the metric of the real tree by $d_T$ and the (rescaled) metric of the Cayley graph $X_{\Gamma_i}$ by $d_i$, for $i<\omega$. 

We first prove a generalization of Proposition \ref{SingleLayer}. We will say that a subsegment $J\subseteq I$ of a geodesic segment $I$ lies in the {\em $1/m$-middle part} of $I$ if it doesn't contain any point within the balls of radius $|I|/m$ centered on the endpoints of $I$. In particular, the subsegment $J$ contained within the ball of radius $\frac{m-2}{m}|I|$ centered at the mid-point of the segment $I$ lies in the $1/m$-middle part of $I$.

To increase the readability, in what follows, the reader may fix the value of $m$ to be some big natural number. We do not do so because the flexibility to increase $m$ will prove useful in some subsequent results. 

\begin{definition}[Single layer configuration]
Let $[a,b]$ be a geodesic segment in $Cay(\Gamma, \Sigma_{\Gamma})$. Let $S$ be a finite subset of $\Gamma$. We say that $[a,b]$ admits an $1/m$ single layer configuration with respect to $S$ if there exist finitely many digons, $D_1, D_2, \ldots, D_k$ a semidigon $sD$ or a double semidigon $sDD$ with lower sides in the $1/m$-middle part of $[a,b]$ that satisfy the following conditions:
\begin{itemize}
\item for every $i,j\leq k$, $D_i\cap D_j<1/8\ell$, and $D_i\cap sD<1/8\ell$; 

\item for every $s\in S$, any $\frac{1}{m-2}$-middle part of  $s\cdot[a,b]$ lies in the union of $[a,b]\cup\{D_1, \ldots, D_k,$ $ sD, sDD\}$.
\end{itemize}
\end{definition}

In the situation of the above definition we say that the union of $[a,b]$ with the digons, the semidigon or the double semidigon covers the $\frac{1}{m-2}$-middle parts of the corresponding geodesics.  

\begin{lemma}\label{ExtendedSingleLayer}
Let $[a,b]$ be a nontrivial segment in the real tree $T$ and $g\in G$ an element that fixes both endpoints. We fix a natural number $m\gg 10$. Then for any $K<\omega$, there exists $n(K,m)$, such that for all $i>n(K,m)$, the segment $[a_i,b_i]$ in $Cay (\Gamma _i,  \Sigma_{\Gamma}),$ where $(a_i)_{i<\omega}$ approximates $a$ and $(b_i)_{i<\omega}$ approximates $b$, admits an $1/m$ single layer configuration with respect to $\{g, g^2, \ldots, g^K\}$. Moreover, exactly one of the following three can  happen:
\begin{itemize}
    \item The union of a double semidigon $sDD$ with $[a_i,b_i]$ covers all $\frac{1}{m-2}$-middle parts of all geodesics;
    \item The union of a semidigon and finitely many digons with $[a_i,b_i]$ covers all $\frac{1}{m-2}$-middle parts of all geodesics;
    \item the union of finitely many digons with $[a_i,b_i]$ covers all $\frac{1}{m-2}$-middle parts of all geodesics.
\end{itemize}
\end{lemma}
\begin{proof}
Let $L$ be the length of $[a,b]$ in $T$. For the given $K$ we choose $n(K)$ so that all powers of $g$ up to $g^K$ move both $a_i$ and $b_i$ in $X_{\Gamma _i}$ at distance not more than $\frac{L}{2m}$ from themselves and $\delta _i< \frac{L}{2(m-1)m}. $  We now consider $Cay (\Gamma _i,  \Sigma_{\Gamma})$ instead of $X_{\Gamma _i}$. Recall that the hyperbolicity constant of  $Cay (\Gamma _i,  \Sigma_{\Gamma})$ is $i$ and $\delta _i=i/\ell _i$, where $\ell _i$ is the rescaling factor. Observe that by the hyperbolicity of the spaces $Cay (\Gamma _i,  \Sigma_{\Gamma})$ the $\frac{1}{m}$-middle parts of $g[a_i,b_i]$ are within the $2i$ ball of the  $1/m$-middle part of $[a_i,b_i]$. Indeed, we can consider, for each $1\leq j\leq K$, the quadruples defined by the geodesics $[a_i,b_i], g^j[a_i,b_i], [a_i,g^ja_i], [b_i,g^jb_i]$. Each such quadruple can be split in two geodesic triangles. Points on  $g^j[a_i,b_i]$ that are at the $\frac{1}{m}$-middle part are at distance greater than  $\frac{L\ell _i}{m}$ from $g^{j}a_i$ and $g^jb_i$.   Since $\frac{L\ell _i}{m}\geq\frac{L\ell _i}{2m}+i$, they are  at distance less than $2i$ from $[a_i,b_i]$.   Since each such quadruple can be split in two geodesic triangles,  we take cases according to the various possibilities for these triangles.
\begin{itemize}
    \item[Case 1] Suppose for some $j_1, j_2\leq K$ the geodesic triangles defined by $\{a_i, g^{j_1}a_i, b_i\}$ and $\{a_i, g^{j_2}b_i,$ $b_i\}$ in $Cay (\Gamma _i,  \Sigma_{\Gamma})$ contain simple geodesic triangles $\Delta_1$ and $\Delta_2$ and the union of their sides that are contained in $[a_i,b_i]$ contain the $1/m$-middle part of $[a_i,b_i]$. Assume, moreover, that their sides intersect in a segment greater than $1/8 i$ (recall that the eighth group $\Gamma_i$ has defining relation length $i$). 
    Triangle $\Delta _1$ can have long semidigons at the corners $a_i, g^{j_1}a_i$. If we take points $\bar a_i$ (resp. $\hat a_i$) on $[a_i,b_i]$ (resp. $[g^{j_1}a_i, b_i]$) at distance at most $L\ell _i/m$ from $a_i$ (resp. $g^{j_1}a_i$), then they are outside of $1/m$-middle part of $[a_i,b_i]$ (resp. $[g^{j_1}a_i, b_i]$) and $[\bar a_i,\hat a_i]<i$. We can in addition take these points to be vertices of the divisors of these corner semidigons.   Since $[\bar a_i,\hat a_i]$ in $\{\bar a_i, \hat a_i, b_i\}$ is shorter than $i$, it cannot be more than four cells  on the path along the side of $\Delta _1$ joining these divisors  by Lemma \ref{LongSegments}. 
    We now take points $\bar c_i$ (resp. $\hat c_i$) on $[a_i,b_i]$ (resp. $[g^{j_1}a_i, b_i]$) at distance $2i$  from $\bar a_i$ (resp. $\hat a_i$), then they are outside of $1/(m-1)$-middle part of $[a_i,b_i]$ (resp. $[g^{j_1}a_i, b_i]$) and on the sides of the semidigon originating from the vertex $b_i$ of $\Delta _1$.
    
    Then  the $1/(m-1)$-middle part of $[a_i,b_i]$ is a side of the double semidigon $sDD$, obtained by merging the two triangles or rather the opposite oriented semidigons (within the $1/(m-1)$-middle part of $[a_i,b_i]$).  
    
    Consider now some geodesic $g^{j_3}[a_i,b_i]$, take points $\hat a_i, \hat b_i$ at distance $L\ell _i/(m-1)$ from its endpoints and points $\bar a_i,\bar b_i$  at distance $L\ell _i/(m-1)$ from endpoints of $[a_i,b_i]$. Consider geodesic triangles defined by $\{\bar a_i, \hat a_i,  b_i\}$ and $\{ a_i, \hat b_i, \bar b_i\}$. Since the side $[\bar a_i,\hat a_i]$ in $\{\bar a_i, \hat a_i, b_i\}$ is shorter than $i$, it cannot be more than three cells all together on the legs from $\bar a_i, \hat a_i$ or touching this side of the simple triangle by Lemma \ref{LongSegments}.  Denote corresponding vertices of the simple triangle by $\bar c_i, \hat c_i$ and the third vertex by $c_i$.  If $|[\bar c_i,\hat c_i]|<i/8$, then by Lemma \ref{short_geod_tr}, $\{\bar c_i, \hat c_i, c_i\}$ is a semidigon or degenerates.  Otherwise, we take a point $\hat d_i$ on $[\hat a_i,b_i]$ at distance $2i$ from $\hat a_i$. This $\hat d_i$ is still outside of $1/(m-2)$-middle part of $[g^{j_3}a_i, b_i]$ because $2i<\frac{L\ell _i}{(m-1)m}<\frac{L\ell _i}{(m-1)(m-2)},$ so $\hat d_i$ is at the distance less than $\frac{L\ell _i}{m-2}$ from $g^{j_3}a_i.$ Similarly we take $\bar d_i$ on the geodesic $[a_i,b_i]$ at distance $2i$ from $\bar a_i$. Then points $\bar d_i,\hat d_i$ are on the sides of the simidigon originating from $c_i$ that is a part of $\{\bar c_i,\hat c_i,c_i\}$.  So we can assume that $\{\bar d_i, \hat d_i, b_i\}$ consists of a semidigon followed by digons and bridges between them.
    
    But the $\frac{1}{m-1}$-middle part of $[a_i,b_i]$ is a  side of the double semidigon $sDD$. Therefore all the cells of $\{\bar d_i, \hat d_i,  b_i\}$ that touch the $\frac{1}{(m-2)}$-middle part of $[g^{j_3}a_i,b_i]$ will merge into $sDD$. And $sDD$  covers the $\frac{1}{m-2}$-middle part of $[g^{j_3}a_i,b_i]$. Similarly  it covers the $\frac{1}{m-2}$-middle part of $[g^{j_3}a_i,g^{j_3}b_i]$ and of all geodesics  $g^{r}[a_i,b_i]$ for $1\leq r\leq K$.
\begin{figure}[ht!]
\centering
\includegraphics[width=.7\textwidth]{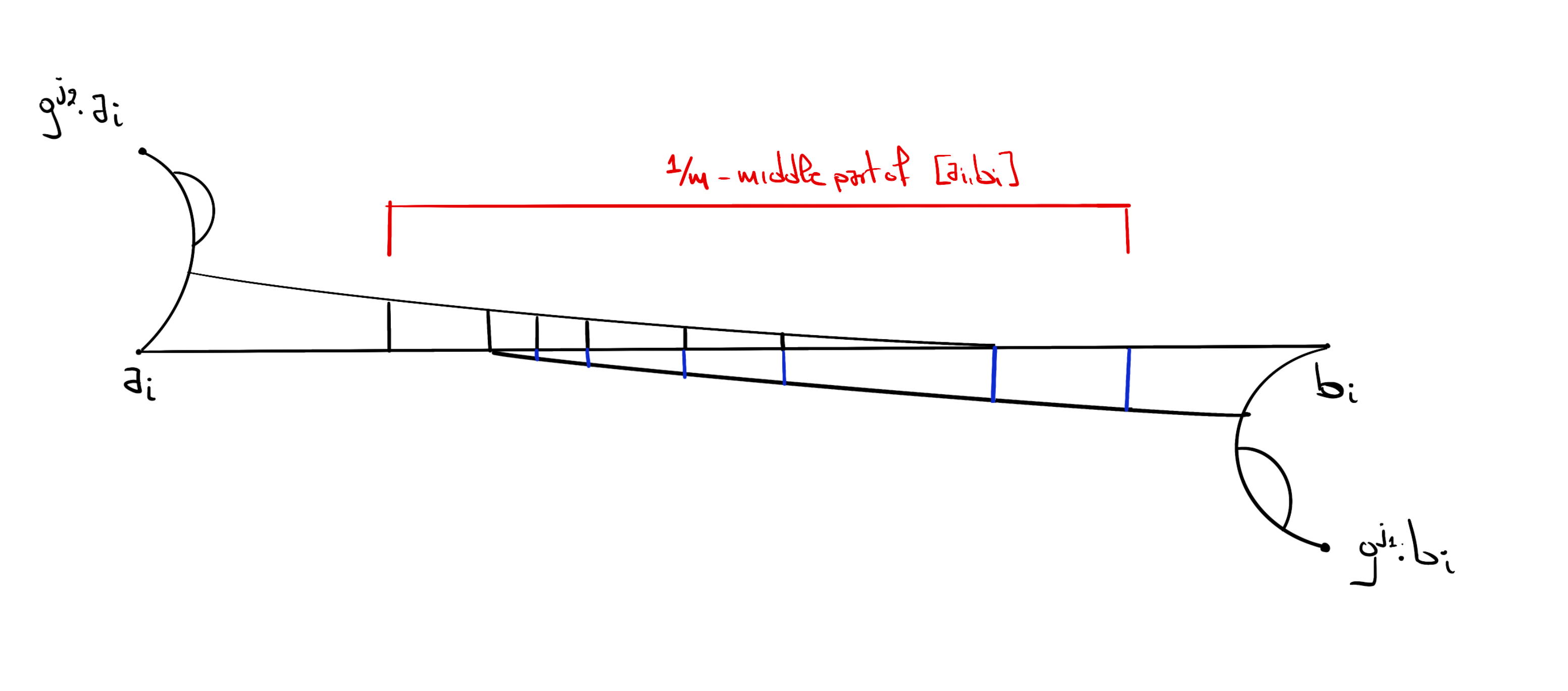}
\caption{Two simple triangles intersecting in the middle part.}
\label{SingleLayerCase1}
\end{figure}

    \item[Case 2] Suppose for no $j_1, j_2\leq K$ the situation in Case 1 happens. Assume moreover, that for some $j\leq K$ the geodesic triangle defined by $\{a_i, g^ja_i, b_i\}$ (or symmetrically $\{a_i, g^jb_i, b_i\}$) has a simple geodesic triangle $\Delta$ and one of its sides contains some long segment, i.e. $>1/8i$, of the $1/m$-middle part of  $[a_i,b_i]$. 
  
  In this case   the $1/m$-middle part of $[a_i,b_i]$ is a union of a side of a semidigon (denote a semidigon with the longest side by $sD$) and a geodesic, and for the geodesic we have the situation as in Proposition \ref{SingleLayer}. Like in the previous case,  we consider some geodesic $g^{j_3}[a_i,b_i]$ and  take similarly points $\hat a_i, \hat b_i,\hat d_i$  and points $\bar a_i,\bar b_i, \bar d_i$. Then  the geodesic triangle defined by $\{\bar d_i, \hat d_i,  b_i\}$ consists of a semidigon followed by digons and bridges between them. Therefore all the cells of $\{\bar d_i, \hat d_i,  b_i\}$ that touch the $\frac{1}{(m-2)}$-middle part of $[g^{j_3}a_i,b_i]$ will merge into $sD$ and finitely many digons. Then the $1/m$-middle part of $[a_i,b_i]$ together with the semidigon $sD$ and finitely many digons, cover the $\frac{1}{m-2}$-middle parts of all geodesics.
\begin{figure}[ht!]
\centering
\includegraphics[width=.7\textwidth]{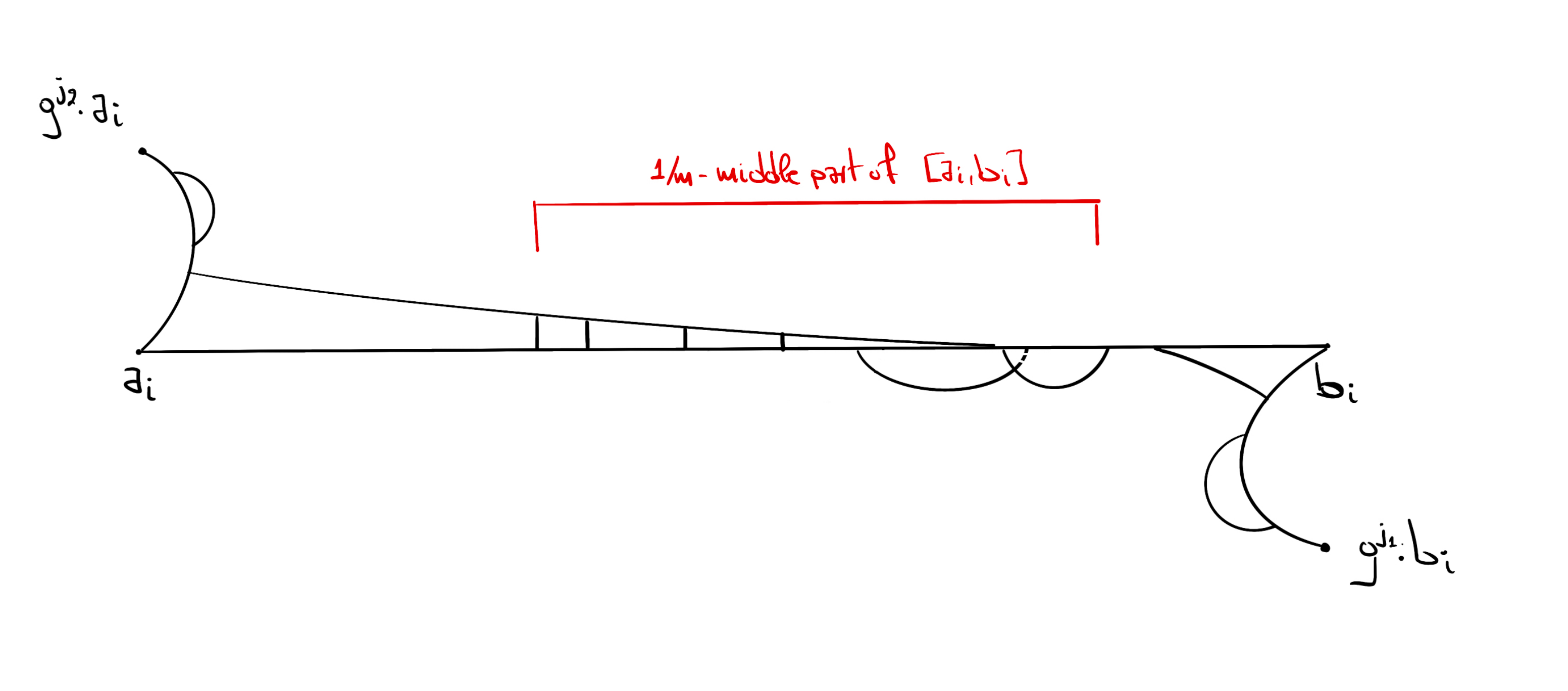}
\caption{A simple triangle intersecting with a digon in the middle part.}
\label{SingleLayerCase2}
\end{figure}
    \item[Case 3] If none of the above cases occur, then, similarly, the $1/m$-middle part of $[a_i,b_i]$ together with some digons cover the $\frac{1}{m-2}$ parts of all geodesics.  
        \begin{figure}[ht!]
\centering
\includegraphics[width=.7\textwidth]{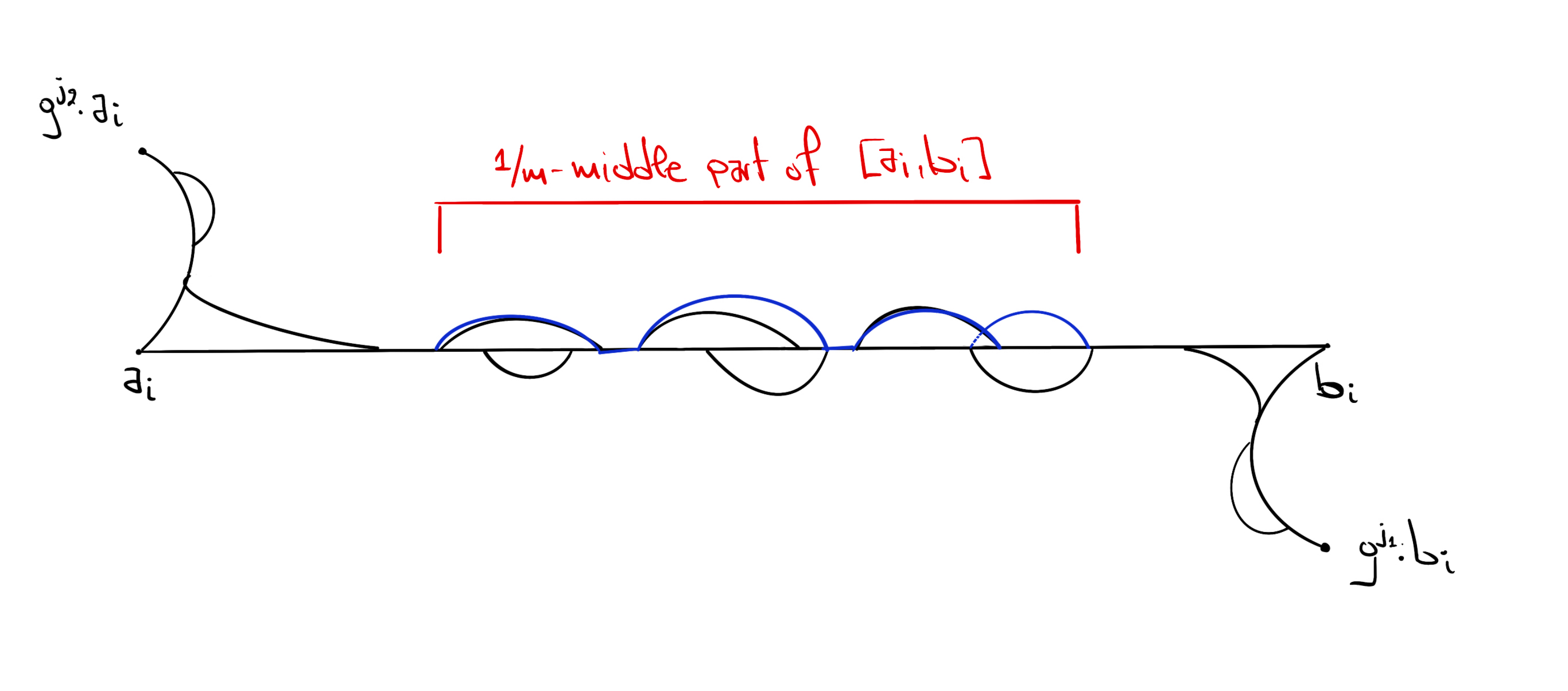}
\caption{Digons intersecting in the middle part.}
\label{SingleLayerCase3}
\end{figure}
\end{itemize}
\end{proof}

\begin{remark}
In the above lemma for the sake of clarity we used powers of the same element $\{g, g^2, \ldots, g^K\}$, but the same proof works for any $K$ elements $\{g_1, g_2, \ldots, g_K\}$. We will use the more general version in Lemma \ref{SegmentStabilizers}.
\end{remark}

We are now in the position to prove the main result of this subsection. 

\begin{lemma}\label{TripodStabilizers}
Let $a,b,c$ be three points in $T$ forming a nontrivial tripod. Let $g\in G$ be an element fixing the tripod. Then $g$ belongs to the stable kernel of the sequence $(h_i)_{i<\omega}:G\rightarrow \Gamma_i$. 
\end{lemma}
\begin{proof}
We observe that since $g$ fixes the tripod, any power of $g$ fixes the tripod. We take $K$ large enough in particular strictly greater than the number of potential geodesics in a single layer configuration of length $8i$. Let $d$ be the middle point of the tripod and $L=max\{|[a,b]|,|[b,c]|,|[c,a]|\}$, $l=min\{|[a,d]|,|[b,d]|,|[c,d]|\}$. Let $m$ be such that $\frac{L}{m} \ll l$. We use Lemma \ref{ExtendedSingleLayer} to choose $n(K,m)$ such that for any approximating sequences $(a_i)_{i<\omega}, (b_i)_{i<\omega},$ $(c_i)_{i<\omega}$, if $i>n(K,m)$, then the segments $[a_i,b_i], [b_i,c_i], [a_i,c_i]$ admit $1/m$ single layer configurations with respect to $\{g,g^2,\ldots, g^K\}$. We recall that the action of $G$ on the limiting sequence is obtained through the morphisms $(h_i)_{i<\omega}:G\rightarrow \Gamma_i$. We will denote this action by $g.x_i$ for any $g\in G$ and $x_i$ in $Cay(\Gamma_i, \Sigma_{\Gamma_i})$.  

By Lemma \ref{MinDistance}, for each $i<\omega$, there exist at most two points, say $e_{i1}, e_{i2}$ in $[a_i,b_i]$ minimizing the distance to $c_i$. We fix $i>n(K,m)$ and we consider the geodesics between $g^j.e_{i1}$ and $g^j.c_i$, for $j\leq K$. By Lemma \ref{ExtendedSingleLayer} some middle part (we omit the ratio as it is not important) of these geodesics is covered by the union of $[e_{i1},c_i]$ with finitely many digons, a semidigon or a double semidigon. In particular there exist at most two points in this union, say $m_{i1}, m_{i2}$ such that any geodesic in this family passes through one of them.  Similarly, for the geodesics between $g^j.e_{i2}$ and $g^j.c_i$, for $j\leq K$. 

We focus on the possible translates of $e_{i1}$ by $\{g, g^2, \ldots, g^K\}$ and we will show, using that these translates are distance minimizing for the corresponding $g^j.[a_i,b_i]$ and $g^j.c_i$, that for all $i>n(K,m)$, there are strictly less than $K$ such points. In particular we will show that the number of such points is bounded by four times the number of possible geodesics in a single layer configuration in a ball of radius $8i$ centered at $e_{i1}$. 

\begin{figure}[ht!]
\centering
\includegraphics[width=.7\textwidth]{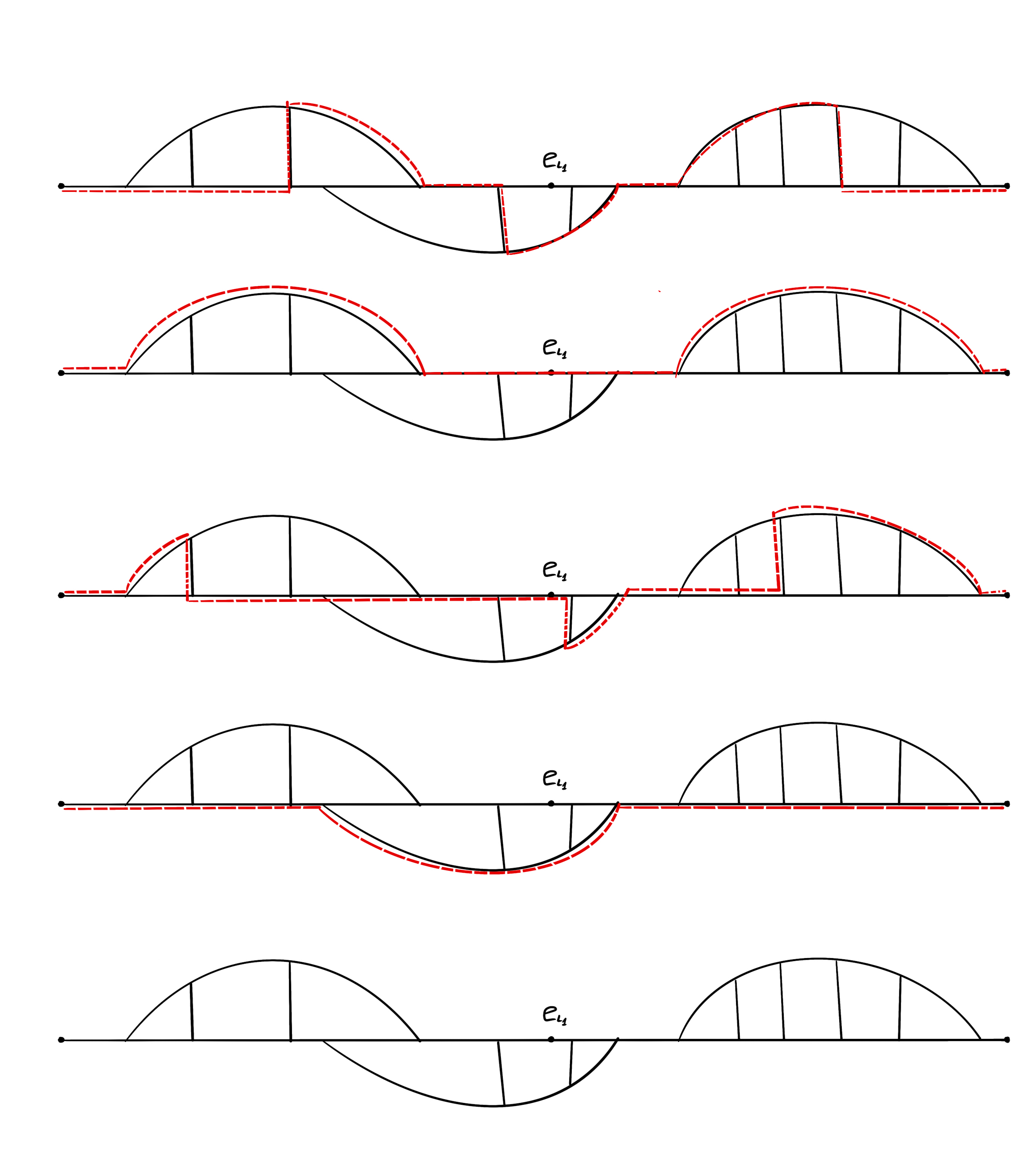}
\caption{A sample of possible geodesics in a single layer configuration within a ball of radius $8i$ centered at $e_{i1}$.}
\label{PossibleGeodesics}
\end{figure}

Indeed, by standard arguments in hyperbolic spaces, since the $g^j$'s do not move the endpoints of the tripod much relative to the length of the edges of the tripod, the translates of $e_{i1}$ lie within a ball of radius $8i$ (recall that the hyperbolicity constant of $\Gamma_i$ is $i$) centered at $e_{i1}$. We consider the part of the single layer configuration for $[a_i,b_i]$ with respect to $\{g, g^2, \ldots, g^K\}$ that lies within this ball of radius $8i$. There are boundedly many digons/semidigons/double semidigons and consequently boundedly many possible geodesic segments. For each such segment there exist at most two points minimizing the distance to $m_{i1}$ and symmetrically two points minimizing the distance to $m_{i2}$. But then these points are among the points that minimize the distance to $g^j.c_i$. Hence four times the number of geodesics (which is bounded independent of $i$) gives an absolute upper bound for the translates of $e_{i1}$ under $\{g, g^2, \ldots, g^K\}$.   Therefore there are two distinct $1\leq j_1,j_2\leq K$ such that $g^{j_1}.e_{i1}=g^{j_2}.e_{i_1}$. Since groups $\Gamma _i$ are torsion free, this implies that $g$ is in the stable kernel.

\end{proof}

\subsection{Segment stabilizers}\label{segment1}  
We continue the study of the limit action under the same assumptions as in the previous subsection. In this subsection we prove the second technical requirement for being able to apply Rips machine. Note that if a subgroup $H$ of $G$ stabilizes (setwise) a segment in the limit tree, then it has an index $2$ subgroup $H_0$ that fixes (pointwise) the segment. As a matter of fact if $g$ belongs to $H$, then $g^2$ belongs to $H_0$. 



\begin{lemma}\label{SegmentStabilizers}
Let $[a,b]$ be a segment in the limit tree $T$. Let $H$ be a subgroup that stabilizes $[a,b]$ (setwise). Then the derived subgroup $[H,H]$, is contained in the stable kernel of the sequence $(h_i)_{i<\omega}:G\rightarrow\Gamma_i$.
\end{lemma} 
\begin{proof} 
First observe that it is enough to prove the result for the derived subgroup $[H_0,H_0]$ of $H_0$, the index $2$ subgroup of $H$ that fixes $[a_,b]$ pointwise. Indeed, for any element of $g\in H$, we have $g^2\in H_0$, and, since the $\Gamma_i$ are CSA, if $[g_1^2,g_2^2]$ is in the stable kernel of $(h_i)_{i<\omega}$, then $[g_1,g_2]$ is in the stable kernel as well.

It is enough to show that for any finite subset $Q$ of commutators of $H_0$ of the form $[g_1,g_2^j]$ there exists $n(Q)$ such that if $i>n(Q)$, there exists a point in $Cay(\Gamma_i, \Sigma_{\Gamma_i})$ whose orbit size under $Q$ is bounded (independent of $i$ or $Q$). Indeed, if we choose $j$ large enough (greater than the absolute bound), then, as $i$ goes to infinity, for some $j_0$ we will have that $[g_1,g_2^{j_0}]$ fixes a point. Hence it must belong to the stable kernel, consequently $[g_1,g_2]$ belongs to the stable kernel, by the same argument as in the beginning of the proof.    

In order to prove the claim we fix a finite set of commutators $Q$ of the above form  and we consider for every commutator $[g_1,g_2^j]$ in $Q$, the elements $g_1, g_2^j, g_1g_2^j, g_1g_2^jg_1^{-1}, [g_1,g_2^j]$ for $1\leq j\leq p$ and their inverses. We call this finite subset of elements $\mathcal{C}(Q)$. 
 
We choose $m\gg 10$ so that the $1/m$ middle part of $[a_i,b_i]$ contains the intersection of $[a_i,b_i]$ with a ball of radius $16i$ centered at the middle point of $[a_i,b_i]$. We can use Lemma \ref{ExtendedSingleLayer} to find $n(|\mathcal{C}(Q)|,m)$ so that for all $i>n(|\mathcal{C}(Q)|,m)$ the segment $[a_i,b_i]$ admits an $1/m$ single layer configuration with respect to $\mathcal{C}(Q)$. 

Now consider the structure of the single layer configuration, we may assume that a digon, a semidigon, or a double semidigon exists near  the midpoint $c_i$ of $[a_i,b_i]$, otherwise if all points $g_2^{-1}.c_i, g_1^{-1}g_2^{-1}.c_i, g_2g_1^{-1}g_2^{-1}.c_i$ and $[g_1,g_2].c_i$ are points in $[a_i,b_i]$ we can easily see that the translations $g_1, g_2$ commute, hence $c_i$ is fixed.

Consider a division point $q_i$ in the middle part of $[a_i,b_i]$. Then a translate of $q_i$ by any element of $Q$ must be a division point as well. Since, by standard arguments, any commutator in $Q$  moves any point in the middle part within a ball of radius $16i$ and since we can only have boundedly many division points in such a ball we get the bound we want. 

\end{proof}

We can now prove the superstability of the faithful action $G/\underrightarrow{Ker}(h_i)$ on the limit tree $T$. Recall, that an action on a real tree is {\em superstable} if whenever, $J\subseteq I$ for segments $I, J$, we have $Fix(J)\neq Fix(I)$, where $Fix(I)$ denotes the pointwise stabilizer of $I$, then $Fix(I)$ is trivial.

\begin{corollary}\label{SuperstableAction}
The action of $G/\underrightarrow{Ker}(h_i)$ on $T$ is superstable.
\end{corollary}
\begin{proof}
Suppose $J\subseteq I$ be segments in $T$ such that $Fix(J)\neq Fix(I)$. In particular, there exists $g\in Fix(J)$ and a point, $x$ in $I\setminus J$ such that $g.x$ is not in $I$. Now, consider any element $h\in Fix(I)$. Since $g$ and $h$ fix the same segment, by Lemma \ref{SegmentStabilizers}, they commute and we have $h.(g.x)=g.(h.x)=g.x$, hence $h$ fixes a tripod and, by Lemma \ref{TripodStabilizers}, it must be trivial.

\end{proof}

\subsection{JSJ decompositions and approximations}
In this subsection we show that all one-ended groups obtained as limits of sequences satisfying the assumptions of the previous subsections admit an abelian JSJ decomposition. An  abelian JSJ decomposition of a group $L$ is a decomposition (in terms of Bass-Serre theory) that, in some sense, encodes all possible abelian decompositions of $L$. It is a splitting where the vertices are either rigid, abelian or QH. Vertices which are QH bear a vertex group which is the fundamental group of a surface with boundary, in addition each incident edge group is a boundary subgroup, i.e. a subgroup of the fundamental group of a connected component of the boundary of the surface. The interested reader can find more details in the excellent book \cite{GuirJSJ}. The crucial fact of JSJ decompositions, which we will later use, is that one can ``read" in them the automorphisms used in the shortening argument. These automorphisms are called modular automorphisms. We give a formal definition of the modular group of automorphims in this subsection. Originally such results have been obtained in \cite{Sel1} for limit groups and generalized to $\Gamma$-limit groups for $\Gamma$ hyperbolic (possibly with torsion) in \cite{ReWe} (for another  generalization of the result in  \cite{Sel1} see \cite{GH}). For preliminaries on Bass-Serre theory we refer the reader to \cite{SerreTrees}.

Let $\mathcal{G}_8$ be the family of eighth groups with defining relators of the same length. Let $\Gamma_i$ be an eighth group with defining relators of length $i$. We consider a stable sequence of morphisms $(h_i)_{i<\omega}:G\rightarrow \Gamma_i$ from a finitely generated group $G$ such that $|h_i|$ dominates $i$, i.e. $\frac{i}{|h_i|}$ goes to $0$ as $i$ goes to infinity.

We call a group obtained as a quotient, $G/\underrightarrow{Ker}(h_i)$, {\em a $\mathcal{G}_8$-limit group} and the canonical map $\phi:G\rightarrow L:=G/\underrightarrow{Ker}(h_i)$ {\em the  limit map}. We show that one-ended $\mathcal{G}_8$ groups admit canonical JSJ abelian decompositions. In addition, since $\mathcal{G}_8$-limit groups may not be finitely presented (or Equationally Noetherian) we cannot claim that $h_i$ eventually factor through the limit map $\phi$. In particular, we cannot straightforwardly transfer the shortening of the limit action to shorten the morphisms of the limiting sequence. To circumvent this difficulty we will approximate the JSJ decomposition of $L$ by ``similar" decompositions of finitely presented groups that map onto $L$.     

We start with some easy observations. Recall that a group $G$ is CSA if maximal abelian subgroups are malnormal, i.e. $A\cap A^g\neq \{1\}$ only if $g\in A$, for any maximal abelian subgroup $A$ of $G$. 

\begin{lemma}\label{CSALImit}
Let $L$ be $\mathcal{G}_8$-limit group. Then $L$ is torsion-free and CSA.
\end{lemma}
\begin{proof}
Let $\phi:G\twoheadrightarrow L$ be the canonical epimorphism from $G$ to $L$. Let $\gamma$ be an element of order $n$ in $L$. For any pre-image $g$ of $\gamma$ we have $h_i(g^n)$ is eventually trivial, and since $\Gamma_i$ are all torsion-free we get that $h_i(g)$ is eventually trivial, hence $\gamma$ is. 

We next show that $L$ is commutative transitive. Let $[\gamma_1,\gamma_2]=1, [\gamma_2,\gamma_3]=1$ and $\gamma_2\neq 1$. We will show that $[\gamma_1,\gamma_3]=1$. Indeed, let $g_i$ be some pre-image of $\gamma_i$, for $i\leq 3$, in $G$. By the assumptions we have that $h_i([g_1, g_2]), h_i([g_2,g_3])$ are eventually trivial, while $h_i(g_2)$ is not. Then, for all large $i$, by the commutative transitivity of $\Gamma_i$, we get that $h_i([g_1,g_3])$ is eventually trivial. Hence $\gamma_1$ commutes with $\gamma_3$. 

Finally, having commutative transitivity, in order to show that $L$ is CSA it is enough to show that centralizers of non-trivial elements are self-normalizing. Indeed, let $\gamma$ be a non-trivial element and $C_L(\gamma)$ be its centralizer in $L$. Consider an element $\alpha$ that normalizes $C_L(\gamma)$. Then for any pre-images $g$ of $\gamma$ and $a$ of $\alpha$, we have that $h_i([a^{-1}ga, g])$ is eventually trivial, and since in any torsion-free hyperbolic group if $g^a$ commutes with $g$, then $a$ commutes with $g$, we get that $h_i([a,g])$ is eventually trivial. Therefore, $\alpha$ is in the centralizer of $\gamma$ as we wanted. 
\end{proof}
Lemma \ref{CSALImit} can also be seen as a consequence of the following fact: since $L$ is fully residually $\{\Gamma _i\}$, the common universal theory of the $\Gamma _i$ is contained in the universal theory of $L$. Since $\Gamma _i$ are torsion-free and CSA and these properties can be expressed by a universal sentence, so is $L$. 

We next record a lemma whose proof is identical to \cite[Lemma 2.1]{Sel1}. 

\begin{lemma}\label{AbelianElliptic}
Let $L$ be a $\mathcal{G}_8$-limit group. Let $M$ be a maximal non-cyclic abelian subgroup. Then: 
\begin{itemize}
    \item $M$ is elliptic in any abelian amalgamated free product of $L$.
    \item If $L$ is an HNN extension $A*_C$, then either $M$ can be conjugated in $A$ or $L$ admits a splitting $A*_CB$ where $M$ can be conjugated in $B$. 
\end{itemize}  
\end{lemma}

The above lemmas allow us to deduce the following theorem exactly as in \cite[Section 4.4]{ReWe}.

\begin{theorem}[JSJ decomposition]\label{JSJLimit}
Let $L$ be a one-ended $\mathcal{G}_8$-limit group. Then $L$ admits a  JSJ decomposition over abelian subgroups.
\end{theorem}

We now define the modular automorphism group of an one-ended $\mathcal{G}_8$-limit group. We quickly note that when an automorphism of a vertex group (of a graph of groups) restricts to conjugation on its edge groups then there is a natural way to extend it to an automorphism of the whole group. We call such an extension the {\em standard extension}. 

\begin{definition}\label{ModAutos}
Let $L$ be a one-ended $\mathcal{G}_8$-limit group and $\Lambda$ an abelian graph of groups decomposition of $L$. Then the group of modular automorphisms with respect to $\Lambda$, $Mod_{\Lambda}(L)\leq Aut(L)$, is generated by the following automorphisms:
\begin{itemize}
    \item Inner automorphisms of $L$.
    \item Dehn twists along edges by elements that centralize an edge group.
    \item Standard extensions of automorphisms of QH subgroups - these automorphisms are obtained by surface homeomorphisms fixing all the boundary components.
    \item Standard extensions of automorphisms of abelian vertex groups that are the identity on their peripheral subgroups.
\end{itemize}
\end{definition}

As a corollary of the properties of the abelian JSJ decomposition we get (see \cite[Lemma 4.16]{ReWe}). 

\begin{corollary}\label{ModularAutos}
Let $L$ be a one-ended $\mathcal{G}_8$-limit group and $\Lambda$ an abelian graph of groups decomposition of $L$. Then $Mod_{\Lambda}(L)\leq Mod_{JSJ}(L)$, for any JSJ decomposition of $L$. 
\end{corollary}

Finally, we record that a JSJ decomposition of $L$ can be approximated by a sequence of finitely presented groups whose graph of groups approximate the JSJ decomposition as described below (see Theorem \ref{ApproxJSJ}). For the notion of a morphism between graph of groups (of possibly non-isomorphic groups) we refer the reader to \cite[Section 4.3]{ReWe}. We only give a brief account so notation makes sense. Recall that A morphism between graphs of groups $f:\mathbb{A}\rightarrow\mathbb{B}$, includes a graph morphism between the underlying graphs, vertex group (homo)morphisms $\psi_v^f$ and edge group (homo)morphisms $\psi_e^f$ for every vertex and edge of $\mathbb{A}$. We also note that whenever  $f:\mathbb{A}\rightarrow\mathbb{B}$ is a morphism between graphs of groups, there is an induced (homo)morphim $f_*:\pi_1(\mathbb{A})\rightarrow\pi_1(\mathbb{B})$ and we say $f$ is surjective if $f_*$ is.
Finally, we record that a JSJ decomposition of $L$ can be approximated by a sequence of finitely presented groups whose graph of groups approximate the JSJ decomposition as described below (see Theorem \ref{ApproxJSJ}). For the notion of a morphism between graph of groups (of possibly non-isomorphic groups) we refer the reader to \cite[Section 4.3]{ReWe}. We only give a brief account so notation makes sense. Recall that $\phi$ is the canonical map, $\phi:G\rightarrow L:=G/\underrightarrow{Ker}(h_i)$. A morphism between graphs of groups $f:\mathbb{A}\rightarrow\mathbb{B}$, includes a graph morphism between the underlying graphs, vertex group (homo)morphisms $\psi_v^f$ and edge group (homo)morphisms $\psi_e^f$ for every vertex and edge of $\mathbb{A}$. We also note that whenever  $f:\mathbb{A}\rightarrow\mathbb{B}$ is a morphism between graphs of groups, there is an induced (homo)morphim $f_*:\pi_1(\mathbb{A})\rightarrow\pi_1(\mathbb{B})$ and we say $f$ is surjective if $f_*$ is.

\begin{theorem}[JSJ approximation] \label{ApproxJSJ}
Let $\mathbb{A}$ be an abelian JSJ decomposition of an one-ended $\mathcal{G}_8$-limit group. Then there exists a sequence of decompositions $(\mathbb{A}^i)_{i<\omega}$ with finitely presented fundamental groups $W_i:=\pi_1(\mathbb{A}^i)$. Moreover there exist graph of groups surjective morphisms $f^i:\mathbb{A}^i\twoheadrightarrow \mathbb{A}^{i+1}$,  $\phi^i:\mathbb{A}^i\twoheadrightarrow\mathbb{A}$ and an epimorphism $\gamma:G\twoheadrightarrow W_0$ such that the following hold:
\begin{enumerate}
    \item $\phi^i$ is a graph isomorphism for all $i<\omega$.
    \item $\phi=\phi_*^0\circ\gamma$.
    \item $\phi^i=\phi^{i+1}\circ f^i$, for all $i<\omega$.
    \item $\underrightarrow{Ker}(h_i)=\bigcup_{k=1}^\infty ker(f_*^k\circ f_*^{k-1}\circ\ldots\circ f_*^1\circ f^0_*\circ\gamma)$.
    \item If $G_v$ is QH in $\mathbb{A}$, then $\psi_v^{\phi^i}:G_{v}^i\rightarrow G_v$ is an isomorphism for all $i<\omega$.
    \item If $G_v$ is of abelian type in $\mathbb{A}$, then $\psi_v^{\phi^i}:G_{v}^i\rightarrow G_v$ is injective for all $i<\omega$. 
    \item The map $\psi_e^{\phi^i}:G_e^i\rightarrow G_e$ is injective for all $i<\omega$. 
    \item For any vertex $v\in \mathbb{A}$ we have $\bigcup \psi_v^{\phi^i}(G_v^i)=G_v$. 
    \item For any edge $e\in \mathbb{A}$ we have $\bigcup \psi_e^{\phi^i}(G_e^i)=G_e$.
\end{enumerate}
\end{theorem}

\begin{figure}[ht!]
\centering
\includegraphics[width=1.\textwidth]{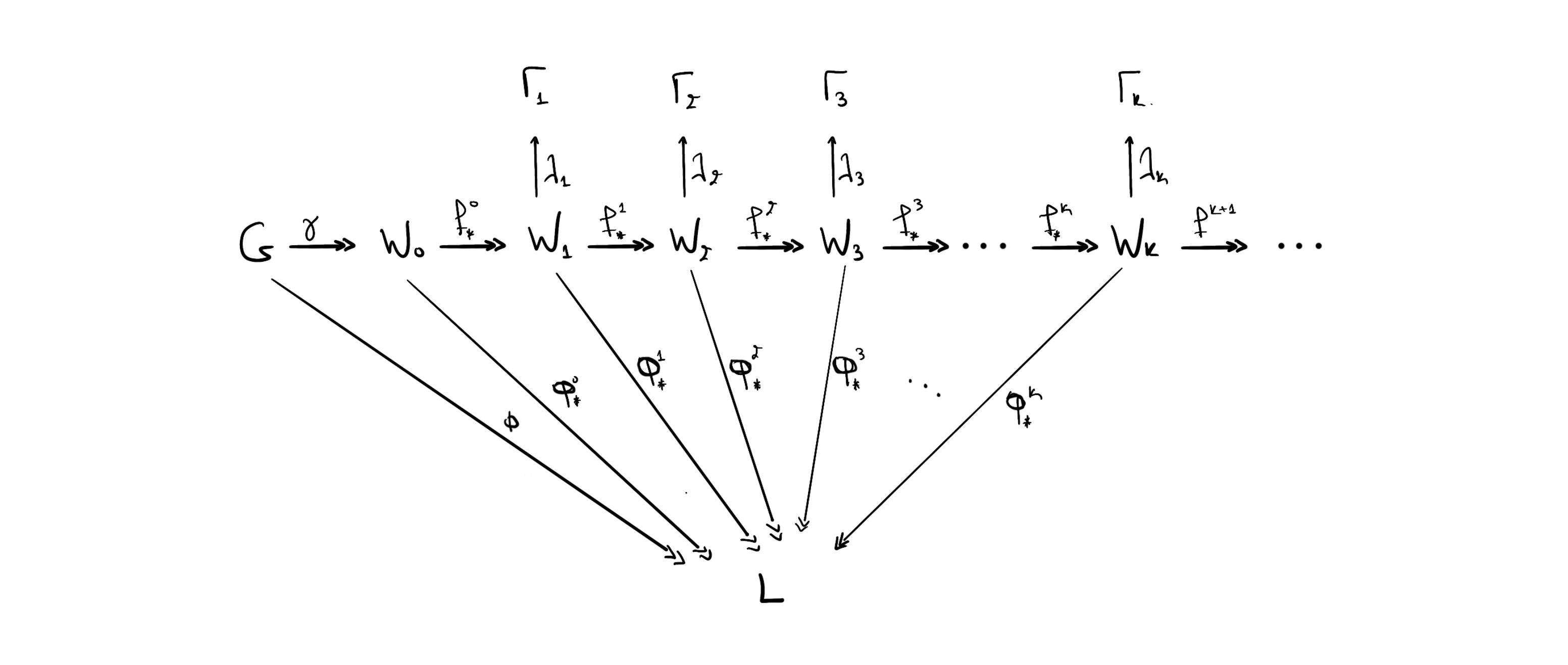}
\caption{An approximation diagram of a JSJ decomposition of $L$.}
\label{ApproximateJSJ}
\end{figure}

\subsection{Short solutions}\label{6.6} 
Let $V(\bar{x})=1$ be a system of equations, $G_V:=\F(\bar{x})/ ncl(V(\bar{x}))$. There is a well-known correspondence between solutions of  $V(\bar{x})=1$ in $\Gamma$ and  morphisms $h:G_V\rightarrow \Gamma$. It will be convenient to adapt to the latter point of view. Towards this we define. 


\begin{definition}\label{nonfree}
Let $G, H$ be finitely generated groups. 
Let $\pi:\mathbb{F}_n\twoheadrightarrow H:=\mathbb{F}_n/\langle\langle R\rangle\rangle$ be the canonical quotient map. We call a morphism $h :G\rightarrow H$ {\em (non-) free} if $h$ does (not) factor through $\pi$, i.e. there does (not) exists $h':G\rightarrow \mathbb F_n$ such that $h=\pi\circ h'.$ 
\end{definition}

A non-free morphim $h:G\rightarrow H$ is {\em short} if $|h|_{S_G}\leq |h'|_{S_G}$, for any non-free morphism  $h':G\rightarrow H$. 

We want to show that for any sequence of short non-free morphisms $h_i:G_V\rightarrow \Gamma_i$, where $\Gamma_i$ is an eighth group with relations of length $i$, there exists a bound $K$ such that $|h_i|\leq Ki$.  


We remark that non-freeness is invariant under post composing by inner automorphisms of $H$. If $h :G\rightarrow H$ is non-free, then for any $b\in H$, $Conj(b)\circ h $ is non-free. Indeed, suppose $Conj(b)\circ h$, factors through $\pi$, i.e. $Conj(b)\circ h=\pi\circ h'$, for some $h':G\rightarrow\mathbb{F}_n$. Then $h=Conj(b^{-1})\circ\pi\circ h'=\pi\circ Conj(\gamma)\circ h'$ for some $\gamma$ in the pre-image of $b^{-1}$ under $\pi$.

If $G_V$ is not a limit group, then 
we have an axiom in the theory of a free group  of the form. 
$$\forall\bar{x}(V(\bar{x})=1\rightarrow \bar w_1(\bar{x})=1\lor \bar w_2(\bar{x})=1\lor\ldots\lor \bar w_k(\bar{x})=1)$$
where the limit quotients of $G_V$ are given  by $\bar w_1(\bar{x}), \ldots, \bar w_k(\bar{x})$. 

Now observe that if a morphim $h:G_V\rightarrow \Gamma$ is free, then it must kill one of the $\bar w_i(\bar{x})$, for $i\leq k$. A map that kills at least one of the $\bar w_i$'s is called {\em free-like}. The advantage of free-like morphisms over free morphisms is that we can define them by the formula $\exists\bar{x}(\bar w_1(\bar{x})=1\lor \bar w_2(\bar{x})=1\lor\ldots\lor \bar w_k(\bar{x})=1)$.   Like in the case of non-free morphisms a non-free-like morphism $h:G\rightarrow H$  is called short if $\abs{h}\leq \abs{h'}$, for any other non-free-like morphism $h':G\rightarrow H$. 


We recall that a random group of density $d<1/16$ at length $i$ satisfies $C'(1/8)$, so the hyperbolicity constant  $\delta_i$ is bounded by $i$ with probability $1$ (see \cite[Corollary 3]{Olli1}). 

\begin{remark}
It would be enough, for our purposes, to prove that for any sequence of non-free like short morphisms $(h_i)_{i<\omega}:G_V\rightarrow \Gamma_i$ there exists a bound $K$ such that $|h_i|\leq Ki$. Nevertheless, we will show, once we have such result, how to generalize it to sequences of non-free morphisms. 
\end{remark}

\begin{theorem} \label{limit} 
We fix $d<1/16$ and a system of equations $V(\bar{x})=1$ such that $G_V$ is not a limit group. Let $(\Gamma_i)_{i<\omega}$ be a sequence of random groups of density $d$ at length $i$ and  $(h_i)_{i<\omega}:G_V\rightarrow \Gamma_i$ a sequence of  short non free-like morphisms. 

Then, there is a constant $K$ (that depends only on $V(\bar{x})=1$ and $d$) such that $\abs{h_i}\leq K i$. 
\end{theorem}
\begin{proof} The proof is by contradiction. Suppose there is no such $K$, then by possibly passing to a subsequence (that we still denote by $(h_i)_{i<\omega}$) we can assume that $\abs{h_i}> k_i i$, with $k_i\rightarrow\infty$ as $i\rightarrow\infty$.

We can now apply the arguments in subsection \ref{ActionsRtrees} for the sequence $(h_i)_{i<\omega}$ rescaled by $|h_i|$. After quotienting out the kernel of the action, that can be identified with $\underrightarrow{Ker}(h_i)$, we get a faithful action and a canonical quotient map $\phi:G_V\twoheadrightarrow L$, where $L:=G_V/\underrightarrow{Ker}(h_i)$.   

Recall that $L$ acts minimally and non-trivially on a based $\mathbb R$-tree 
$T$. Moreover by Corollary \ref{SuperstableAction} the action is superstable, hence we can use the main theorem in \cite{GuirGoA}. The latter implies that either $L$ splits as a free product or the action of $L$ on the limit tree can be decomposed, as a graph of actions, in simpler building blocks of simplicial, axial or surface (Seifert) type. We first show how to obtain a contradiction when $L$ is freely indecomposable. At this point we observe that $L$ might not be finitely presented, hence we cannot assume that $h_i$ factors through $\phi$ for all large enough $i$. A remedy for this is given by Theorem \ref{ApproxJSJ}.  

Let $\mathbb{A}$ be an abelian JSJ decomposition of $L$ guaranteed by Corollary \ref{JSJLimit}. Now, by the previously mentioned theorem we can approximate $\mathbb{A}$ using a sequence of splittings $\mathbb{A}^i$. Since each $W_i$ is finitely presented after possibly refining the sequence $(h_i)_{i<\omega}$ we may assume that for each $i<\omega$, $h_i$ factors through $\xi_i$, where $\xi_i=f_*^{i-1}\circ\ldots\circ f_*^1\circ f_*^0\circ\gamma$. In particular, $h_i=\lambda_i\circ\xi_i$ for some $\lambda_i:W_i\rightarrow\Gamma_i$. 

By standard techniques (see \cite[Section 5.2]{ReWe}) the limiting sequence of actions can be shortened. As a matter of fact if $T$ contains either an axial or a surface type component the limit action of $L$ on $T$ can be shortened. We take cases according to what type of components exist in the graph of actions decomposition for this action.
\ \\ \\
{\bf Surface type component:} In this case, we can read the automorphism, used to shorten the action, of the stabilizer of such component in the JSJ decomposition of $L$. Indeed, the surface type vertex group in the graph of actions corresponds to a subsurface in a QH type vertex in the JSJ decomposition of $L$. Let $\alpha$ be such modular automorphism of a QH group in the JSJ decomposition of $L$. By the properties of the approximating sequence, since every QH group in $\mathbb{A}$ corresponds to an isomorphic QH subgroup of $W_i$ in $\mathbb{A}_i$, we deduce that for each $i$ there exists $\alpha_i\in Mod(W_i)$ such that  $\phi_*^i\circ\alpha_i=\alpha\circ\phi_*^i$, and makes $\lambda_i$ shorter (with respect to the generating set $\xi_i(\bar{x})$). To see the latter, suppose $\rho$ is the limit action of $L$ on $(T, \ast)$ and consider the following inequality 
$$|\phi_*^i\circ\alpha_i\circ\xi_i|_\ast=|\alpha\circ\phi_*^i\circ\xi_i|_\ast=|\alpha\circ\phi|_\ast=|\rho\circ \alpha|_\ast<|\rho|_\ast=|\phi|_\ast=|\phi_*^i\circ\xi_i|_\ast$$ (where by $|\rho|_\ast$ we denote the sum of the dispacements of $\ast$ by the fixed set $\phi(S_G)$ of generators) In particular, since $|\lambda_i\circ\alpha_i\circ\xi_i|$ approximates $|\phi_*^i\circ\alpha_i\circ\xi_i|_\ast$ and 
$|\lambda_i\circ\xi_i|$ approximates $|\phi_*^i\circ\xi_i|_\ast$ as $i$ goes to infinity, we have that 
$\hat{h}_i=\lambda_i\circ\alpha_i\circ\xi_i$ is eventually strictly shorter than $h_i$. Since we have chosen $h_i$ to be the shortest non free-like morphism, the sequence  $(\hat{h}_i)_{i<\omega}$ must be eventually free-like. Thus, there exists $\bar w_j(\bar{x})$ that is eventually killed by $(\hat{h}_i)_{i<\omega}$, but then 
we must have for all large enough $i$ that $\phi_*^i(\alpha_i(\xi_i(\bar w_j)))=1$. By the way the $\alpha_i$'s are defined,  $\phi_*^i(\alpha_i(y))=\alpha(\phi_*^i(y))$ for any $y$ in $W_i$, thus we get $\alpha(\phi^i(\xi_i(\bar w_j)))=1$. The latter is impossible, since $\phi_*^i(\xi_i(\bar w_j))\neq 1$ for all $i<\omega$, a contradiction.

\ \\ \\ 
{\bf Axial type components:}
This case is almost identical to the previous one. First observe that since the action on axial components have dense orbits the abelian group of the axial component cannot be cyclic. Hence it is elliptic in the JSJ decomposition of $L$. We actually may assume that it is a vertex group of the JSJ since otherwise we can further refine the JSJ decomposition. Now consider a modular automorphism $\alpha$ , which is a standard extension of the abelian vertex group and shortens the action. In contrast to the surface type case,  the vertex (homo)morphisms that correspond to abelian vertices given by the graph of groups morphisms from  $\mathbb{A}^i$ to $\mathbb{A}$ are not (in principle) epimorphisms, but for $i$ large enough $G_{v_i}$ (the abelian vertex group in $\mathbb{A}^i$ that corresponds to the abelian group of the axial component) admits an automorphism whose standard extension $\alpha_i$ is such that $\phi_*^i\circ\alpha_i=\alpha\circ\phi_*^i$. The rest of the proof works as in the previous case.  
\ \\ \\
Hence we may assume that no axial or surface type component exists. In particular the graph of actions is a just a simplicial splitting of $L$.
\ \\ \\
{\bf Simplicial type component:}
In this case there is not a single automorphism but a sequence that is used to shorten the sequence of the limiting actions.
Let $\mathbb B$ be the graph of action of $L$. Then splittings of $\mathbb B$ either correspond to the splittings along the edges of $\mathbb A$ or along simple closed curves on $QH$ vertices of $\mathbb A$. We can see them in the splitting of $W_i$ as $\mathbb{B}^i$.  Let $\hat\lambda _i:W_i\rightarrow \Gamma _i$ be a homomorphism obtained from shortening $\lambda _i$   by precomposition with an automorphism in $Mod_{\mathbb{B}^i}(W_i)$ and postcomposition with an inner automorphism, we define $\eta _i=\hat\lambda _i\circ\xi_i$.
Since the sequence $(h_i)_{i<\omega}$ consists of the shortest non free-like morphisms, the morphism $\eta _i$, for all large enough $i$, must be free-like. 
Consider first the case where $\mathbb{B}$ is an amalgamated free product. Then $W_i=A_i*_{C_i}B_i$, where $C_0$ is f.g. abelian, $C_0=\langle d_1,\ldots ,d_r\rangle .$ Let $A_0=\langle a_1,\ldots ,a_s\rangle,$  $B_0=\langle b_1,\ldots ,b_q\rangle.$ 
Let $d_1=u(a_1,\ldots ,a_s)$.
 Notice that all elements in all $\bar w_j$ cannot be elliptic with respect to the splitting  $\mathbb{B}^i$  because the images of rigid subgroups  $\mathbb{B}^i$ under $\lambda _i$ and $\hat\lambda _i$ are the same.
Therefore some of the $w_j$ is hyperbolic with respect to the splitting of $W_0$  and $L$. Let $\lambda _i\circ f_*^{i-1}\circ\ldots\circ f_*^1\circ f_*^0(d_1)=
{c_i}^{k_i}$,  where $c_i$ is a generator of a cyclic subgroup in $\Gamma _i$. Let $\alpha _i$ be the Dehn twist of $W_i$ corresponding to the conjugation of $B_i$ by $f_*^{i-1}\circ\ldots\circ f_*^1\circ f_*^0(d_1)$. Then pre-composition of $\lambda _i$ with $\alpha _i$ conjugates images of $B_i$ under $\lambda _i$ in $\Gamma _i$ by ${c_i}^{k_i}.$ 

Then $\hat \lambda _i\circ{\alpha _i}^{r_i}\circ\xi _i$, $r_i\in \mathbb Z$, may be free-like for some values of $r_i$ and not free-like for some other values. They are  non free-like for all sufficiently large absolute values of $r_i$ (different for each $\Gamma _i$),
because for each $\Gamma _i$ we have "big powers property"  \cite[Lemma 2.3]{Olsh1} that implies for all sufficiently large  $|r_i|,$
$\hat \lambda _i\circ{\alpha _i}^{r_i}\circ\xi _i(\bar w_j)\neq 1.$   Let $\{p_i-1\}$ be the sequence of the largest positive numbers such that
$\hat \lambda _i\circ{\alpha _i}^{p_i-1}\circ\xi _i$ are free-like. Then $\{\hat \lambda _i\circ{\alpha _i}^{p_i-1}\circ\xi _i\}$ converges to some limit group $\bar Q$ that is a quotient of $L$ and a quotient of $F/ncl(\bar w_j)$ for some $j\in \{1,\ldots ,k\}$. 
  We also have that $\hat \lambda _i\circ{\alpha _i}^{p_i}\circ\xi _i$ are almost surely non-free-like.  
\ \\ \\
\noindent {\bf Claim:} Both sequences $\{\hat \lambda _i\circ{\alpha _i}^{p_i}\circ\xi _i\}$ and $\{\hat \lambda _i\circ{\alpha _i}^{p_i-1}\circ\xi _i\}$ almost surely converge to $L$.
\ \\ \\
{\bf Proof of Claim:} If $g\in ker \{h_i\}$, then eventually $g\in ker \xi _i$ and, therefore  $g\in ker \hat \lambda _i\circ{\alpha _i}^{p_i}\circ\xi _i$ for large $i$. Therefore, the sequence $\{\hat \lambda _i\circ{\alpha _i}^{p_i}\circ\xi _i\}$ converges to a quotient of $L$. We have to prove that this is not a proper quotient.

Notice that $\hat \lambda _i\circ{\alpha _i}^{p_i}\circ\xi _i$ is not shorter than $ h_i$, because $h_i$ is minimal non-free. We now apply to $\{h_i\}$ instead of shortening argument to simplicial component \cite[Section 5.2.3]{ReWe} the ``lengthening argument" which goes similarly and shows that the limiting tree for the sequence $\{\hat \lambda _i\circ{\alpha _i}^{p_i}\circ\xi _i\}$ has the same simplicial component as the tree for $L$. Splitting of $T$ as a graph of actions  is described in \cite[Th 3.4]{ReWe}. 

The morphisms of the sequence $\{\hat \lambda _i\circ{\alpha _i}^{p_i}\circ\xi _i\}$ are not shorter than for $\{ h_i\}$,
let $\hat \lambda _i\circ{\alpha _i}^{p_i}=\lambda _i\circ\sigma _i^{m_i}$, where $\sigma _i$ is the corresponding Dehn twist.

Let $g\neq 1$ in $L$. Let like in the proof of \cite[Proposition 5.19]{ReWe} $q=a_0,e_1,a_,e_2,\ldots ,a_n$ be a normal form $\mathbb A$-path such that   $g=[q]$, and $\tilde e_1,\ldots ,\tilde e_n$ be the reduced edge path in the Bass-Serre tree $\tilde {\mathbb A}$ from $\tilde v_0$ to $g\tilde v_0.$ Then
$$d_{\mathcal G}(x_0,gx_0)= d_{\tilde v_0}(\bar x_0,p_{\tilde e_1}^{\alpha})+\sum _{k=1}^{n-1} d_{\omega (\tilde e_k)}(p_{\tilde e_k}^{\omega},p_{\tilde e_{k+1}}^{\alpha})+\sum _{k=1}^{n} d_{\tilde e_k}(p_{\tilde e_k}^{\alpha},p_{\tilde e_k}^{\omega})+d_{g\tilde v_0}(p_{\tilde e_n}^{\omega},g\bar x_0).$$
\begin{figure}[ht!]
\centering
\includegraphics[width=1.1\textwidth]{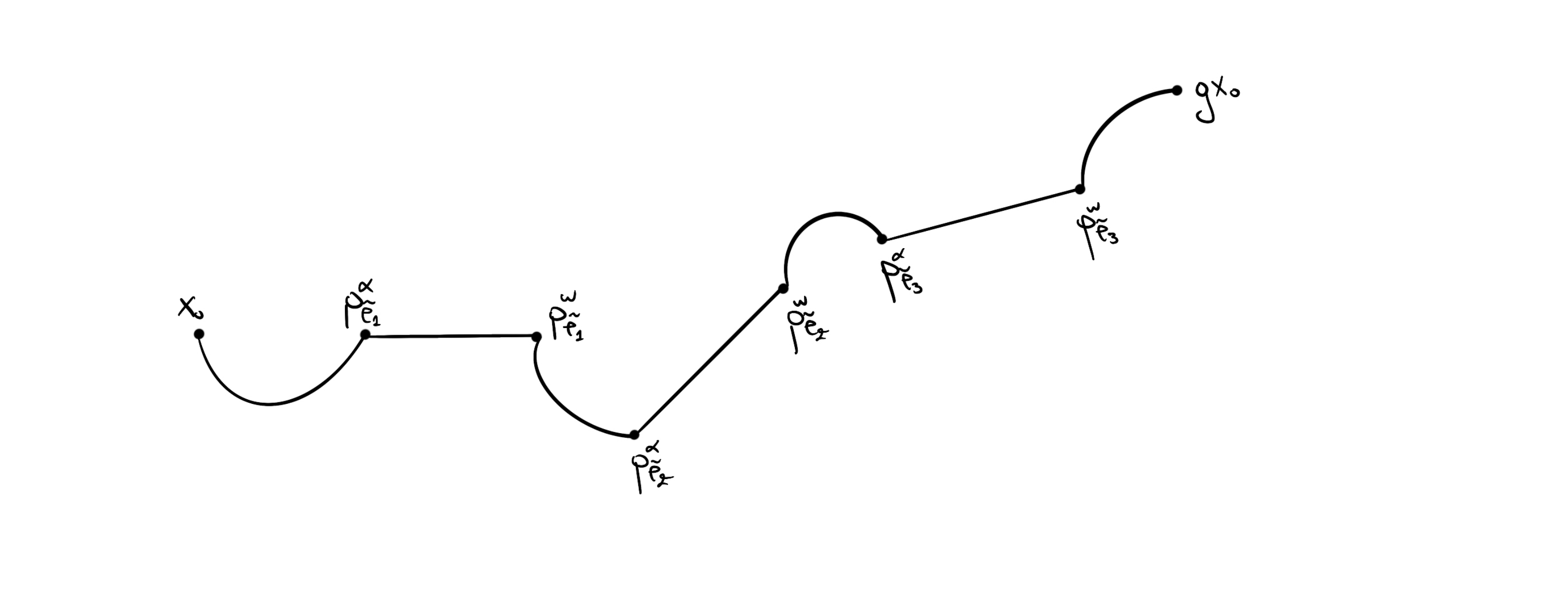}
\caption{The segment $[x_0,gx_0]\subset T$, where $g=[a_0,e_1,\ldots, a_3]$}
\label{path}
\end{figure}

Now for each $k$, let $(p_{k,i}^{\alpha})_{i\in\mathbb N}$ and $(p_{k,i}^{\omega})_{i\in\mathbb N}$ be approximating sequences of $p_{\tilde e_k}^{\alpha}$ and  $p_{\tilde e_k}^{\omega}$  respectively. Recall that $(h_i(1))$ and $( h_i(g))$ are approximating sequences for $x_0$ and $gx_0$ respectively, and $h_i(1)=1.$  

 We have   $$\sum_{g} d_X(1,h_i(g))\leq \sum_{g}d_X(1,h _i\circ\sigma ^{m_i}(g)),$$ where the summation is along all the generators. From the properties of approximation, for any $\epsilon >0$ and large enough $i$, we have
$$d_i(1, h_i(g))\leq d_i(1,p_{1,i}^{\alpha})+\sum _{k=1}^{n-1} d_i(p_{k,i}^{\omega},p_{k+1,i}^{\alpha})+\sum _{k=1}^{n} d_i(p_{k,i}^{\alpha},p_{k,i}^{\omega})+d_i(p_{n,i}^{\omega}, h_i(g))+\epsilon $$ $$ d_i(1,h _i\circ\sigma^{m_i}(g))\geq 
d_i(1,\bar p_{1,i}^{\alpha})+\sum _{k=1}^{n-1} d_i(\bar p_{k,i}^{\omega},\bar p_{k+1,i}^{\alpha})+\sum _{k=1}^{n} d_i(\bar p_{k,i}^{\alpha},\bar p_{k,i}^{\omega})+d_i(\bar p_{n,i}^{\omega}, h_i\circ \sigma ^{m_i}(g))-\epsilon .$$
By 1-4 in the Proof of Proposition 5.19 in \cite{ReWe}, the first, second and fourth terms of the sum in the first line and the sum in the second line are equal.
Moreover, in the third term $d_i(p_{k,i}^{\alpha},p_{k,i}^{\omega})=d_i(\bar p_{k,i}^{\alpha},\bar p_{k,i}^{\omega})$ for all edges $\tilde e_k$ that do not correspond to the amalgamated product edge $e^{\pm 1}$. For $\tilde e_k$ which is a lift of $e^{\pm 1}$, we have
$$d_i(p_{k,i}^{\alpha},p_{k,i}^{\omega})\leq d_i(\bar p_{k,i}^{\alpha},\bar p_{k,i}^{\omega})$$ and both terms do not depend on which lift of $ e^{\pm 1}$ is taken.
Therefore $$d_i(1, h_i(g))\leq d_i(1,h _i\circ\sigma^{m_i}(g)). $$

Therefore when $h$ runs through the set of generators of $G_V$, $\{\hat \lambda _i\circ{\alpha _i}^{p_i}\circ\xi _i (h)\}$ converges to the elements that generate the same fundamental group as the limits of $\{\phi _i(h)\}$ because for sufficiently large $i$ if $\phi _i(g)\neq 1$ then $\hat \lambda _i\circ{\alpha _i}^{p_i}\circ\xi _i(g)\neq 1.$


We again have  $\phi_*^i\circ\alpha_i=\alpha\circ\phi_*^i$ and therefore $\hat \lambda _i{\alpha _i}^{p_i-1}\xi _i$  converge to $\alpha ^{-1} (L)$. This proves the claim. \qed

Then $\bar w_j=1$ is also satisfied in $L$ on the elements that are limits of  sequences 

$\{\hat \lambda _i{\alpha _i}^{p_i}\xi _i(a_p), 
\hat \lambda _i{\alpha _i}^{p_i}\xi _i(b_p)\}$.

This is a contradiction with the assumption that almost all $ h_i$ are  non-free-like, but almost all $\eta _i$ are free-like.

Similarly one considers the case when $L$ is an HNN-extension. 

 Suppose now $\mathbb{B}$ has several edges and all vertex groups are rigid. We can assume this by the previous cases (no abelian and QH subgroups). We begin with $\{h_i\}$ and go along the edges of $\mathbb{B}_i$.  If there is a splitting along some edge of $\mathbb{B}_i$ such that applying Dehn twists along this edge we can get  a sequence of free-like morphisms from $\{h_i\}$, then we will get a contradiction the same way as we got it for the amalgamated product  case. Otherwise $\{h_i\}$ is minimal with respect to Dehn twists along the edges of $\mathbb{B}_i$. But this contradicts to the assumption that $\{h_i\}$ converges to $L$.
\ \\ \\
We now tackle the case where $L$ is a non-trivial free product. Let $L_1*L_2*\ldots *L_k*\F$ be the Grushko decomposition of $L$. We can modify the sequence $(h_i)_{i<\omega}$ by mapping $\F$ to the identity. The modified sequence is not linearly bounded  because $(h_i)_{i<\omega}$ are not free-like. So we can assume that $\F$ is trivial. As a first step we approximate $L$ by a free product 
decomposition of finitely presented groups, $G'=L_1'*L_2'*\ldots*L'_k$ with the following properties: each factor admits a (homo)morphism into some conjugate of a distinct $L_i$ and the free product of them is a quotient of $G$ via an epimorphism $\gamma:G\twoheadrightarrow G'$. Since $G'$ is finitely presented, the sequence $(h_i)_{i<\omega}$ factors through $\gamma$ as $h_i=h_i'\circ\gamma$ and the restriction of $h_i'$ on each factor converges to a Grushko factor of $L$. We can now apply the shortening argument to each factor and obtain a similar contradiction as before, noting that a $w_j$ must be elliptic.  The theorem is proved.  
\end{proof}

\begin{corollary}\label{limit1} For any system of equations $V(\bar{x})=1$
there is a constant $K$ (that depends only on $V(\bar{x})=1$ and $d$) such that for any short non-free solution $\abs{h _i}\leq K i$. 
\end{corollary}
\begin{proof} If $G_V$ is not a limit group, then we know from  Theorem \ref{limit} that if there is no such a bound for shortest solutions $\{h_i\}$, then they
are free-like, therefore they factor through a limit quotient $\bar L$ of $G_V.$ If $G_V$ is a limit group, then $G_V=\bar L$. If the sequence 
$\{h_i\}$ converges to $\bar L$, then they can be shortened by an automorphism $\alpha$  of $\bar L$ and the shorter solutions $\{\alpha\circ h_i\}$ are also non-free. Since this cannot happen,  $\{h_i\}$ converges to a proper quotient of $\bar L$. Let $G_{V_1}$ be a proper finitely presented quotient of $\bar L$. We now apply Theorem \ref{limit} to the universal axiom 
$$\forall \bar x (V_1(\bar x)\rightarrow \vee \bar w_i=1),$$
where $\bar w_i=1$ are equations corresponding to limit quotients of $G_{V_1}$ and show that either
$\{h _i\}$ are linearly bounded (by $K_1i$) or 
free-like with respect to this sentence. If they are free-like, they factor through a proper limit quotient of $\bar L$. Since $\F$ is equationally Noetherian, every chain of proper limit quotients is finite, all shortest non-free
$h_i$ are linearly bounded.
 \end{proof}

\section{Diagrams for solutions of equations}

We fix the density parameter $d<1/16$. Let $\Gamma_\ell$ be a random group at density $d$ with relator length $\ell$. Let $V(\bar x)=1$ be a system consisting of the following equations $v_i(\bar x)=1$, for $i\leq m$. Suppose $\bar{b}_\ell$ in $\Gamma_\ell$ is a short non-free solution. The latter implies the existence of a family of $m$ van Kampen diagrams over $\Gamma_\ell$ with boundary labels corresponding to  
$v_i(\bar{b}_\ell)$. Note that some of these diagrams have cells because $\bar{b}_\ell$ is non-free. 

Our goal is to show that the probability of existence of such $DVK$ in $\Gamma_\ell$ approaches zero as $\ell \rightarrow\infty$. In the next subsection we give some direct consequences of the bound on short solutions. In particular, we get an absolute bound on the number of faces that can be thought of as a local to global result. 
In subsection \ref{ReducedBoundary} we will obtain a {\em decoration} in the defining relators of a random group. This will be, in turn, used to decrease the independent choices of letters in them. In the last subsection we bring everything together to prove the main result of this paper, namely a universal sentence is almost surely true in a random group (of density $d$) if and only if it is true in the free group.

\subsection{Total number of aDVK for short non-free solutions}\label{loc-gl}

We first observe that any system of equations, $V(\bar{x})=1$, is equivalent to a triangular system in new variables. The advantage of this transformation to a triangular system will be obvious in Lemma \ref{aDVK} and Proposition \ref{xtoy}.
Indeed, any equation $x_1x_2w(\bar x)=1$ is equivalent to the system of shorter equations $x_1x_2z^{-1}=1, zw(\bar x)=1$ where $z$ is a new variable.
Therefore we may assume that:

$$
V(\bar x) := \{x_{\sigma(j,1)}x_{\sigma(j,2)}x_{\sigma(j,3)},\ j=1,\ldots,q\}, \textrm{where} \  \sigma(j,k)\in\{1,\ldots,l\}.$$

 Note, that we can always assume that a variable $x_s$ appears at least twice in $V(\bar{x})=1$, otherwise we can just remove the equation the single variable  appears in. 
 
 In order to bound the number of faces we will use the following theorem (\cite[Theorem 13]{Olli}). 
 
\begin{theorem}
Let $\Gamma_{\ell}$ be a random group of density $d<1/2$ at length $\ell$. For every $\epsilon>0$, every reduced diagram $D$ over $\Gamma_{\ell}$ almost surely satisfies 
$$\vartheta D > f\ell (1-2d-\epsilon)$$
where $\vartheta D$ is the boundary length of $D$ and $f$ is the number of its faces. 
\end{theorem}

\begin{proposition}\label{FacesBound}
Let $\mathbb D$ be a family of reduced diagrams that correspond to a short non-free solution of $V(\bar{x})=1$ in $\Gamma_\ell$. Let $K$ be the constant from Corollary \ref{limit1} and $r$ a bound on the number of appearances of every variable $x_i$, $1\leq i\leq t$. 

Then the total number of faces is almost surely strictly less than $N=Kr/(1-2d-\epsilon)$ for every $\epsilon>0$.
\end{proposition}
\begin{proof} By  Corollary \ref{limit1}, the value of each variable $x_i$ in the solution is no longer that  $K\ell$.
Therefore $rK\ell \geq \vartheta \mathbb D$, where  $\vartheta \mathbb D$ is the total boundary length of diagrams in $\mathbb D$. Now by the previous theorem we get  $\vartheta \mathbb D > f\ell (1-2d-\epsilon)$, for every $\epsilon>0$. Combining the two we get $f<Kr/(1-2d-\epsilon)$ for every $\epsilon>0$. 
\end{proof}

\begin{lemma} \label{planar} The number of planar graphs (non necessarily connected) with degrees of vertices not less than $3$  and with $f$ faces 
is less than $2^{10f}$.
\end{lemma}
\begin{proof}  The Euler formula for a planar graph with $n$ connected components is $v-e+f=n+1$, where $v,e,f$ are the numbers of vertices, edges and faces. If degrees of vertices $\geq 3$, then $3v\leq 2e$. Therefore $\frac{2}{3} e-e+f\geq n+1$ and  $e\leq 3f-3(n+1)$, $v\leq 2f-2(n+1)$.

The estimate was given in \cite{BGHPS}  on the number $p(v)$ of planar  graphs with  $v$  vertices, $p(v)\leq 2^{5v}$ (see also \cite{SV}). Therefore the number of planar graphs with vertices of degree $\geq 3$ is  not greater than $2^{10f }.$ 
\end{proof}
From now we assume that our DVK (respectively aDVK) consist of triangular diagrams (respectively abstract triangular diagrams). So each diagram has a form as in figure \ref{TrianglewithVariable} and variables are assigned to the sides of triangles. 

\begin{figure}[ht!]
\centering
\includegraphics[width=1\textwidth]{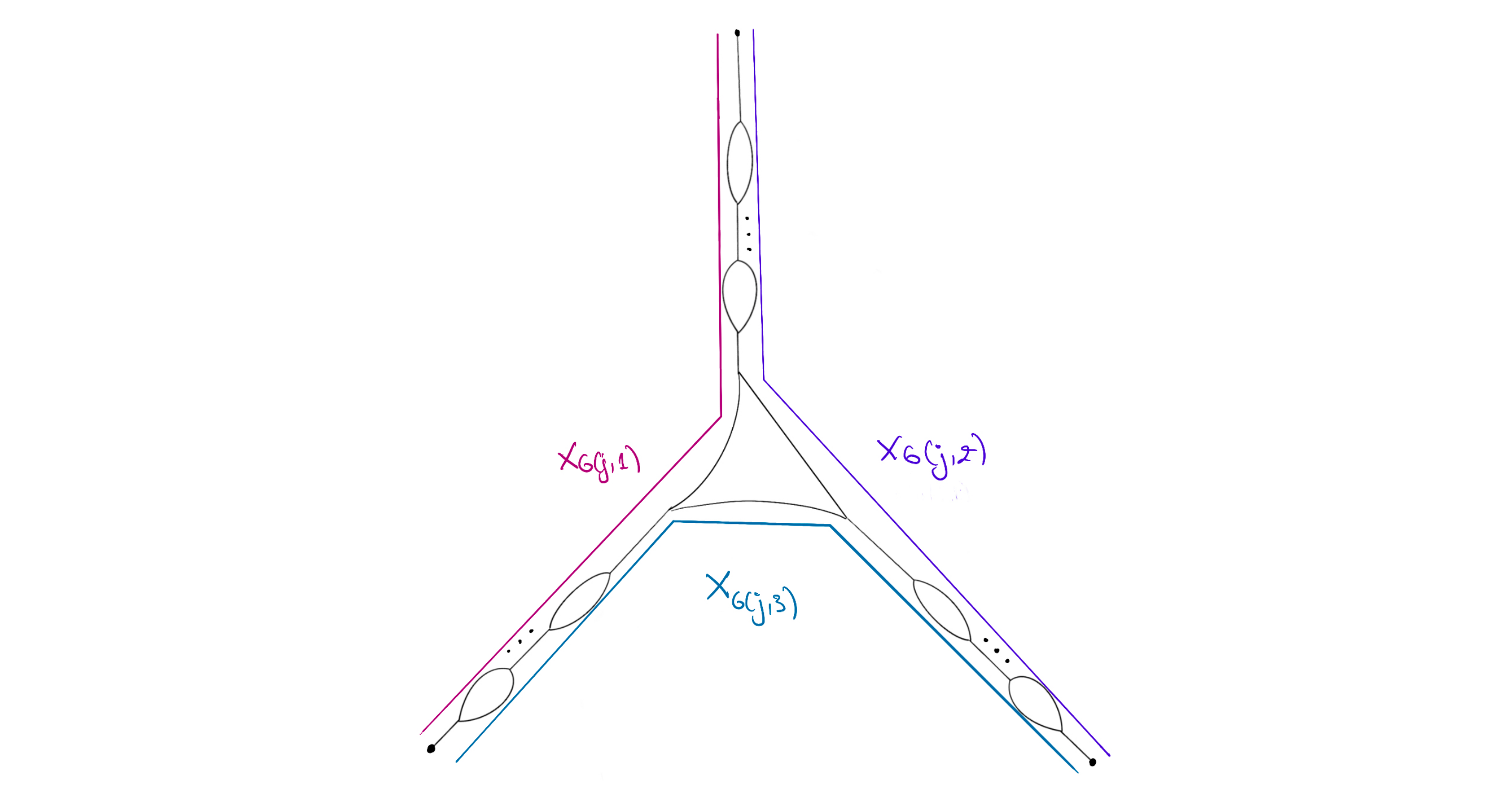}
\caption{A triangular diagram with assigned variables.}
\label{TrianglewithVariable}
\end{figure}

The number of ways to associate $n$ relators to $f$ faces is  $S(f,n)$ - the Stirling number (the number of ways to partition a set of $f$ elements into $n$ non-empty subsets).

\begin{lemma} \label{aDVK} 
Let $\ell$ be the length of each relator. The total number of families aDVK  with $f$ faces labelled by $n\leq f$ different relators is  bounded by $(K2^{13}\ell ^7)^{f+2q}S(f,n)$, where $S(f,n)$ is the Stirling number and $q$ is the number of triangles (which is equal to the number of triangle equations). 
\end{lemma}
\begin{proof}   
In order to transform a family of planar graphs into a family of abstract van Kampen diagrams we need to make certain choices. We explain below how these choices affect the total number of abstract van Kampen diagrams.

First, we multiply the estimate from Lemma \ref{planar} by $(\ell)^{3f}$, because each edge of the planar graph without degree 2 vertices  can  be given maximum  length $\ell$ to obtain an aDVK over a group with defining relator length $\ell$. 

Second, since after adding (at most) $\ell$ vertices for each edge we have more options for placing the bridges, we multiply by $(2\ell)^{2(f+2q)}$. Indeed, since in our case an aDVK is triangular, there are not more than $f+2q$ bridges. Moreover, adding (at most) $\ell$ vertices to each edge gives us $2\ell$ new vertices around the original vertex the bridge was attached to in each of the two non-filamentous components (it could be the case that a bridge is attached to only one non-filamentous component but calculating the possibilities as if each bridge is attached to two non-filamentous component only increases the estimate).  Hence, deciding between which of the $2\ell$ new vertices around the original vertex of the non-filamentous components we place the bridge to gives us $(\ell^2)^{f+2q}$ choices. 

Third, to account for orientation and starting point of each face we multiply by  $(2\ell)^f$.

Fourth, we  multiply by $(K\ell )^{f+2q}$  because the total length of  bridges (and, therefore, the length of each of them) is bounded by the length of the solution which is, in turn, bounded by $K\ell$. 

Finally, we multiply by the Stirling number to take into  account the choices that some faces are labelled by the same relator. 
\end{proof}

The important conclusion of the previous lemma is that the bound is polynomial in $\ell$.

\subsection{Non-filamentous aDVK  with decorated boundaries}\label{ReducedBoundary}

From now on all diagrams in aDVKs have triangular form (see figure \ref{TrianglewithVariable}), with $<N$ faces (where $N$ is the bound from Proposition \ref{FacesBound}) and relators  of length $\ell$. Recall that, by the definition of $aDVK$, this implies that we know the graphs, the initial point and the orientation for every face as well as which faces will be assigned the same defining relators. In addition, we know the number, position, and length of all bridges. Since the $aDVK$ is obtained from a (non-free) solution of $V(\bar x)=1$ we know which variable corresponds to which side of a triangle and we note in passing that, we may assume, every variable appears at least twice (and also that the $aDVK$ is non-trivial, i.e. it has at least one non-filamentous component). 

Our next goal is to ``decorate" the boundaries (and then the $2$-cells) of non-filamentous components of $aDVKs$. We will assign labels to subpaths of these boundaries in a way that subpaths given the same label can only be fulfilled by the same letters. This apparently imposes some restrictions in fulfilling an $aDVK$ - a fact that we will exploit in the next subsection.  


Suppose an $aDVK$ has $|\mathcal{D}|$ faces and $1\leq n\leq |\mathcal{D}|$ distinct relators. Then we say that an $n$-tuple $\bar w:=(w_1, \ldots, w_n)$ {\em fulfills the $aDVK$} if for each two faces $f_1, f_2$ bearing relators $i_1, i_2$, such that the $k_1$-th edge of $f_1$ is equal to the $k_2$-th edge of $f_2$, then the $k_1$-th letter of $w_{i_1}$ and the $k_2$-th letter of $w_{i_2}$ are inverse (when the orientations of $f_1$ and $f_2$ agree) or equal (when they disagree).  

\begin{proposition} \label{xtoy} 
Let $V(\bar x)=1$ be a system of equations and $\mathcal{AD}$ an $aDVK$ obtained from a (non-free) solution of $V(\bar x)=1$ in $\Gamma_\ell$. Then there exists a system of equations $\Sigma(\bar y,\bar c)$ over  $\F_n$ with parameters $\bar c$ such that:
\begin{enumerate}
\item each $c_i\in \bar c$ is assigned to either a side of a digon or a side of a simple triangle of some diagram in $\mathcal{AD}$ in a way that there is a correspondence between the elements of $\bar c$ and the set of all sides (of digons and triangles).
\item A solution of $V(\bar x)=1$ in $\Gamma_{\ell}$ that satisfies  $\mathcal{AD}$, i.e. $\mathcal{AD}$ is fulfillable and the (reduced)  words on the boundaries of the triangles coincide with the values of the corresponding variables, gives a solution of $\Sigma(\bar y,\bar c)$ in $\F_n$ (without cancellations), such that values of the parameters $\bar c$ satisfy the corresponding boundaries of non-filamentous components of the $\mathcal{AD}$ (i.e. digons and simple triangles).

\item  A solution of $\Sigma(\bar y,\bar c)$ in $\mathbb F_n$ such that values of the parameters $\bar c$ satisfy 
the corresponding relations in $\Gamma_\ell$ (for example if $c_1,c_2,c_3$ are such so that $c_1c_2c_3$ corresponds to the boundary of a simple triangle in some diagram in $\mathcal{AD}$, they satisfy $c_1c_2c_3=1$ in $\Gamma_\ell$), gives a solution of $V(\bar x)=1$ in $\Gamma_\ell$.
\end{enumerate}
\end{proposition}
\begin{proof}
We will show how to obtain such a parametric system. Suppose, for definiteness, that each diagram in $\mathcal{AD}$ has at most three digons (see figure \ref{reps1}). The general case is similar, but hard to handle notation wise.

Recall that $ \Gamma_\ell=\langle e_1,\ldots ,e_n|\mathcal R\rangle$ is  a quotient of $\mathbb F_n=\langle e_1,\ldots ,e_n\rangle.$ Every solution of $V(\bar x)=1$ can be considered as a homomorphism $\psi: \mathbb{F}(\bar x)\rightarrow \Gamma_\ell$, where $\F(\bar{x})$ is the free group in $\{\bar{x}\}$. We denote by $g_{\sigma(j,k)}$ a geodesic word equivalent to $\psi(x_{\sigma(j,k)})$. Let $\theta (g_{\sigma(j,k)})$
be the same word in $\mathbb F_n$. We call it the representative of a solution.
Then there exist
$h_{k}^{(j)}, \bar h_{k}^{(j)}, c_{k}^{(j)}, \bar c_{k}^{(j)},\hat c_{k}^{(j)}\in \mathbb{F}_n$, for $j\leq q$ and $k\leq 3$,  with the following properties:
\begin{itemize}

\item $c_{1}^{(j)}c_{2}^{(j)}c_{3}^{(j)}  =  1$ in $\Gamma_\ell$, \label{RepsCond1} and there exists a van Kampen diagram over $\Gamma_\ell$ with  reduced boundary $c_{1}^{(j)}c_{2}^{(j)}c_{3}^{(j)}$ or $c_{1}^{(j)}=c_{2}^{(j)}=c_{3}^{(j)}=1$. In addition $\bar c_1^{(j)}\hat c_3^{(j)}=1, \bar c_2^{(j)}\hat c_1^{(j)}=1, \bar c_3^{(j)}\hat c_2^{(j)}=1$ in $\Gamma_{\ell}$. 
\item \label{RepsCond2} The representatives satisfy the following equations in $\mathbb{F}_n$ (see figure \ref{reps1}):
\begin{eqnarray}
\theta (g_{\sigma(j,1)}) & = & h_{1}^{(j)} \bar c_{1}^{(j)}\bar h_{1}^{(j)}c_1^{(j)} \left(\bar h_{2}^{(j)}\right)^{-1}\hat c_1^{(j)}\left(h_{2}^{(j)}\right)^{-1} \label{CanonReps1}\\
\theta (g_{\sigma(j,2)}) & = & h_{2}^{(j)} \bar c_{2}^{(j)}\bar h_{2}^{(j)}c_2^{(j)} \left(\bar h_{3}^{(j)}\right)^{-1}\hat c_2^{(j)}\left(h_{3}^{(j)}\right)^{-1}\\
\theta (g_{\sigma(j,3)}) & = &  h_{3}^{(j)} \bar c_{3}^{(j)}\bar h_{3}^{(j)}c_3^{(j)} \left(\bar h_{1}^{(j)}\right)^{-1}\hat c_3^{(j)}\left(h_{1}^{(j)}\right)^{-1}.\label{CanonReps3}
\end{eqnarray}
\end{itemize}

\begin{figure}[ht!]
\centering
\includegraphics[width=1\textwidth]{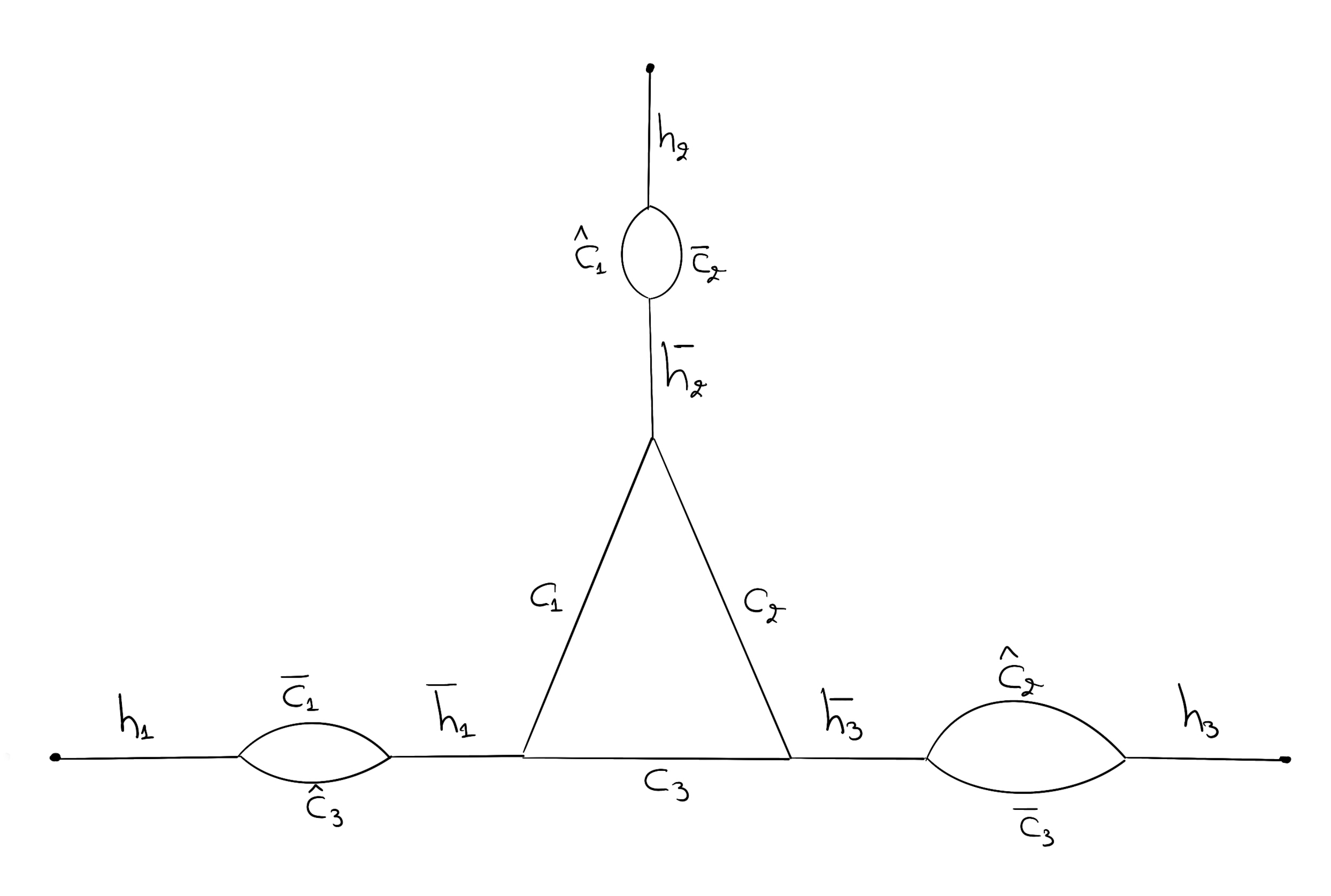}
\caption{A solution split in pieces.}
\label{reps1}
\end{figure}

In particular, when $\sigma(j,k)=\sigma(j',k')$ (which corresponds to two occurrences in $V(\bar{x})=1$ of the variable $x_{\sigma(j,k)}$) we have in $\mathbb{F}_n$:
\begin{equation}
h_{k}^{(j)} \bar c_{k}^{(j)}\bar h_{k}^{(j)}c_k^{(j)} \left(\bar h_{k+1}^{(j)}\right)^{-1}\hat c_k^{(j)}\left(h_{k+1}^{(j)}\right)^{-1}=h_{k'}^{(j')} \bar c_{k'}^{(j')}\bar h_{k'}^{(j')}c_k^{(j)} \left(\bar h_{k'+1}^{(j')}\right)^{-1}\hat c_{k'}^{(j')}\left(h_{k'+1}^{(j')}\right)^{-1}.\label{Hequality}
\end{equation} without cancellation.

We next construct a parametric system of equations, $\Sigma(\bar y,\bar c)$, where $\bar{c}$ are the parameters, such that a solution of $V(\bar{x})=1$ in $\Gamma_{\ell}$ induces a solution of $\Sigma(\bar y,\bar c)$ in $\mathbb{F}_n$, and vice versa a solution of $\Sigma(\bar y,\bar c)$ in $\mathbb{F}_n$ with some extra conditions on the parameters $\bar{c}$ induces a solution of $V(\bar{x})=1$ in $\Gamma_\ell$.  

The system of equations, $\Sigma(\bar y,\bar c)$, with parameters $\bar c=\{c_1,\ldots ,c_{3q}\}$  in $\mathbb{F}_n$ is constructed as follows:

For $9q$ parameters  $c_{1}^{(j)},c_{2}^{(j)},c_{3}^{(j)},\bar c_{1}^{(j)},\bar c_{2}^{(j)},\bar c_{3}^{(j)},\hat c_{1}^{(j)},\hat c_{2}^{(j)},\hat c_{3}^{(j)}$ in $\mathbb{F}_n$, for  $j\leq q$, 
we build a system $\Sigma(\bar y, \bar c)=1$
consisting of the equations:
\begin{equation}
y_{k}^{(j)} \bar c_{k}^{(j)}\bar y_{k}^{(j)}c_{k}^{(j)} \left(\bar y_{k+1}^{(j)}\right)^{-1}\hat c_k^{(j)}\left(y_{k+1}^{(j)}\right)^{-1}=y_{k'}^{(j')} \bar c_{k'}^{(j')}\bar y_{k'}^{(j')}c_k^{(j)} \left(\bar y_{k'+1}^{(j')}\right)^{-1}\hat c_{k'}^{(j')}\left(y_{k'+1}^{(j')}\right)^{-1} \label{Eqn:SC1}
\end{equation}
where an equation of type   (\ref{Eqn:SC1}) is included whenever  $\sigma(j,k)=\sigma(j',k')$. 

We first show that a solution of $V(\bar{x})=1$ in $\Gamma_\ell$ induces a solution of $\Sigma(\bar{y},\bar{c})=1$ in $\mathbb{F}_n$. 
%
Let $\psi:\mathbb{F}(\bar x)\rightarrow \Gamma_\ell$ be a solution of $V(\bar x)=1$ in $\Gamma_\ell$, with $\psi(x_{\sigma(j,k)})=g_{\sigma(j,k)}$.  We can extract elements $h_{k}^{(j)}, \bar h_{k}^{(j)}, c_{k}^{(j)}, \bar c_{k}^{(j)},\hat c_{k}^{(j)}$ in $\mathbb{F}_n$, for $j\leq q$ and $k\leq 3$ so that if we plug them in the corresponding $\bar{y}$ and  $\bar{c}$ they solve $\Sigma(\bar y, \bar c)=1$ in $\mathbb{F}_n$.  
These elements are just pieces of words $\theta (g_{\sigma(j,k)})$ and we even know their lengths.
 
Conversely, for any solution  
 of $\Sigma(\bar y, \bar c)=1$  in $\mathbb F_n$ such that the values of the parameters satisfy $c_{1}^{(j)}c_{2}^{(j)}c_{3}^{(j)}=1$, $\bar c_1^{(j)}\hat c_3^{(j)}=1, \bar c_2^{(j)}\hat c_1^{(j)}=1, \bar c_3^{(j)}\hat c_2^{(j)}=1$ in $\Gamma_\ell$ one can do the following: for $1\leq s\leq l$, we pick any $j,k$ such that $\sigma(j,k)=s$ and we define:  
$$
\rho (x_s) = y_{k}^{(j)} \bar c_{k}^{(j)}\bar y_{k}^{(j)}c_{k}^{(j)} \left(\bar y_{k+1}^{(j)}\right)^{-1}\hat c_k^{(j)}\left(y_{k+1}^{(j)}\right)^{-1}. $$

Now giving to each $x_s$, for $1\leq s\leq l$, the value obtained by plugging in the values of $\bar{y}$ and $\bar{c}$ in $\rho(x_s)$, we get values for the $x_s$'s that satisfy $V(\bar{x})=1$ in $\Gamma_\ell$. Note, that we can always assume that a variable $x_s$ appears at least twice in $V(\bar{x})=1$, otherwise we can just remove the equation the single variable  appears in. Hence, we will necessarily have all corresponding variables represented in $\Sigma(\bar{y},\bar{c})$. 
\end{proof}

We next define the notion of a {\em decoration} on the boundary of an $aDVK$. 

\begin{definition}[van Kampen diagram decoration]
A decoration of an $aDVK$ by a set $P:=\{{p_1}^{\pm 1},\ldots ,{p_k}^{\pm 1}\}$ is an assignment from $P$ to the set of $1$-cells  on the boundary of each $adVK$ (counted with multiplicity as we transverse the boundary) such that: 
\begin{itemize}
    \item each $1$-cell is assigned an element from $P$.
    \item each element $p_i$, for $i\leq k$, is assigned to at least two $1$-cells (as $p_i$ or $p_i^{-1}$).
\end{itemize}

We call the assignment a pre-decoration if each $p_i$ is allowed to appear also once. 
\end{definition}

We may also decorate ``unknown words" of fixed length. Unknown words can be thought of as place holders for letters of an alphabet. The reader should consider $2$-cells in $adVKs$ as such place holders.  

\begin{definition}[Relators decoration]
Let $R$ be a set of unknown words of length $\ell$. A decoration of $R$ by a set $Q:=\{{q_1}^{\pm 1},\ldots,{q_k}^{\pm 1}\}$ is an assignment from $Q$ to the set of letters of each word in $R$ such that: 
\begin{itemize}
    \item each such letter is assigned an element from $Q$.
    \item each element $q_i$, for $i\leq k$, is assigned to at least two letters (as $q_i$ or $q_i^{-1}$).
\end{itemize}
We call the assignment a pre-decoration if each $q_i$ is allowed to appear also once.
\end{definition}

\begin{proposition} \label{sol} 
Let $V(\bar x)=1$ be a system of equations and $\mathcal{AD}$ an $aDVK$ obtained from a (non-free) solution of $V(\bar x)=1$ in $\Gamma_\ell$. Let $\Sigma (\bar y,\bar c)$ be the corresponding system of equations over $\mathbb{F}_n$ (see Proposition \ref{xtoy}). Then $\Sigma (\bar y,\bar c)$ induces a natural pre-decoration on $\mathcal{AD}$. 
\end{proposition}
\begin{proof}

Each variable in $\bar{x}$ corresponds to a side of a geodesic triangle in some triangular diagram in $\mathcal{AD}$. When the same variable covers different sides we get an equation in $\Sigma(\bar y, \bar c)$. It is visually helpful, but not essential, to place each variable $x_1, x_2, \ldots, x_l$ in order, one after the other. Then, since we know the length of each, we split each variable in as many unit intervals as its length. Following the steps of the proof of Proposition \ref{xtoy} we may place above each variable intervals labeled by $y's$ and $c's$ where we additionally record their orientation (see figure \ref{generalized}).   
\begin{figure}[ht!]
\centering
\includegraphics[width=1.3\textwidth]{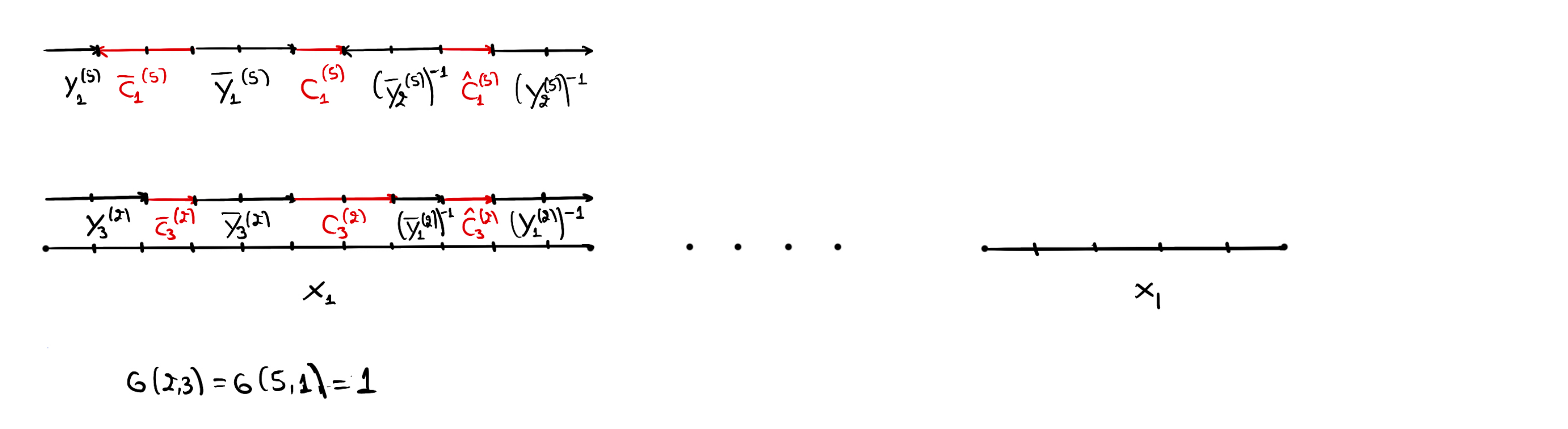}
\caption{System $\Sigma (\bar y,\bar c)=1$, where the parameters $\bar c$ are in red color.}
\label{generalized}
\end{figure}

Notice that each  $y_i^{(j)}$ appears exactly twice. Indeed, by construction, it appears in only one diagram in $\mathcal{AD}$ (the $j$-th diagram) and  either in two variables in $\bar x$ or in two places of the same variable. On the other hand, similar considerations easily show that every $c_i^{(k)}$ appears exactly once (we suppress notation-wise $\bar c$ and $\hat c$ or more generally the use of a double sup-index for $c$). 

From the data at hand, we now construct a finite graph $\mathbb G$ as follows: the vertices of $\mathbb G$ correspond to the unit intervals we split the variables $\bar x$ in, there is an edge between two vertices, say $x_i^{r_1}$, $x_j^{r_2}$ (the $r_1$-th unit interval of $x_i$ and the $r_2$-th unit interval of $x_j$), of $\mathbb G$ if the same unit interval part of some $y_i^{(j)}$ (that, recall, appears exactly twice) lies above both $x_i^{r_1}$ and $x_j^{r_2}$. 

Suppose $\mathbb{G}$ has $k$ connected components. Then $\mathcal{AD}$ is covered by $P:=\{p_1^{\pm},\ldots, p_k^{\pm}\}$ in the following way. 
Each connected component is assigned an element $p_i$. Consequently, every unit interval of $y_i^{(j)}$ or $c_i^{(k)}$ that lies above the vertices of this component is assigned an element $p_i$ or $p_i^{-1}$ according to its orientation. We note in passing that the assignment is consistent, i.e. we cannot assign $p_i$ and $p_i^{-1}$ to the same unit interval of some $y_i^{(j)}$, as then we have that $\Sigma (\bar y,\bar c)$ is inconsistent.



\end{proof}

\begin{remark}
Given a system of $q$ triangular equations $V(\bar x)=1$ and an $aDVK$ with $q$ many triangular diagrams not necessarily obtained from a solution of $V(\bar x)=1$ in $\Gamma_\ell$ we may still try to decorate it in the same fashion as in the proof of the previous proposition. In this case, though, we might run into an inconsistency of orientations in the assignment of the $p_i$'s. This implies that there is no solution of $V(\bar x)=1$ that satisfies the given $aDVK$. If, on the other hand, we can decorate it, then we can actually find a solution of $V(\bar x)=1$ in $\Gamma_\ell$ that fulfills the given $aDVK$.   
\end{remark}

To clarify the construction in the proof of the above proposition we give an example. 

\begin{example}
We consider the following system of (triangular) equations:
\[
\left\{
\begin{array}{l}
x_1x_1x_4=1 \\
x_2x_2x_5=1\\
x_3x_3x_6=1\\
x_4x_5x_6=1
\end{array}
\right.
\]
and the following family of abstract van Kampen diagrams (see figure \ref{aDVKDecor}). Note that we do not assume that they are obtained from a solution of the above system. 
\begin{figure}[ht!]
\centering
\includegraphics[width=1.\textwidth]{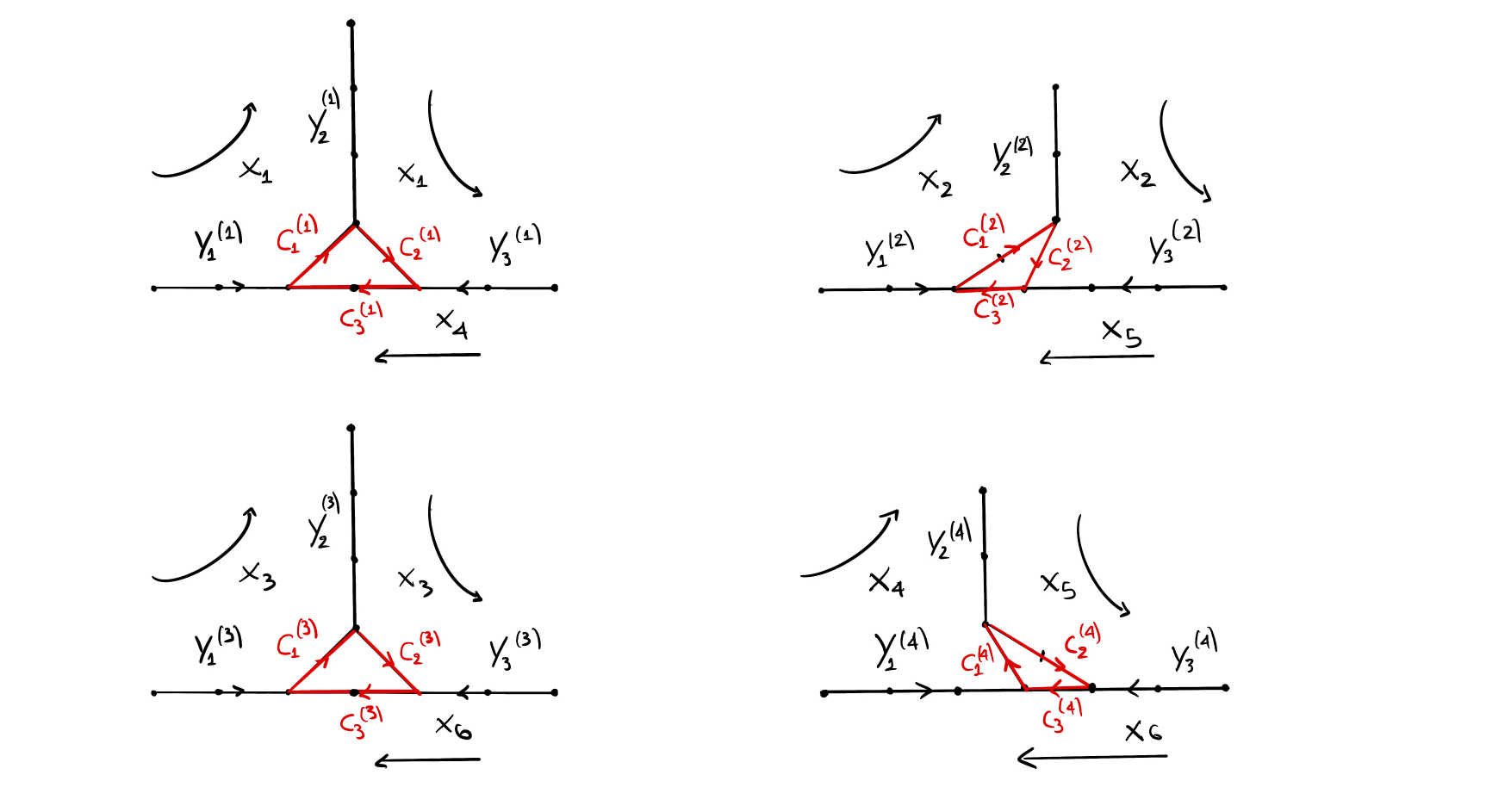}
\caption{A family of abstract dVKs with fixed lengths.}
\label{aDVKDecor}
\end{figure}

In the next figure \ref{Decor} we show how we can obtain the graph that will help us decorate the $aDVK$.
\end{example}

\begin{figure}[ht!]
\centering
\includegraphics[width=1.\textwidth]{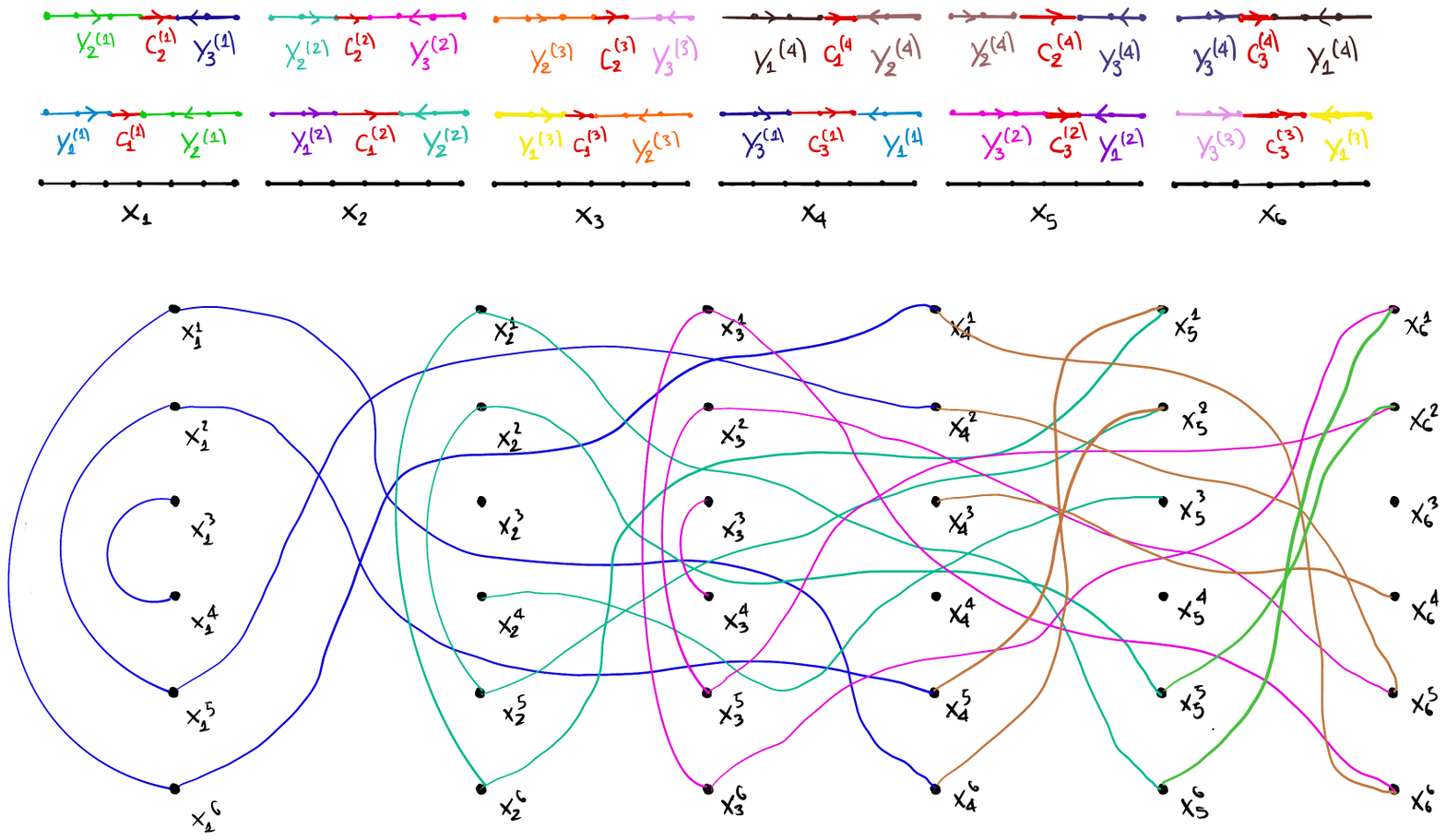}
\caption{The graph obtained from the system of equations and the given $aDVK$.}
\label{Decor}
\end{figure}

\newpage

For the discussion that follows we fix a system of equations $V(\bar x)=1$, a family $aDVK$ constructed from a (non-free) solution of $V(\bar x)=1$ and a pre-decoration $P$ of $aDVK$ obtained as in the proposition above.  
We argue that such a pre-decoration induces a decoration on a subfamily of $aDVK$ after removing some triangular diagrams from it. Indeed, suppose, without loss of generality, that the element $p_k$ appears only once on the union of the boundaries in the family $aDVK$. 
Then we can remove the diagram in which $p_k$ appears. After removing this diagram we might have another element, say $p_{k-1}$, appearing only once. We continue in the same fashion removing diagrams for which an element from the set of decorations $P$ appears only once. In finitely many steps (as $P$ is finite) we would either have no diagrams, or a family of trivial diagrams, i.e. diagrams with no $2$-cells, or a family with at least one non-trivial diagram. We claim that only the latter can happen. Indeed, in the first two cases, one can easily see, that the solution of $V(\bar x)=1$ for which we obtained the $aDVK$ is a free solution. 

We next observe that a decoration or a pre-decoration on an $aDVK$, induces a pre-decoration on the non-filamentous components of $aDVK$. Assume, that a decoration has been obtained for a subfamily of an $aDVK$ as in the previous paragraph. We show how, up to removing some non-filamentous components, we can obtain a decoration of the non-filamentous components of this subfamily. The idea is identical to the idea of the previous paragraph. We inductively remove non-filamentous components in which an element $p_i$ appears only once (in the whole set of non-filamentous components). Eventually, we must have a non-empty set of non-filamentous components which is decorated, otherwise, again, the solution we started with is a free solution.    

The above discussion implies the following.

\begin{corollary}  \label{sol2} 
Let $V(\bar x)=1$ be a system of equations and $\mathcal{AD}$ an $aDVK$ obtained from a (non-free) solution of $V(\bar x)=1$ in $\Gamma_\ell$. Suppose $\mathcal{AD}$ is naturally pre-decorated as in Proposition \ref{sol}. Then there is a non-empty family of non-filamentous components of $\mathcal{AD}$ which inherits a decoration from $\mathcal{AD}$. 
\end{corollary}

\begin{figure}[ht!]
\centering
\includegraphics[width=0.8\textwidth]{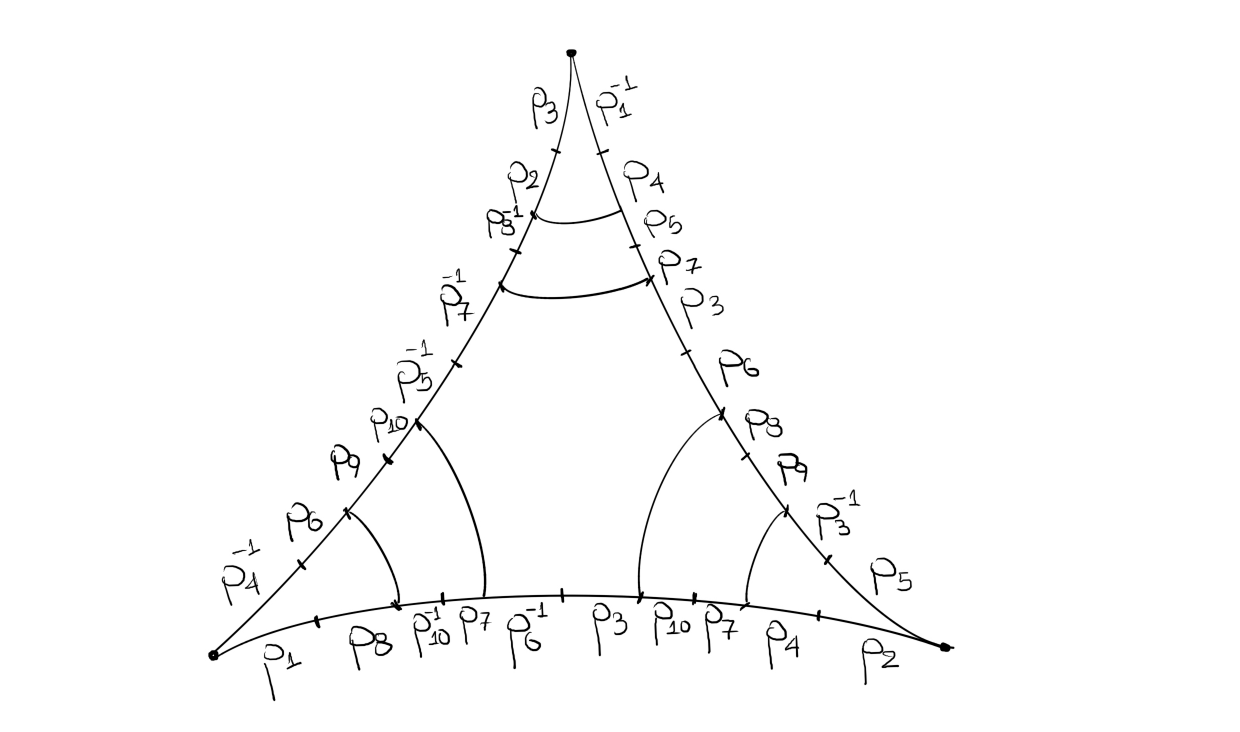}
\caption{A decorated non-filamentous component.}
\label{decoration3}
\end{figure}

We will next show how from the decoration of the subfamily of non-filamentous components we can obtain a pre-decoration on the family of 2-cells (seen as unknown words) that appear in them.

\begin{lemma}\label{red}
Suppose $\mathcal{NFAD}$ is a family of non-filamentous abstract van Kampen diagrams which is decorated by $P$. Then we may obtain a pre-decoration $Q$ on the family of $2$-cells of $\mathcal{NFAD}$, which is inherited from $P$ and it is either a decoration or, if not, there is an element $q_i\in Q$ that appears only in the boundary of the non-filamentous components and in the same position of several $2$-cells that are assigned the same relator.       
\end{lemma}

\begin{proof} 
Suppose $\mathcal{NFAD}$ contains $|\mathcal{D}|$ faces and $1\leq n\leq |\mathcal{D}|$ distinct relators assigned to the faces. The pre-decoration $Q$ on the family of $n$  relators  is obtained in a similar way the pre-decoration was obtained in the proof of Proposition \ref{sol}. 
We position all distinct relators $r_1,\ldots,r_n$ in order. Each relator corresponds to an interval of length $\ell$ and we split each such interval in $\ell$ unit subintervals. 

Above each relator we place the $2$-cells that correspond to it (with their orientation) again split in unit intervals and according to the fixed position that comes with the data of $\mathcal{NFAD}$. We now construct a graph $\mathbb{G}$ in analogy to the construction in Proposition \ref{sol}. The vertices of $\mathbb{G}$ are the unit subintervals we split the $n$ relators in. Moreover, two vertices are connected by an edge if above the unit subintervals that correspond to the vertices there are edges that coincide in the diagram. Suppose the graph $\mathbb{G}$ has $m$ connected components. Then each component will be assigned an element from $Q:=\{q_1^{\pm}, \ldots, q_m^{\pm}\}$. Some components will contain vertices, that have unit intervals above them that are already decorated by $P$, we identify the corresponding $q$'s with the existing $p$'s in a way that $P\subseteq Q$.   



This gives a pre-decoration on the family of $n$  relators that fails to be a decoration if there is a $q_i$ that appears only once. This $q_i$ cannot label a connected component that has an internal part of some $2$-cell above some unit subinterval that corresponds to a vertex of the component. The only possibility is that $q_i$ labels a connected component that has a single vertex and unit intervals above it correspond to the same edge of $2$-cells that only appears in the boundary of some diagrams in $\mathcal{NFAD}$.   


\end{proof}

\subsection{Probability to fulfill a decorated  aDVK}\label{DDec}

\begin{proposition} \label{A} Suppose we have a non-filamentous  aDVK with $f$ faces that should be filled in by $n\leq f$ relators. Let $P$ be a decoration on the set of $n$ relators.  Then  the number of letters that can be chosen arbitrarily in the relators  to fulfill the aDVK  is bounded from above by $\frac{1}{2} n\ell$.     

The probability that this aDVK  can be labelled by $n$ relators of the group with $(2m-1)^{d\ell}$ random defining relators of length $\ell$ such that the labeling satisfies $P$
is not more than $$(2m-1)^{-n\ell(1/2-d)}.$$
\end{proposition}

\begin{proof} The sum of the lengths of the relators that can label faces in the  aDVK is $n\ell$, one interval of length $\ell$ for each of the $n$ relators. 
The existence of the decoration on  the set of $n$ relators implies that we can choose not more letters than the half of the total length of the relators.  This gives us not more than $\frac{1}{2} n\ell$ degrees of freedom. 

The total number of possibilities 
to choose $n$ different words of length $\ell$   is $(2m-1)^{n\ell}$.  We can only select not more than
$(2m-1)^{n\ell /2}$ letters arbitrarily in the relators that fulfill the aDVK . Therefore the probability  that $n$ given words of length $\ell$ can label such aDVK  is less than
$$\frac{(2m-1)^{n\ell/2}}{(2m-1)^{n\ell}}=(2m-1)^{-n\ell/2}.$$

 This is the probability that $n$ given relators of a random presentation
fulfill the diagram. Now there are $(2m-1)^{d\ell}$ relators in a random presentation. Since the probability of the union of events is not greater than the sum of their probabilities,  the probability that we can find $n$ relators fulfilling the aDVK 
is at most 
$$((2m-1)^{d\ell})^n(2m-1)^{-n\ell/2}=(2m-1)^{- \ell n (1/2-d)}.$$ This proves the proposition.
\end{proof}

It is possible that the  decoration on the boundary of a non-filamentous aDVK  that can be labelled by $n$ relators does not give a decoration, but only a pre-decoration on the set of $n$ relators. In this case we have the following result.

\begin{corollary}\label{A1} Suppose we have a non-filamentous  decorated aDVK with $f$ faces that should be filled in by $n\leq f$ relations and the decoration on the boundary of aDVK was imposed by  a non-free solution of $V(\bar x)=1.$ Then the probability that this aDVK  can be labelled by $n$ relations of the group with $(2m-1)^{d\ell}$ random relations of length $\ell$ such that the labeling satisfies $P$
is not more than $$(2m-1)^{-\ell(1/2-d)}.$$
\end{corollary}
\begin{proof} In the proposition we computed the probability in the case there exists a decoration on the set of $n$ relators. If we only have a pre-decoration, using  Lemma \ref{red} we can remove the face labelled by the relator that has a variable that occurs once,  everywhere. The boundary of the new diagram may become non reduced, but we can make cancellations and obtain a decoration or pre-decoration on the boundary of the remaining diagram with $n-1$ relations.  The problem is eventually reduced to the  decorated aDVK with fewer than $n$ relators.
But at the end we should necessarily have at least one relator left, otherwise the solution has a pre-image in a free group that solves the equation.
\end{proof}

\begin{figure}[ht!]
\centering
\includegraphics[width=.8\textwidth]{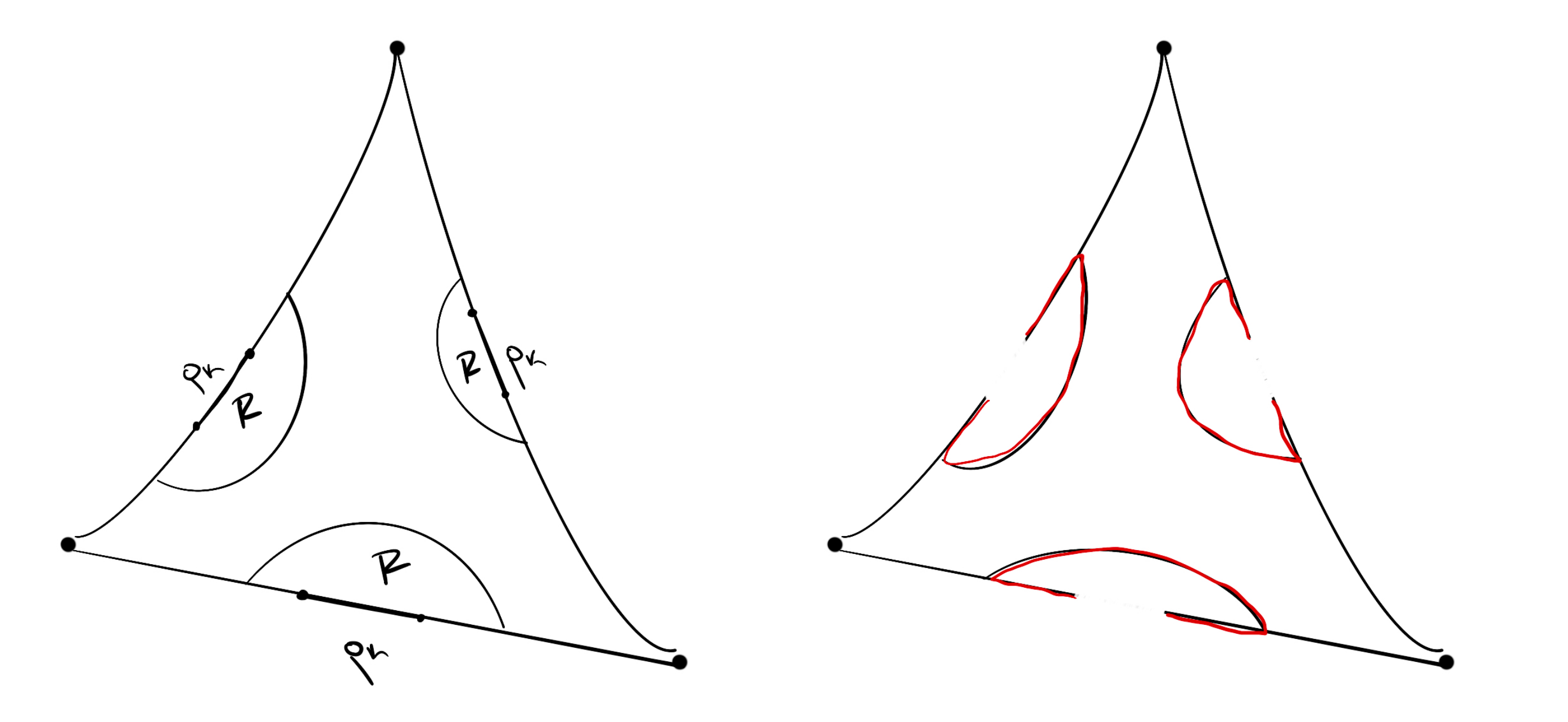}
\caption{Decoration with one face less.}
\label{decorP}
\end{figure}

\begin{proposition} 
Let $d<1/16$. The probability that for a random group  $\Gamma_\ell$ of density $d$ at length $\ell$, there exists a non-trivial DVK  with decoration $P$ corresponding to a solution of $\Sigma (\bar y, \bar c)$  that comes from a short non-free solution of  $V(\bar x)=1$ approaches $0$ as $\ell\rightarrow\infty$.
\end{proposition}
\begin{proof}
Indeed, the probability defined in the statement of the proposition is bounded by the probability in Corollary \ref{A1} multiplied by the sum from 1 to $N$ of the number of aDVK with $f$ faces obtained in Lemma \ref{aDVK}.
Thus, we have the bound:

$$\frac{\Sigma _{f =1}^{N}\Sigma _{n=1}^f (K2^{13}\ell ^7)^{f+2q}S(f,n)}{(2m-1)^{-\ell(1/2-d)}}$$

Finally, it is easy to see that the sum approaches zero when $\ell\rightarrow\infty$, as the numerator is a polynomial in $\ell$ while the denominator is exponential in $\ell$.
\end{proof} 

The above proposition implies the next result. 

\begin{proposition} \label{univ} 
Let $d<1/16$. Let $V(\bar{x})=1$ be a system of equations. Suppose $\Gamma_\ell$ is a random group of density $d$ at length $\ell$ and $\pi:\F_n\rightarrow\Gamma_\ell$ the natural quotient map. 

Then, every solution $\bar b_\ell$ of $V(\bar x)=1$ in $\Gamma_\ell$ is the image of a solution $\bar c_\ell$ of $V(\bar x)=1$ in $\F_n$, under the canonical quotient maps, i.e.  $\pi(\bar{c}_\ell)=\bar{b}_\ell$, with probability approaching $1$ as $\ell$ goes to infinity. 
\end{proposition}

Following the discussion in Section \ref{universal}, our main result follows.

\begin{theorem}\label{univ2}
Let $d<1/16$. Let $\sigma$ be a universal sentence in the language of groups. Then $\sigma$ is almost surely true in a random group (of density $d$) if and only if it is true in a nonabelian free group.
\end{theorem}

Finally, we observe that our main result generalizes to universal sentences with constants from a fixed nonabelian free group  $\mathbb{F}_n:=\langle e_1, \ldots, e_n\rangle$. For this we formally work in an expansion of the language of groups $\mathcal{L}$ where we add $n$ constant symbols, $c_1, \ldots, c_n$ that will be interpreted as $e_1, \ldots, e_n$ in $\mathbb{F}_n$ and as the images of those generators under the canonical maps in a random group $\Gamma_\ell$. We first prove a generalization of Proposition \ref{univ}.

\begin{proposition} \label{univ3} 
Let $d<1/16$. Let $V(\bar{x} ,e_1,\ldots ,e_n)=1$ be system of equations over $\mathbb{F}_n$. Suppose $\Gamma_\ell$ is a random group of density $d$ at length $\ell$ and $\pi:\F_n\rightarrow\Gamma_\ell$ the natural quotient map. 

Then, every solution $\bar b_\ell$ of $V(\bar x)=1$ in $\Gamma_\ell$ is the image of a solution $\bar c_\ell$ of $V(\bar x)=1$ in $\F_n$, under the canonical quotient maps, i.e.  $\pi(\bar{c}_\ell)=\bar{b}_\ell$, with probability approaching $1$ as $\ell$ goes to infinity. 
\end{proposition} 
\begin{proof}
We observe that if a system of equations has constants, then we can consider it as a system without constants, substituting constants with variables, with the extra restriction that these variables will be given fixed values. Thus, in this case, the probability to satisfy an $aDVK$ is less than the probability of satisfying it without restrictions, and the latter has been shown that goes to $0$.
\end{proof}


\begin{theorem}
Let $d<1/16$. Let $\sigma$ be a universal sentence in the language $\mathcal{L}\cup \{c_1,\ldots ,c_n\}$. Then $\sigma$ is almost surely true in a random group (of density $d$) with constants interpreted as $\pi (e_1),\ldots ,\pi (e_n)$ if and only if it is true in $\mathbb{F}_n$.
\end{theorem}
\begin{proof}
All arguments in Section \ref{universal} go through word for word in the presence of constants. Therefore, Proposition \ref{univ3} implies that if $\sigma$ is true in $\mathbb{F}_n$, then it is almost surely true in a random group (of density $d$). 

For the other direction, we need to show that if an existential sentence with constants
$$\exists \bar x (V(\bar e, \bar x)=1\wedge w_1(\bar e, \bar x)\not =1\wedge\ldots\wedge w_k(\bar e, \bar x)\not =1)$$ is true in $\mathbb F_n$, then it is almost surely true in a random group. Suppose $V(\bar e, \bar b)=1\wedge w_1(\bar e,\bar b)\not =1\wedge\ldots\wedge w_k(\bar e, \bar b)\not =1$ in $\mathbb F_n$, and $\bar b_R$ is the image of $\bar b$ in a random group. Then $V(\bar e, \bar b_R)=1$ and, by Proposition \ref{univ3}, the probability that $w_1(\bar e, \bar b_R) =1\vee\ldots\vee w_k(\bar e, \bar b_R)=1$ goes to $0$ as $\ell\rightarrow\infty$.
\end{proof}
 
We conclude our paper with a conjecture generalizing J. Knight's original conjecture. 

\begin{conjecture}
Let $d<1/16$. Suppose $\sigma$ is a first-order sentence in the language of groups. Then $\sigma$ is almost surely true in a random group of density $d$ if and only if it is true in a nonabelian free group.   
\end{conjecture}

\end{document}